\documentclass[a4paper, reqno, 11pt]{amsart}
\usepackage[all]{xy}
\usepackage{amsmath, amsfonts, amsthm, amssymb}
\usepackage{graphicx, wrapfig} 

\numberwithin{equation}{section}

\newcommand{\Q}{\mathbb{Q}}
\newcommand{\F}{\mathbb{F}}
\newcommand{\R}{\mathbb{R}}
\newcommand{\C}{\mathbb{C}}
\newcommand{\N}{\mathbb{N}}
\newcommand{\Z}{\mathbb{Z}}

\newtheorem{thm}{Theorem}[section]
\newtheorem{lem}{Lemma}[section]

\newtheorem*{heuristic}{Heuristic}
\newtheorem{con}{Conjecture}[section]

\renewcommand{\mod}[1]{\hspace{-2.9mm}\pmod{#1}}
\newcommand{\x}{{\bf x}}
\newcommand{\y}{{\bf y}}
\newcommand{\z}{{\bf z}}
\newcommand{\rom}{\mathrm}
\newcommand{\bfP}{\mathbb{P}}
\newcommand{\A}{\mathbb{A}}

\newcommand{\ma}{\mathbf}
\newcommand{\ben}{\begin{enumerate}}
\newcommand{\een}{\end{enumerate}}
\newcommand{\eit}{\begin{itemize}}

\newcommand{\ve}{\varepsilon}

\newcommand{\mcal}{\mathcal}
\newcommand{\lab}{\label}
\newcommand{\al}{\alpha}

\newcommand{\D}{\Delta}
\newcommand{\be}{\beta}

\newcommand{\colt}[2]{\genfrac{}{}{0pt}{1}{#1}{#2}}

\renewcommand{\d}{\mathrm{d}}
\renewcommand{\leq}{\leqslant}
\renewcommand{\geq}{\geqslant}

\newcommand{\cQ}{\overline{\mathbb{Q}}}

\newcommand\nub{N_{U}(B)}
\newcommand{\Esix}{{\mathbf E}_6}
\newcommand{\Dfour}{{\mathbf D}_4}
\newcommand{\Dfive}{{\mathbf D}_5}
\newcommand{\tS}{{\widetilde S}}

\DeclareMathOperator{\Res}{Res}
\DeclareMathOperator{\Div}{Div}
\DeclareMathOperator{\PDiv}{PDiv}
\DeclareMathOperator{\meas}{meas}
\DeclareMathOperator{\NS}{NS}

\DeclareMathOperator{\ord}{ord}
\DeclareMathOperator{\Pic}{Pic}

\DeclareMathOperator{\hcf}{gcd}

\DeclareMathOperator{\Gal}{Gal}
\DeclareMathOperator{\err}{error}
\DeclareMathOperator{\eff}{\Lambda_{\mathrm{eff}}}

\newtheorem{ex}{Exercise}

\theoremstyle{definition}
\newtheorem*{ack}{Acknowledgements}

\begin{document}

\title{The Manin conjecture in dimension $2$}
\author{T.D. Browning}
\address{School of Mathematics, University of Bristol, Bristol BS8 1TW}
\email{t.d.browning@bristol.ac.uk}


\date{\today}

\maketitle

\begin{center}
\includegraphics[scale=0.57]{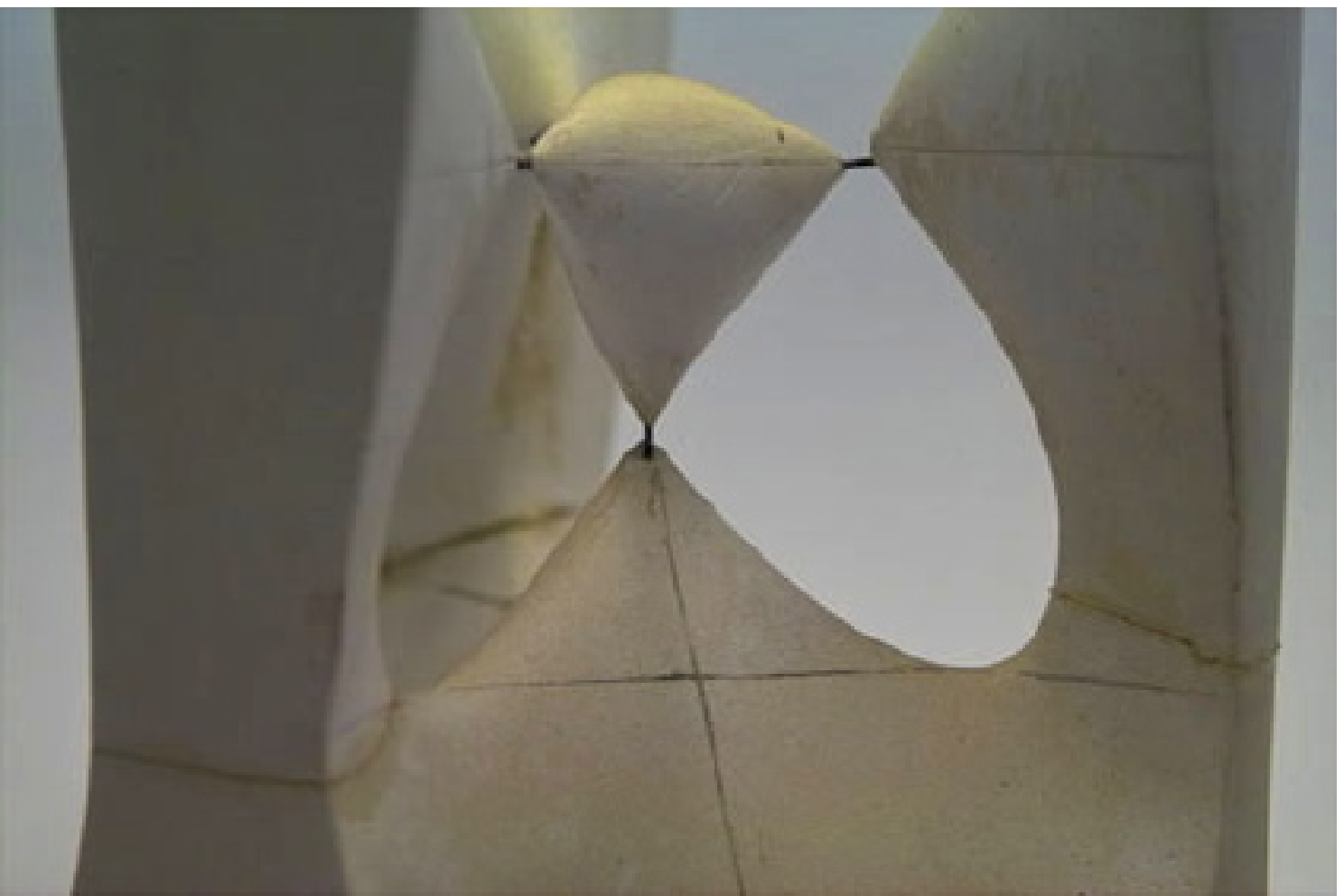}
\end{center}

\tableofcontents

\section{Introduction}\label{s:intro}

The study  of  integer solutions to Diophantine equations is a topic
that is almost as old as mathematics itself. Since its inception at the hands of 
Diophantus of Alexandria in 250 A.D., it has been 
found to relate to virtually every mathematical field. 
The purpose of these lecture notes is to focus attention upon an
aspect of Diophantine equations that has only crystallised within the last
few decades, and which exhibits a fascinating interplay
between the subjects of analytic number theory and algebraic geometry.

Suppose that we are given a polynomial $f\in \Z[x_1,\ldots,x_n]$, and
write 
$$
S_f:=\{\x=(x_1,\ldots,x_n)\in\Z^n\setminus\{\ma{0}\}: f(\x)=0\}
$$ 
for the corresponding locus of non-zero solutions.
There are a number of basic questions that can be asked about the
set $S_f$. When is $S_f$ non-empty? How large is $S_f$ when it is non-empty? 
When $S_f$ is infinite can we describe the set in some way?
A lot of the work to date has been driven by trying to understand the
situation for equations in only $n=2$ or $3$ variables.  The last $50$ years in
particular has delivered a remarkable level of understanding
concerning the arithmetic of curves.  In stark contrast to
this, the situation for equations in $4$ or more
variables remains a relatively untamed frontier, with only a
scattering of results available.

We will restrict attention to the study of Diophantine equations $f=0$ 
for which the corresponding zero set $S_f$ is infinite. The
description that we will aim for is quantative in 
nature, the main goal being to understand how the counting function
\begin{equation}
  \label{eq:count'}
N(f;B):=\#\{\x\in S_f: |\x| \leq B\}
\end{equation}
behaves, as $B \rightarrow \infty$. Here, as throughout these lecture notes,
$|\ma{z}|$ denotes the norm $\max_{1\leq i\leq n}|z_i|$ for any $\ma{z}\in
\R^n$.  Aside from being intrinsically interesting in their own right,
as we will see shortly, the study of functions like $ N(f;B)$ is
often an effective means of determining whether or not the equation
$f=0$ has any non-trivial integer solutions at all. 
In many applications of the Hardy--Littlewood circle method, for
example, one is able to prove that $S_f$ is non-empty by showing that $N(f;B)>0$ for large enough values of $B$.
In fact the method usually carries with it a proof of the fact that $S_f$ is
infinite. In the context of the circle method at least, it is useful
to have a general idea of which polynomials $f$ might have an infinite 
zero locus $S_f$.

Suppose that $f\in\Z[x_1,\ldots,x_n]$ has degree $d\geq
1$.  Then for the vectors $\x\in\Z^n$ counted by $N(f;B)$, the values
of $f(\x)$ will all be of order $B^d$. In fact a
positive proportion of them will have exact order $B^d$.  Thus
the probability that a randomly chosen value of $f(\x)$ should
vanish might be expected to be of order $1/B^{d}$.  Since the number of
$\x$ to be considered has order $B^n$, this leads us to the following
general expectation.

\begin{heuristic}
When $n\geq d$ we have
\begin{equation}
  \label{eq:heuristic}
B^{n-d}\ll N(f;B)\ll B^{n-d}.
\end{equation}
\end{heuristic}

As a crude first approximation, therefore, this heuristic tells us that
we might expect polynomials whose degree does not exceed the number of
variables to have infinitely many solutions.  Unfortunately there are
a number of things that can conspire to upset this heuristic expectation.
First and foremost, local conditions will often provide a reason for
$N(f;B)$ to be identically zero no matter the values of $d$ and
$n$.  By local obstructions we mean that the obvious necessary
conditions for $S_f$ to be non-empty fail. These are the conditions
that the equation $f(\x)=0$ should have a real solution $\x\in\R^n$,
and secondly, that the congruence
$$
f(\x)\equiv 0 \pmod{p^k}
$$
should be soluble for every prime power $p^k$.
When $f$ is homogeneous we must take care to ignore the trivial solution
$\x=(0,\ldots,0)$ in both cases.  

It is quite easy to construct examples that illustrate the failure of
these local conditions. For example, when $d$ is even, the equation
$$
x_1^{2d}+\cdots+x_n^{2d}=0
$$ 
doesn't have any integer solutions, since
it patently doesn't have any real solutions.  
Let us now exhibit an example, due to Mordell
\cite{mordell}, of a polynomial equation that fails to have
integer solutions because it fails to have solutions as a congruence
modulo a prime $p$.  Let $K$ be an algebraic number field
of degree $d$, with ring of integers $\mathcal{O}_K$. Write 
$$
\ma{N}(y_1,\ldots,y_{d}):=N_{K/\Q}(y_1\omega_1+\cdots+y_{d}\omega_d)
$$
for the corresponding norm form, where $\omega_1,\ldots,\omega_d$ is
a basis for $K$ over $\Q$. It is clear that $\ma{N}$ is a
homogeneous polynomial of degree $d$, with coefficients in $\Z$. 

\begin{ex}\label{ex:norm}
Let $\y\in \Z^n$ and let $p$ be a rational prime such that the ideal
$(p)\subset \mathcal{O}_K$ is prime.
Show that $p\mid \ma{N}(\y)$ if and only if $p\mid \y$.
\end{ex}

We define the homogeneous polynomial
\begin{equation}
  \label{eq:f1}
  f_1:=\mathbf{N}(x_1,\ldots,x_d)+p\mathbf{N}(x_{d+1},\ldots,x_{2d})
+\cdots+p^{d-1}\mathbf{N}(x_{d^2-d+1},\ldots,x_{d^2}),
\end{equation}
which has degree $d$ and $d^2$ variables.
We claim that the only integer solution to the equation
$f_1(\x)=0$ is the trivial solution $\x=\ma{0}$. To see this we
argue by contradiction. Thus we suppose there to be a vector 
$\x\in\Z^{d^2}$ such that $f_1(\x)=0$, with $\gcd(x_1,\ldots,x_{d^2})=1$.
Viewed modulo $p$ we deduce that $p\mid \mathbf{N}(x_1,\ldots,x_{d})$, whence
$p\mid x_1,\ldots x_d$ by Exercise \ref{ex:norm}.
Writing $x_i=py_i$ for $1\leq i\leq d$, and substituting into the
equation $f_1=0$, we deduce that
$$
p^{d-1}\mathbf{N}(y_1,\ldots,y_d)+\mathbf{N}(x_{d+1},\ldots,x_{2d})
+\cdots+p^{d-2}\mathbf{N}(x_{d^2-d+1},\ldots,x_{d^2})=0.
$$
But then we deduce in a similar fashion that $p\mid
x_{d+1},\ldots,x_{2d}$. We may clearly continue in this fashion,
ultimately concluding that $p\mid x_1,\ldots,x_{d^2}$, which is a
contradiction. This polynomial illustrates that for any
$d$ it is possible to construct examples of homogeneous polynomials in $d^2$
variables that have no non-zero integer solutions. This fits with
the facts rather well: when $d=2$ we know from Meyer's theorem that an
indefinite quadratic form always has non-trivial solutions as soon as
its rank is at least $5$. Similarly, it is conjectured that $10$
variables are always enough to ensure the solubility in integers of an
arbitrary homogeneous cubic equation. 

So far we have only seen examples of polynomials $f$ for which the
zero locus $S_f$ is empty. In this case the  corresponding counting
function $N(f;B)$ is particularly easy to estimate! There are also examples which show that $N(f;B)$
may grow in quite unexpected ways, even when $n\geq d$. 
An equation that illustrates excessive growth is provided by the
polynomial
\[
f_2:=x_1^d-x_2(x_3^{d-1}+\cdots+x_n^{d-1}).
\]
Here there are ``trivial'' solutions of the type $(0,0,a_3,\ldots,a_n)$ which
already contribute $\gg B^{n-2}$ to the counting function $N(f;B)$, whereas
\eqref{eq:heuristic} predicts that we should have exponent $n-d$.

It is also possible to construct examples of varieties which demonstrate inferior
growth. Let $n>d^2$ and choose any $d^2$ linear forms
$L_1,\ldots,L_{d^2}\in\Z[x_1,\ldots,x_n]$ that are linearly
independent over $\Q$. Consider the form 
$$
f_3:=f_1(L_1(x_1,\ldots,x_n),\ldots,L_{d^2}(x_1,\ldots,x_n)),
$$
where $f_1$ is given by \eqref{eq:f1}.  Then it is clear that
$N(f_3;B)$ has the same order of magnitude as the counting function
associated to the system of linear forms $L_1=\cdots=L_{d^2}=0$. Since
these forms are linearly independent we deduce that $N(f_3;B)$ has
order of magnitude $B^{n-d^2}$, whereas \eqref{eq:heuristic} led us to
expect an exponent $n-d$.  

We have seen lots of reasons why \eqref{eq:heuristic} 
might fail ---
how about some evidence supporting it? One of the most outstanding
achievements in this direction is the following very general result
due to Birch \cite{birch-61}.

\begin{thm}\label{t:birch}
Suppose $f\in\Z[x_1,\ldots,x_n]$ is a non-singular homogeneous
polynomial of degree $d$ in $n>(d-1)2^d$ variables.
Assume that $f(\x)=0$ has non-trivial solutions in $\R$ and each
$p$-adic field $\Q_p$.  Then there is a constant $c_f>0$ such that
\[
N(f;B)\sim c_f B^{n-d},
\]
as $B\rightarrow\infty$.
\end{thm}

Birch's result doesn't apply to either of the polynomials $f_2,f_3$
that we considered above, since both of these actually have a rather
large singular locus.
Since generic homogeneous polynomials are non-singular, Birch's result answers our
initial questions completely for typical forms with $n>(d-1)2^d$.  It would be of
considerable interest to reduce the lower bound for $n$, but except
for $d\leq 4$ this has not been done.  Theorem \ref{t:birch} is
established using the circle method, and exhibits a common feature of
all Diophantine problems successfully tackled via this machinery:
the number of variables involved needs to be large compared to
the degree. In particular, there is an obvious disparity between the
range for $n$ in Birch's result and the range for $n$ in
\eqref{eq:heuristic}. 
The main aim of these lecture notes is to discuss the
situation when $n$ is comparable in size with $d$.

\medskip

It turns out that phrasing things in terms of single polynomial
equations is far too restrictive. It is much more satisfactory to work
with  projective algebraic varieties $V\subseteq \bfP^{n-1}$.  
All of the varieties that we will work with are assumed to be cut out
by a finite system of homogeneous equations, all of which are defined over $\Q$.  
In line with the above, our main interest lies with those varieties
for which we expect the set $V(\Q)=V\cap \bfP^{n-1}(\Q)$ to be infinite.
Let $x=[\x] \in \bfP^{n-1}(\Q)$ be a projective rational point, with
$\x=(x_1,\ldots,x_n)\in\Z^n$ chosen so that
$\hcf(x_1,\ldots,x_{n})=1$. Then we define the height of $x$ to be 
$
H(x):=|\x|,
$
where as usual $|\ma{z}|$ denotes the norm $\max_{1\leq i\leq
  n}|z_i|$. Given any subset $U\subseteq V$, we may then define the
counting function
\begin{equation}\lab{count}
\nub:= \#\{ x \in U(\Q):  H(x) \leq B\}, 
\end{equation}
for each $B\geq 1$. The main difference between this counting function 
and the quantity introduced in \eqref{eq:count'} is that we
are now only interested in {\em primitive} integer solutions, by which
we mean that the components of the vector $\x\in\Z^n$ should share no common prime factors.
When the polynomial in \eqref{eq:count'} is homogeneous, this
formulation has the advantage of treating all scalar multiples of a
given non-zero integer solution as a single point.

Recall the definition of the M\"obius function $\mu: \N\rightarrow
\{0,1\}$, which is given by
$$
\mu(n)=\left\{
\begin{array}{ll}
0, & \mbox{if $p^2\mid n$ for some prime $p$},\\
1, & \mbox{if $n=1$},\\
(-1)^r, & \mbox{if $n=p_1\cdots p_r$ for distinct primes $p_1,\ldots,p_r$}.
\end{array}
\right.
$$
The M\"obius function is a multiplicative arithmetic function, and
will play a very useful r\^ole in our work.

\begin{ex}\label{e:mob}
Let $S\subseteq \Z^n$ be an arbitrary set. 
Show that
$$
\#\{\x\in S: \hcf(x_1,\ldots,x_n)=1\}=\sum_{k=1}^\infty \mu(k)
\#\{\x\in S: k\mid x_i,  ~(1\leq i\leq n)\}.
$$
\end{ex}

We now have the tools with which to relate the counting function \eqref{count} to
our earlier counting function $N(f;B)$ in \eqref{eq:count'}, when $U=V$ and $V\subset \bfP^{n-1}$ is a
hypersurface with underlying homogeneous polynomial
$f\in\Z[x_1,\ldots,x_n]$.  On noting that $\x$ and
$-\x$ represent the same point in $\bfP^{n-1}$, it follows from
Exercise \ref{e:mob} that
\begin{equation}
  \label{eq:111.4}
N_V(B)=\frac{1}{2}\sum_{k=1}^\infty \mu(k) N(f;B/k).
\end{equation}
When $f$ is non-singular of degree $d$, with $n>(d-1)2^d$, it can be
deduced from Theorem \ref{t:birch} that
$N_V(B)\sim \tilde{c}_f B^{n-d}$, where 
$\tilde{c}_f=\frac{1}{2} \zeta(n-d)^{-1}c_f$.

Returning to the counting function \eqref{count}, it is easy to check
that $\nub$ is bounded 
for each $B$, no matter what the choice of $U$ and $V$.
This follows on combining the fact that $N_V(B)\leq N_{\bfP^{n-1}}(B)$
with the self-evident inequalities 
$
N_{\bfP^{n-1}}(B)\leq \#\{\x\in\Z^n: |\x|\leq B\}\leq (2B+1)^n.
$  
In fact it is not so hard to
establish an asymptotic formula for $N_{\bfP^{n-1}}(B)$.

\begin{ex}\label{ex:proj}
Let $n\geq 2$. Use Exercise \ref{e:mob} to show that
$$
N_{\bfP^{n-1}}(B)= \frac{2^{n-1}}{\zeta(n)}B^{n}+O_n\big(B^{n-1}(\log
B)^{b_n}\big),
$$
where $b_2=1$ and $b_n=0$ for $n>2$.
\end{ex}

\subsection{Notation}

Before embarking on the main thrust of these lecture notes, we take a
moment to summarise some of the key pieces of notation that we will
make use of.

\begin{itemize}
\item
$A(x)=O(B(x))$ means that there exists a constant $c>0$ and $x_0\in\R$
such that $|A(x)|\leq c B(x)$ for all $x\geq x_0$.  
Throughout our work we will follow the convention that the implied
constant is absolute unless explicitly indicated otherwise by an  appropriate subscript.
We will often use the alternative notation $A(x)\ll B(x)$
or $B(x)\gg A(x)$.
\item
$A(x)\asymp B(x)$ means  $A(x)\ll B(x)\ll  A(x)$.
\item
$A(x)=o(B(x))$ means  $\lim_{x\rightarrow \infty} A(x)/B(x)=0$.
\item
$A(x)\sim B(x)$ means $\lim_{x\rightarrow \infty} A(x)/B(x)=1$.
\item
$\N=\{1,2,3,...\}$ will denote the set of natural numbers.
\item
$Z^n$ will denote the set of primitive vectors in $\Z^n$, and $Z_*^n$
will denote the set of $\ma{v}\in Z^n$ such that $v_1\cdots v_n\neq 0$.
\item
$|\ma{z}|:=\max_{1\leq
  i\leq n}|z_i|$, for any vector $\z\in\R^n$.
\end{itemize}

\subsection{The Manin conjectures}\label{s:conjectures}

Around 1989 Manin initiated a program to relate the asymptotic
behaviour of counting functions to the intrinsic geometry of the
underlying variety, for suitable families of algebraic varieties. 
It is precisely this rich interplay between arithmetic and geometry that this set of
lecture notes aims to communicate.

Several of the varieties that we have looked at so far have many
rational points, in the sense that $N_{V}(B)$ grows like a power of
$B$. For such varieties it is natural to look at the quantity
$$
\be_{V}:=\lim_{B\rightarrow \infty} \frac{\log N_{V}(B)}{\log B},
$$
assuming that this limit exists.  In general we may consider $\be_U$
for any Zariski open subset $U\subseteq V$. It is clear that $\be_U$ gives a
measure of ``how  large'' the set $U(\Q)$ is, since we will have
$$
B^{\be_U-\ve}\ll N_U(B)\ll B^{\be_U+\ve}
$$
for sufficiently large values of $B$ and any $\ve>0$. 
The insight of Manin was to try and
relate $\be_U$ to the geometry of $V$ via the 
introduction of a certain quantity $\al(V)$. 
Before defining this quantity we will need some facts from algebraic
geometry. The facts that we will need are summarised in more detail in
the book of Hindry and Silverman \cite[\S A]{hind}.

Assume that $V\subset\bfP^{n-1}$ is non-singular
and let 
$\Div(V)$ be the free abelian group generated 
by finite formal sums of the shape $D=\sum n_Y Y$, with $n_Y\in\Z$ and
$Y$ running over geometrically 
irreducible codimension $1$ subvarieties of $V$. 
A divisor $D\in\Div(V)$ is {\em effective} if $n_Y\geq 0$ for all $Y$,
and $D$ is said to be {\em principal}  if 
$D=\sum_{Y}\ord_Y(f) Y=D_f$, say,  for some rational function $f\in \C(V)$.
The intuitive idea behind the definition of the $\ord_Y$
function for a codimension $1$ subvariety $Y$ is that $\ord_Y(f)=k$
if $f$ has a zero of order $k$ along $Y$, while $\ord_Y(f) = -k$ if
$f$ has a pole of order $k$ along $Y$. If $f$ has
neither a zero nor a pole along $Y$, then $\ord_Y(f) = 0$.
Since $D_f+D_g=D_{fg}$ and $D_{1/f}=-D_f$, the
principal divisors form a subgroup $\PDiv(V)$ of $\Div(V)$. 
We define the {\em geometric Picard group} associated to
$V$ to be 
$$
\Pic_{\cQ}(V):=\Div(V)/\PDiv(V).
$$
A divisor class $[D]\in\Pic_{\cQ}(V)$ is {\em effective} if there
exists an effective divisor in the class.
One may also construct the {\em geometric N\'eron--Severi group}
$\NS_{\cQ}(V)$, which is 
$\Div(V)$ modulo a further equivalence relation called ``algebraic
equivalence''. When $V$ is covered by curves
of genus zero, as in all the
cases of interest to us in these lecture notes, it turns out that 
$\NS_{\cQ}(V)=\Pic_{\cQ}(V)$.
We illustrate the definition of $\Pic_{\cQ}(V)$ by calculating it in
the simplest possible case $V=\bfP^{n-1}$.

\begin{lem}\label{lem:calPIC}
We have $\Pic_{\cQ}(\bfP^{n-1})=\Z$.
\end{lem}

\begin{proof}
An irreducible divisor on $\bfP^{n-1}$ has the form $Y=\{F=0\}$ for
some absolutely irreducible form $F\in\C[x_1,\ldots,x_n]$. For such a divisor, define the degree of $Y$ to be
$\deg Y = \deg F$. Extend the definition of degree additively, so that
$$
\deg\Big( \sum_Y n_Y Y\Big) =\sum_Y n_Y\deg Y.
$$
The map $\deg: \Div(\bfP^{n-1})\rightarrow \Z$ is clearly a homomorphism, and
to establish the lemma it will suffice to show that the kernel of this
map is precisely the subgroup $\PDiv(\bfP^{n-1})$. 
To see this, we note that $\deg D_f = 0$ for any rational function
$f = F_1/F_2$. Indeed, the sum of the
positive degree terms will be $\deg F_1$, whereas the sum of the negative degree
terms will be $\deg F_2$, and this two degrees must coincide in order
to have a well-defined rational function. Conversely, if $D = n_1Y_1 + \cdots + n_kY_k$
has degree zero, with $Y_i = \{F_i=0\}$ for $1\leq i\leq k$, then 
$f = F_1^{n_1}\cdots F_k^{n_k}$  is a well-defined rational function
on $\bfP^{n-1}$ with $D_f=D$. This completes the proof of the lemma.
\end{proof}

Returning to the setting of arbitrary non-singular varieties
$V\in\bfP^{n-1}$, let $H\in\Div(V)$ be a divisor
corresponding to a hyperplane section. Furthermore, let 
$K_V\in \Div(V)$ be the canonical
divisor. This is a common abuse of notation: really $K_V$ refers to
the class of $D_{\omega}$ in $\Pic_{\cQ}(V)$ for any differential
$(\dim V)$-form
$\omega$ of $V$.  It would take us too far afield to include precise
definitions of these objects here.
We may now define the real number
$$
\al(V):=\inf\{r\in\R: \mbox{$r[H]+[K_V]\in \eff(V)$}\},
$$
where 
$$
\eff(V):=\{ c_1[D_1]+\cdots+c_k [D_k]: c_i\in\R_{\geq 0},
~\mbox{$[D_i]\in \NS_{\cQ}(V)$ effective}\}
$$
is the so-called {\em effective cone of divisors}.
It does not matter too much if this definition is currently
meaningless: the main thing is that $\al(V)$ depends in an explicit
way on the geometry of $V$ over $\C$. We now have the following
basic conjecture due to Batyrev and Manin  \cite[Conjecture A]{b-m}.

\begin{con}\label{c:b<a}
For all $\ve>0$ there exists a Zariski open subset $U\subseteq V$ such
that
$
\be_{U}\leq \al(V)+\ve.
$
\end{con}

A non-singular variety $V\subset\bfP^{n-1}$ is said to be {\em Fano} if $K_V$ does not lie in
the closure of the effective cone $\eff(V)\subset \NS_{\cQ}(V)\otimes_\Z
\R$. This is equivalent to $-K_V$ being ample, and implies in
particular that $V$ is covered by rational curves. 
As an example, suppose that $V$ is a complete
intersection, with $V=W_1\cap \cdots \cap W_t$ for hypersurfaces $W_i \subset \bfP^{n-1}$ of degree $d_i$.
Then $V$ is Fano if and only if 
$
d_1+\cdots+d_t< n.
$  
With this in mind we have the following supplementary prediction.

\begin{con}\label{c:b=a}
Assume that $V$ is Fano and $V(\Q)$ is Zariski dense in $V$. Then 
there exists a Zariski open subset $U\subseteq V$ such
that
$
\be_{U}=\al(V).
$
\end{con}

We have $\al(V)=n-d_1-\cdots-d_t$ when $V$ is a non-singular complete intersection as above.
In particular, when $V$ is a hypersurface of degree $d$ we may deduce from Theorem
\ref{t:birch} that Conjecture \ref{c:b=a} holds when $n$ is
sufficiently large in terms of $d$.
It also holds for $n\geq 3$ when $d=2$ (see Heath-Brown \cite{hb-crelle}, for example).
Finally we remark that Conjecture \ref{c:b=a} holds
for projective space. This follows from Exercise
\ref{ex:proj} and the fact that $[K_{\bfP^{n-1}}]=[-nH]$ in
$\Pic_{\cQ}(\bfP^{n-1})$, whence
$\al(\bfP^{n-1})=n$. 

\medskip

The title of these lecture notes suggests that we will focus
our attention on the situation for varieties of dimension $2$.  Before
doing so, let us consider the situation for curves briefly. For
simplicity we will discuss only projective plane curves $V\subset
\bfP^2$ of degree $d$. There is a natural trichotomy among such curves, according to
the genus $g$ of the curve. For curves with $g=0$, otherwise known as
{\em rational curves}, it is possible to show that $N_V(B)\sim c_V
B^{2/d}$. This is in complete accordance with the Manin conjecture. 
It is an amusing exercise to check that such an asymptotic formula
holds with $d=2$ when $V$ is given by the equation
$x_1^2+x_2^2=x_3^2$, for example.
When $g=1$ and $V(\Q)\neq \emptyset$, the curve is elliptic and it has
been shown by N\'eron \cite[Theorem B.6.3]{hind} that
$$
N_V(B)\sim c_V (\log B)^{r_V/2},
$$
where $r_V$  denotes the rank of $V$.  Thus although there can be
infinitely many points in $V(\Q)$, we see that the corresponding
counting function grows much more slowly than for rational
curves. Elliptic curves are not Fano, and so this is not covered by
the Manin conjecture. However it does confirm Conjecture~\ref{c:b<a},
since $\al(V)=0$. 
When $g\geq 2$ the work of
Faltings  \cite{faltings} shows that $V(\Q)$ is always finite, and so
it does not make sense to study $N_V(B)$.

Let us now concern ourselves with Fano varieties of dimension $2$.
We begin with some simple-minded numerics. Suppose that we are given a Fano
variety $V$ of dimension $2$ and degree $d$, which is a non-singular
complete intersection in $\bfP^{n-1}$. 
Thus $V=W_1\cap \cdots \cap W_t$ for hypersurfaces $W_i
\subset \bfP^{n-1}$ of degree $d_i$, and we assume that the
intersection is transversal at a generic point of $V$. We are not
interested in hyperplane sections of $V$, and so we will assume without
loss of generality that $d_i\geq 2$ for each $1\leq i\leq t$.
Then the following inequalities must be satisfied:
\begin{enumerate}
\item $d_1+\cdots+d_t< n$, [Fano]
\item $n-1-t=2$, [complete intersection of dimension $2$]
\item $d=d_1\cdots d_t$, [B\'ezout]
\item $d_t\geq \cdots \geq d_1\geq 2$.
\end{enumerate}
It follows that the only possibilities are
$$
(d;d_1,\ldots,d_t;n;t)\in\big\{(2;2;4;1),(3;3;4;1),(4;2,2;5;2)\big\}.
$$
These surfaces correspond to a quadric in $\bfP^3$, a cubic surface in
$\bfP^3$, and an intersection of $2$ quadrics in $\bfP^4$, respectively.
We have already observed that the Manin conjecture holds for
quadrics. Hence one would like to examine the latter two
surfaces. In fact these are the most familiar examples of ``del Pezzo
surfaces''.  We will see in \S \ref{s:dp5} that not all del Pezzo surfaces are 
complete intersections,  and so we have missed out on several surfaces
in this analysis. Nonetheless, a substantial portion of these lecture notes will focus on 
cubic surfaces in $\bfP^3$ and intersections of $2$
quadrics in $\bfP^4$. 

It is now time to give a formal definition of a del Pezzo surface. 
Let us begin with a discussion of non-singular del Pezzo surfaces.
Let $d\geq 3$. Then a {\em del Pezzo surface of degree $d$} is a
non-singular surface $S \subset \bfP^d$ of degree $d$,
with very ample anticanonical divisor $-K_S$. This latter condition is
equivalent to the equality $[-K_S]=[H]$ in $\Pic_{\cQ}(S)$, for a
hyperplane section $H\in \Div(S)$. 
The facts that we will recall here are all established in the book of
Manin \cite{manin-book}, for example.  It is well-known that del Pezzo surfaces
$S\subset\bfP^d$  arise either as the quadratic Veronese embedding of a quadric
in $\bfP^3$, which is a del Pezzo surface of degree $8$ in
$\bfP^8$ (isomorphic to $\bfP^1\times \bfP^1$), or as the blow-up of $\bfP^2$ along $9-d$ points in general
position, in which case the degree of $S$ satisfies $3
\leq d \leq 9$.   We will meet the notion of ``general position'' 
when $d=3$ in \S \ref{s:lines}. 
Since $[-K_S]=[H]$ in $\Pic_{\cQ}(S)$, we see that
$\al(S)=1$ for non-singular del Pezzo surfaces of degree $d$. 

The geometry of del Pezzo surfaces is very beautiful and well-worth
studying. However, to avoid straying from the main focus of these
lecture notes, we will content ourselves with simply quoting the facts
that are needed. One of the remarkable features
of del Pezzo surfaces of small degree is that each such surface contains finitely
many lines. The precise number of lines is recorded in Table \ref{t:lines}.
\begin{table}[!ht]
\begin{center}
\begin{tabular}{|c|c|}
\hline
$d$ & number of lines  \\
\hline
\hline
$3$ & $27$\\
$4$ & $16$\\
$5$ & $10$\\
$6$ & $6$\\
\hline
\end{tabular}
\end{center}
\caption{Lines on non-singular del Pezzo surfaces of degree
  $d$}
\label{t:lines}
\end{table}

It turns out that dealing with del Pezzo surfaces of degree $d$ gets easier
as the degree increases. In these lecture notes we will focus our
attention on the del Pezzo surfaces of degree $d\in\{3,4,5,6\}$. 
It turns out that for del Pezzo surfaces of degree
$d$, the geometric Picard group
$\Pic_{\cQ}(S)$ is a finitely generated free $\Z$-module, with 
\begin{equation} 
  \label{eq:gPic} 
\Pic_{\cQ}(S)\cong \Z^{10-d}.
\end{equation} 
This is established in  Manin \cite{manin-book}, where an explicit basis for the group is also
provided (see \S \ref{s:lines} for a concrete example). 
Let $K$ be a splitting field for the finitely many lines
contained in $S$. The final invariant that we will need to introduce is the 
{\em Picard group} 
\begin{equation} 
  \label{eq:pg} 
\Pic(S):=\Pic_{\cQ}(S)^{\Gal(K/\Q)}
\end{equation} 
of the surface.  This is just the set of elements in $\Pic_{\cQ}(S)$
that are fixed by the action of the Galois group.
Write $\rho_S$ for the rank
of $\Pic(S)$. Let $U\subset S$ be the
Zariski open subset formed by deleting the finitely many lines from
$S$. Then we have the following \cite[Conjecture $C'$]{b-m}.

\begin{con}\label{c:manin-ns}
Suppose that $S\subset \bfP^d$ is a non-singular del Pezzo surface of
degree $d$.   Then there exists a non-negative constant $c_{S, H}$ such that 
\begin{equation}\lab{c1}
\nub = c_{S,H}B (\log B)^{\rho_S-1}\big(1+o(1)\big).
\end{equation}
\end{con}

In these lecture notes this is what will commonly be termed as ``the
Manin conjecture''.  Note that the exponent of
$B$ agrees with Conjecture \ref{c:b=a}, since $\al(S)=1$. 
Moreover the exponent of $\log B$ is at most $9-d$, since the geometric Picard group  has rank $10-d$.
We will develop some heuristics to support this power of $\log B$ in \S
\ref{s:further}. The value of the constant $c_{S,H}$ has also received
a conjectural interpretation  at the hands of Peyre \cite{MR1340296},
an interpretation that has been extended by Batyrev and
Tschinkel~\cite{b-t}, and  by Salberger~\cite{MR1679841}.

There are a number of refinements to Conjecture \ref{c:manin-ns} that
are currently emerging, which we will not have space to discuss here.
Some of these are discussed in more details in the author's survey
\cite[\S 2]{gauss}, for example.  One such refinement is that there should exist a
polynomial $P\in\R[x]$ of degree $\rho_S-1$,  and a real number $\delta>0$, such that
\begin{equation}\lab{c3}
\nub=B P(\log B) +O(B^{1-\delta}).
\end{equation}
One obviously expects the leading coefficient of $P$ to agree
with Peyre's prediction, but there has so far been rather little
investigation of the lower order terms. 
All of the del Pezzo surfaces that we have discussed so far have been
non-singular. In the following section we will meet some singular
ones.

\subsection{Degree $3$ surfaces}

The del Pezzo surfaces $S\subset \bfP^3$ of degree $3$ are 
the geometrically integral cubic surfaces in $\bfP^3$, which are not
ruled by lines.  In particular, this definition covers both singular and
non-singular del Pezzo surfaces of degree $3$.
Given such a surface $S$ defined over $\Q$, we may always find an absolutely
irreducible cubic form $C\in \Z[x_1,x_2,x_3,x_4]$ such that $S$ is
defined by the equation $C=0$.  In this section we will discuss
the Manin conjecture in the context of cubic surfaces.
Let us begin by considering the situation for non-singular cubic
surfaces, for which one takes $U\subset S$ to be the open subset formed by
deleting the famous $27$ lines.  Peyre and Tschinkel \cite{p-t1, p-t2} have provided 
ample numerical evidence for the validity of the Manin conjecture for
diagonal cubic surfaces. However we are still rather far away from
proving it for any single example.  The best upper bound available is 
\begin{equation}
  \label{eq:hh} 
\nub=O_{\varepsilon,S}(B^{4/3+\varepsilon}),
\end{equation}
due to Heath-Brown \cite{MR98h:11083}.  This applies when the surface
$S$ contains $3$ coplanar lines defined over $\Q$, and in particular
to the {\em Fermat cubic surface} 
$$
x_1^3+x_2^3=x_3^3+x_4^3.
$$
Heath-Brown \cite{hb-ast} has extended the bound \eqref{eq:hh}
to all non-singular cubic surfaces, subject to a natural conjecture
concerning the size of the rank of elliptic curves over $\Q$.

The problem of proving lower bounds is somewhat easier.  Under the
assumption that $S$ contains a pair of skew lines defined over $\Q$,
Slater and Swinnerton-Dyer \cite{s-swd} have shown that $\nub\gg_S
B(\log B)^{\rho_S-1}$, as predicted by the Manin conjecture.
This does not apply to the Fermat cubic surface, however,
since the only skew lines contained in this surface are defined over
$\Q(\sqrt{-3})$.

It turns out that much more can be said if one permits $S$ to contain 
isolated singularities.  For the remainder of this section let 
$S \subset \bfP^3$ be a geometrically integral cubic surface, which
has only isolated singularities and is 
not a cone. 
Then there exists a unique ``minimal desingularisation'' $\pi: \tS\rightarrow S$ of the
surface, which is just a sequence of blow-up maps, and furthermore, that the asymptotic formula \eqref{c1} is
still expected to hold, with $\rho_S$ now taken to be the rank of the Picard group
of $\tS$. As usual $U\subset S$ is obtained  by deleting all of the lines from $S$.
The classification of singular cubic surfaces $S$ is a
well-established subject, and can be traced back to the work of Cayley \cite{cayley} and Schl\"afli
\cite{cayley'} over a century ago. A contemporary classification of
singular cubic surfaces has since been given by Bruce and Wall
\cite{b-w}, over $\overline{\Q}$. 
Of course, if one is interested in a classification over the ground
field $\Q$, then many more singularity types can occur (see Lipman
\cite{lipman}, for example).
In Table \ref{t:class1}  we have provided a classification table of the $20$ singularity
types over $\overline{\Q}$, including the number of lines that each surface contains.
We will presently meet some explicit examples of cubic forms $C\in\Z[x_1,x_2,x_3,x_4]$
that typify some surface types. 

\begin{table}[!ht]
\begin{center}
\begin{tabular}{|c|c|c|}
\hline
type &  $\#$ lines & singularity \\
\hline
\hline
\texttt{i} &  $21$ & $\textbf{A}_1$\\
\texttt{ii} & $16$ & $2\textbf{A}_1$\\
\texttt{iii} & $15$ & $\textbf{A}_2$\\
\texttt{iv} & $12$ & $3\textbf{A}_1$ \\
\texttt{v} & $11$ & $\textbf{A}_1+\textbf{A}_2$\\
\texttt{vi} & $10$ & $\textbf{A}_3$ \\
\texttt{vii} & $9$& $4\textbf{A}_1$ \\
\texttt{viii} & $8$ & $2\textbf{A}_1+\textbf{A}_2$ \\
\texttt{ix} & $7$ & $\textbf{A}_1+\textbf{A}_3$ \\
\texttt{x} & $7$ & $2\textbf{A}_2$ \\
\texttt{xi} & $6$ & $\textbf{A}_4$ \\
\texttt{xii} & $6$ & $\textbf{D}_4$\\
\texttt{xiii} &$5$ & $2\textbf{A}_1+\textbf{A}_3$ \\
\texttt{xiv} & $5$ & $\textbf{A}_1+2\textbf{A}_2$ \\
\texttt{xv} & $4$ & $\textbf{A}_1+\textbf{A}_4$ \\
\texttt{xvi} & $3$ & $\textbf{A}_5$ \\
\texttt{xvii} & $3$ & $\textbf{D}_5$ \\
\texttt{xviii} & $3$ & $3\textbf{A}_2$ \\
\texttt{ix}  & $2$ & $\textbf{A}_1+\textbf{A}_5$ \\
\texttt{xx} & $1$ & $\textbf{E}_6$ \\
\hline
\end{tabular}
\end{center}
\caption{Classification (over $\overline{\Q}$) of singular del Pezzo
surfaces of degree $3$ in $\bfP^3$}
\label{t:class1}
\end{table}

The labelling of each singularity type corresponds to the
``Dynkin diagram'' that describes the intersection behaviour of the 
exceptional divisors obtained by resolving the
singularities in the surface. 
For example, consider the cubic surface
\begin{equation}
  \label{eq:e6}
S_1=\{x_1^2x_3+x_2x_3^2+x_4^3=0\}.
\end{equation}
Up to isomorphism over $\overline{\Q}$ this is the unique cubic
surface of type \texttt{xx} in the table, and is discussed further in \cite{MR2029868}.
The process of resolving the singularity gives $6$
exceptional divisors $E_1,\ldots,E_6$ and produces the 
minimal desingularisation $\tS_1$ of the surface $S_1$.
If $L$ denotes the strict transform of the unique line on $S_1$, then 
$L,E_1,\ldots,E_6$ satisfy the intersection behaviour
encoded in the Dynkin diagram
$$
\xymatrix{
      & & E_2 \ar@{-}[d] \\
E_1 \ar@{-}[r] & E_3 \ar@{-}[r] & E_6 \ar@{-}[r] & E_5 \ar@{-}[r] & E_4
\ar@{-}[r] & L }
$$
There is a line connecting two divisors in this diagram if and only if they meet in
$\tS_1$.  In what follows the reader can simply think of these
Dynkin diagrams as a convenient way to label the surface type.

It turns out, as discussed in \cite{b-w},  that some types of surfaces
do not have a single normal form,
but an infinite family.  This happens precisely for the surfaces of type
\texttt{i}, \texttt{ii}, \texttt{iii}, \texttt{iv}, \texttt{v},
\texttt{vi} and \texttt{ix}. By \cite[Lemma 4]{b-w} the type \texttt{xii}
surface, with a $\Dfour$ singularity, is the only surface that has
more than one normal form, but not a family.  In fact it has precisely
two normal forms, given by 
\begin{equation}
  \label{eq:S2}
S_2=\{x_1x_2(x_1+x_2)+x_4(x_1+x_2+x_3)^2=0\}
\end{equation}
and 
\begin{equation}
  \label{eq:d4s3}
S_3=\{x_1x_2x_3+x_4(x_1+x_2+x_3)^2=0\}.
\end{equation}
That these equations actually define distinct surfaces can be seen by calculating the
corresponding Hessians in each case.

Let $\tS$ denote the minimal desingularisation of any surface $S$ from
Table~\ref{t:class1}, and assume that all of its singularities and
lines are defined over $\Q$. In this case the surface is said to be
{\em split}, and it follows that the Picard
group of  $\tS$ has maximal rank $7$ by \eqref{eq:gPic}, since $\Pic(\tS)=\Pic_{\cQ}(\tS)$. 
For example, $[L],[E_1],\ldots,[E_6]$ provide a basis for $\Pic(\tS_1)$.
One would like to try and establish \eqref{c1} for each such surface
$S$, with $\rho_S=7$.  Several del Pezzo surfaces are actually special cases of
varieties for which the Manin conjecture is already known to hold.  
Recall that a variety of dimension $D$ is said to be {\em toric} if it contains 
the algebraic group variety $\mathbb{G}_m^{D}$ as a dense open subset, 
whose natural action on itself extends to all of the variety.
The Manin conjecture has been established for all 
toric varieties by  Batyrev and Tschinkel \cite{MR1620682}.
It can be checked that the 
surface representing type \texttt{xviii} is toric.  
In fact this particular surface has been studied by numerous
authors, including la Bret\`eche \cite{MR2000b:11074},
la Bret\`eche and Swinnerton-Dyer \cite{b-swd},
 Fouvry \cite{MR2000b:11075}, 
Heath-Brown and Moroz \cite{MR2000f:11080}, and Salberger \cite{MR1679841}.
Of the unconditional asymptotic formulae obtained, the most impressive is the first.
This consists of an estimate like \eqref{c3} for any $\delta \in (0,1/8)$, with
$\deg P=6$.  

The next surface to have received serious attention is the {\em Cayley
cubic surface} 
$$
S_4=\{x_1x_2x_3+x_1x_2x_4+x_1x_3x_4+x_2x_3x_4=0\},
$$
which is the type \texttt{vii} surface in the table.
Heath-Brown \cite{hb-cayley} has shown that
there exist absolute constants $A_1,A_2>0$ such that 
$$
A_1 B (\log B)^6 \leq \nub \leq A_2 B (\log B)^6.
$$
An estimate of precisely the same form has been obtained by the author
\cite{d4} for the $\Dfour$ surface $S_3$ in \eqref{eq:d4s3}.
In both cases the lines in the surface are all defined over $\Q$, so that the
surfaces are split. Thus the corresponding Picard groups have rank $7$
and the exponents of $B$ and $\log B$ agree with Manin's
prediction. In this set of lecture notes we will establish an upper
bound for the remaining $\Dfour$ cubic surface $S_2$ in \eqref{eq:S2}.  
This will be carried out in \S \ref{s:d4_cubic} in two basic
attacks. First we will
give a completely self-contained account of  the upper bound
$\nub=O_\ve(B^{1+\ve})$, for any $\ve>0$. Next, by making use of the
work in \cite{d4}, we will establish the following finer result. 

\begin{thm}\label{main-d4}
Let $S_2$ be given by \eqref{eq:S2}. 
We have $\nub \ll B (\log B)^6.$
\end{thm}

The cubic surface $S_2$ contains the unique singular point
$[0,0,0,1]$, together with the $6$ lines
\begin{equation}
  \label{eq:d4_lines}
x_i=x_4=0, \quad x_1+x_2=x_j=0,\quad 
x_i=x_1+x_2+x_3=0,
\end{equation}
for distinct indices $i\in\{1,2\}$ and $j\in\{3,4\}$.
Thus the surface is split and it follows that Theorem \ref{main-d4}
agrees with the Manin conjecture.

\begin{ex}
Check that \eqref{eq:d4_lines} are all of the lines contained in $S_2$. 
\end{ex}

The final surface to have been studied extensively is the $\Esix$
cubic surface $S_1$ that we discussed above.  
The figure below, which was constructed by Derenthal, shows all  the rational points
of height $\leq 1000$ on this surface. Recent
\begin{wrapfigure}{l}{7.9cm}
   \includegraphics[scale=0.26]{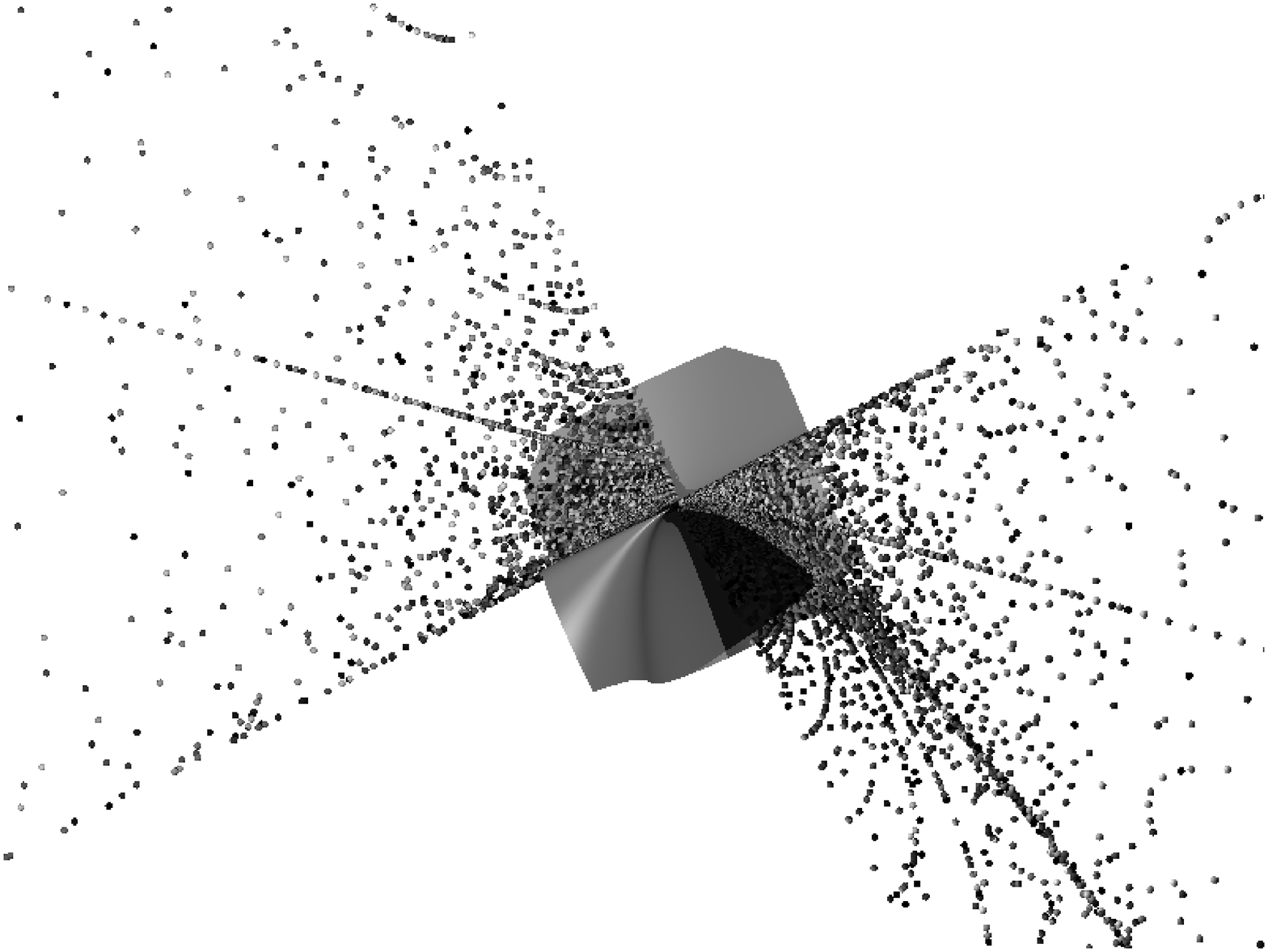}
\end{wrapfigure}
joint work of the
author with la Bret\`eche and Derenthal \cite{e6} has succeeded in
establishing the Manin conjecture for this surface. In fact an 
asymptotic formula of the shape \eqref{c3} is achieved, with 
$P$ of degree $6$ and any $\delta \in (0,1/11)$. 
It should be remarked that Dr. Michael Joyce has 
also established the Manin conjecture 
for $S_1$ in his doctoral thesis, albeit with
a weaker error term of  $O(B (\log B)^5)$.

\subsection{Degree $4$ surfaces}

A quartic del Pezzo surface $S \subset \bfP^4$, that is defined over $\Q$, can be
recognised as the zero locus of a suitable pair of quadratic forms 
$Q_1,Q_2 \in \Z[x_1,\ldots,x_5]$.  Again we do not stipulate that the
surface should be non-singular. 
As usual let  $U\subset S$
denote the open subset formed by deleting all of the lines from $S$.
Let us begin by discussing the situation for non-singular surfaces,
where there are $16$ lines to delete. The best result available is the estimate
$
\nub=O_{\ve,S}(B^{1+\ve}),
$
valid for any quartic non-singular del Pezzo surface
$S\subset \bfP^4$ containing a conic defined over $\Q$. 
This result was established in an unpublished note
due to Salberger in 2001.
It would be interesting to see whether one could adapt the methods of 
\cite{hb-ast} to show that $\nub =O_{\ve,S}(B^{5/4+\ve})$ for {\em
  any} non-singular del Pezzo surface of degree $4$, assuming 
the same hypothesis on the ranks of elliptic curves.

As previously, it emerges that much more can be said if one permits $S$ to contain 
isolated singularities.  For the remainder of this section let 
$S \subset \bfP^4$ be a geometrically integral intersection of two
quadric hypersurfaces, which has only isolated singularities and is
not a cone, and let $\tS$ be the minimal desingularisation of $S$. Then the asymptotic formula \eqref{c1} is
still expected to hold, with $\rho_S$ now taken to be the rank of the Picard group
of $\tS$, and $U\subset S$ obtained  by deleting all of the lines from $S$.
In particular, when $S$ is split one always has $\rho_S=6$. 
The classification of singular quartic del Pezzo surfaces can be
extracted from the work of Hodge and Pedoe \cite[Book IV, \S
XIII.11]{h-p}, where it is phrased in terms of the so-called Segre
symbol.  The {\em Segre symbol} of a matrix $\ma{M}\in M_5(\C)$ is defined as follows. If the Jordan
form of $\ma{M}$ has Jordan blocks of sizes $a_1,\ldots,a_n$, with $a_1+\cdots+a_n=5$, then the Segre
symbol is the symbol 
$$
(a_1,\ldots,a_n)
$$ 
with extra parentheses around the Jordan blocks with
equal eigenvalues.  Suppose that our quartic del Pezzo surface $S$ is
defined by a pair of quadric hypersurfaces, with underlying symmetric matrices
$\ma{A},\ma{B}\in M_5(\Q)$. Then the Segre symbol of $S$ is defined to be the
Segre symbol associated to $\ma{A}^{-1}\ma{B}$. A crucial property of
the Segre symbol is that it does not depend on the choice of $\ma{A}$
and $\ma{B}$ in the pencil of quadrics defining $S$.  
Since we are assuming that $S$ is not a cone, one may always suppose
that $\ma{A},\ma{B}$ are chosen so that $\ma{A}$ has full rank.

To illustrate the calculation of the Segre symbol, let us consider the
surface $S$ defined by the pair of equations
\begin{equation}
  \label{eq:rom_walk}
x_1x_2+x_3x_4=0,\quad x_1x_4+x_2x_3+x_3x_5+x_4x_5=0.
\end{equation}
Let $\ma{A},\ma{B}\in M_5(\Q)$ denote the underlying matrices of the
first and second equations, respectively. Then $\ma{A}$ has rank 
$4$, and so we
replace it with $\ma{A}+2\ma{B}$, which has full rank. A simple
calculation reveals that the matrix 
$(\ma{A}+2\ma{B})^{-1}\ma{B}$ has Jordan form
$$
\ma{J}=\left(
\begin{array}{ccccc}
0&0&0&0&0\\
0&1&0&0&0\\
0&0&\frac{1}{3}&0&0\\
0&0&0&\frac{1}{2}&1\\
0&0&0&0&\frac{1}{2}
\end{array}
\right).
$$
This matrix has $4$ Jordan blocks, one of size $2$ and the rest of
size $1$. The eigenvalues associated to the different Jordan blocks
are all different, and so it follows that the surface \eqref{eq:rom_walk} has Segre symbol
$
(2,1,1,1).
$

\begin{ex}
Find any matrices $\ma{A},\ma{B}\in M_5(\Q)$ so that the corresponding
surface $\x^t \ma{A}\x=\x^t \ma{B}\x=0$ is non-singular. Show that the
surface has Segre symbol $(1,1,1,1,1)$.
\end{ex}

So far we have given a very easy way to check the isomorphism type of a
given singular del Pezzo surface of degree $4$. How do we match this
up with a classification according to the singularity type, as in our discussion of
cubic surfaces in Table \ref{t:class1}?
It turns out that up to isomorphism over $\overline{\Q}$, there are
$15$ possible singularity types for $S$.
Over $\overline{\Q}$, Coray and Tsfasman  \cite[Proposition 6.1]{c-t} have calculated the
extended Dynkin diagrams for all of the $15$ types, and Kn\"orrer
\cite{knorrer} has determined the precise correspondence between the singularity type and
the Segre symbol. Table \ref{t:class2} is extracted from
this body of work, and matches each possible singularity type with 
the Segre symbol, and the number of lines that the surfaces contains.

\begin{table}[!ht]
\begin{center}
\begin{tabular}{|c|c|c|c|c|}
\hline
type & Segre symbol & $\#$ lines & singularity \\
\hline
\hline
\texttt{i} & (2,1,1,1) & $12$ & $\textbf{A}_1$\\
\texttt{ii} & (2,2,1)& $9$ & $2\textbf{A}_1$\\
\texttt{iii} & ((1,1),1,1,1) & $8$ & $2\textbf{A}_1$\\
\texttt{iv} & (3,1,1) & $8$ & $\textbf{A}_2$\\
\texttt{v} & ((1,1),2,1)   & $6$ & $3\textbf{A}_1$ \\
\texttt{vi} & (3,2)  & $6$ & $\textbf{A}_1+\textbf{A}_2$ \\
\texttt{vii} & (4,1) & $5$& $\textbf{A}_3$ \\
\texttt{viii} & ((2,1),1,1)  & $4$& $\textbf{A}_3$ \\
\texttt{ix} & ((1,1),(1,1),1)  & $4$ & $4\textbf{A}_1$ \\
\texttt{x} & ((1,1),3) & $4$ & $2\textbf{A}_1+\textbf{A}_2$ \\
\texttt{xi} & ((2,1),2) & $3$ & $\textbf{A}_1+\textbf{A}_3$ \\
\texttt{xii} &(5)  & $3$ & $\textbf{A}_4$ \\
\texttt{xiii} & ((3,1),1) & $2$ & $\textbf{D}_4$ \\
\texttt{xiv} & ((2,1),(1,1))  & $2$ & $2\textbf{A}_1+\textbf{A}_3$ \\
\texttt{xv} & ((4,1)) & $1$ & $\textbf{D}_5$ \\
\hline
\end{tabular}
\end{center}
\caption{Classification (over $\overline{\Q}$) of singular del Pezzo
  surfaces of degree $4$ in $\bfP^4$}
\label{t:class2}
\end{table}

In general, given a particular Segre symbol, its not entirely
straightforward to determine explicit equations that define a singular del Pezzo
surface of degree $4$ having this symbol. Nonetheless in Table \ref{t:class2'} we
have done precisely this for each Segre symbol that
occurs. In doing so we have retrieved some of the calculations
carried out by Derenthal \cite{der1}. An important feature of the table is that the surfaces
recorded are split over $\Q$. It remains a
significant open challenge to establish the Manin conjecture for the $15$
surfaces given in Table \ref{t:class2'}. This will furnish a proof of
the Manin conjecture for the class of split singular del Pezzo
surfaces of degree $4$ that are defined over $\Q$, and is undoubtedly a
key stepping stone on the way towards a resolution of the conjecture for
all del Pezzo surfaces. There is huge potential for further work in
this area, and I hope that these lecture notes succeed in showing that analytic number
theorists are well placed to make an important contribution.

\begin{table}[!ht]
\begin{center}
\begin{tabular}{|c|c|c|}
\hline
type & $Q_1(\x)$ & $Q_2(\x)$  \\
\hline
\hline
\texttt{i} & $x_1x_2-x_3x_4$ & $x_1x_4-x_2x_3+x_3x_5+x_4x_5$\\
\texttt{ii} & $x_1x_2-x_3x_4$ & $x_1x_4-x_2x_3+x_3x_5+x_5^2$\\
\texttt{iii} & $x_1x_2-x_3^2$ & $x_1x_3-x_2x_3+x_4x_5$\\
\texttt{iv} & $x_1x_2-x_3x_4$ & $(x_1+x_2+x_3+x_4)x_5-x_3x_4$\\
\texttt{v} & $x_1x_2-x_3^2$ & $x_2x_3+x_3^2 +x_4x_5$\\
\texttt{vi} & $x_1x_2-x_3x_4$ & $x_1x_5+x_2x_3+x_4x_5$  \\
\texttt{vii} & $x_1x_2-x_3x_4$ &  $x_1x_4+x_2x_4+x_3x_5$\\
\texttt{viii} & $x_1x_4-(x_2-x_3)x_5$  & $(x_1+x_4)(x_2+x_3)+x_2x_3$\\
\texttt{ix} & $x_1x_2-x_3^2$ & $x_3^2-x_4x_5$\\
\texttt{x} & $x_1x_2-x_3^2$ & $x_2x_3-x_4x_5$  \\
\texttt{xi} & $x_1x_4-x_3x_5$ & $x_1x_2+x_2x_4+x_3^2$ \\
\texttt{xii} & $x_1x_2-x_3x_4$ & $x_1x_5+x_2x_3+x_4^2$\\
\texttt{xiii} & $x_1x_4-x_2x_5$ & $x_1x_2+x_2x_4+x_3^2$  \\
\texttt{xiv} & $x_1x_2-x_3^2$ & $x_1^2-x_4x_5$  \\
\texttt{xv} & $x_1x_2-x_3^2$  &$x_1x_5+x_2x_3+x_4^2$  \\
\hline
\end{tabular}
\end{center}
\caption{Split surfaces representing the $15$ singularity types}
\label{t:class2'}
\end{table}

\begin{ex}
Calculate the Segre symbol for each of the surfaces in 
Table~\ref{t:class2'}, and check they match up with the correct
singularity type in Table~\ref{t:class2}.
\end{ex}

Whereas they share the same singularity type,
the surfaces of type \texttt{vii} and 
\texttt{viii} differ because in the former there are $5$
lines, $4$ of which pass through the singularity, whereas in the
latter all $4$ lines pass through the singularity.
Similarly, an important difference between the surfaces
of type \texttt{ii} and \texttt{iii} in Tables~\ref{t:class2} and
\ref{t:class2'} is that for the surface of type \texttt{ii}, the line
joining the two singularities is contained in the surface, whereas for
the surface of type \texttt{iii} it is not. 
When the $2$ singular points are defined over a quadratic
extension of $\Q$, the latter surface is called an {\em Iskovskih
surface}. There is ample evidence available (see Coray and Tsfasman
\cite{c-t}, for example) to the effect that
Iskovskih surfaces are the most arithmetically interesting surfaces
among the singular del Pezzo surfaces of degree $4$. In fact they are
the only such surfaces for which the Hasse principle can fail to hold.
The main focus of these lecture notes is upon the situation for split
singular del Pezzo surfaces, and so we will say no more about
Iskovskih surfaces here.

As usual, let $\tS$ denote the minimal desingularisation of any surface $S$ from
Table~\ref{t:class2'}. Then the Picard group of $\tS$ has rank 
$\rho_S=6$.  The goal recorded above is to try and establish \eqref{c1} for each
$S$.  As in the case of singular cubic surfaces several of the surfaces are actually special cases of
varieties for which the Manin conjecture is already known to hold.  
Thus it can be shown that the surfaces representing types \texttt{ix}, \texttt{x}, \texttt{xiv} are
all toric, so that \eqref{c1} already holds in these cases by the work
of Batyrev and Tschinkel \cite{MR1620682}.  In a very real 
sense these surfaces are the ``easiest'' to deal with in our list.

\begin{ex} 
Show that $\nub=O_{\ve}(B^{1+\ve})$ for the surfaces of
type \texttt{ix}, \texttt{x} and \texttt{xiv}.
\end{ex}

It has also been shown by  Chambert-Loir and Tschinkel \cite{ct} that the Manin conjecture
is true for equivariant compactifications of the algebraic group $\mathbb{G}_a^2$.
Although identifying such surfaces in the table is not entirely routine,
it transpires that the type \texttt{xv} surface (with a $\Dfive$
singularity) is covered by this work. 
In joint work with la Bret\`eche, the author \cite{dp4-d5} has
provided an independent proof of the Manin conjecture for this
particular surface. In addition to obtaining a finer asymptotic
formula of the shape given in \eqref{c3}, this work has provided a useful line of attack for several
other singular del Pezzo surfaces.

In Table \ref{t:class2''} we have recorded
a list of progress towards the final resolution of the Manin
conjecture for the split singular del Pezzo surfaces of degree $4$. We
have included the relevant
reference in the literature, and whether the result attained amounts
to an asymptotic formula for the counting function, or an upper
bound. We will not pay attention here to the quality of the error
term in the asymptotic formula, but each upper bound
is of the correct order of magnitude $B(\log B)^5$.
There is still plenty left to do!

\begin{table}[!ht]
\begin{center}
\begin{tabular}{|c|c|}
\hline
type & type of estimate achieved  \\
\hline
\hline
\texttt{v} & upper bound \cite{gauss}\\
\texttt{ix} & asymptotic formula \cite{MR1620682}\\
\texttt{x} & asymptotic formula \cite{MR1620682}\\
\texttt{xiii}  & asymptotic formula \cite{der+tsch}\\
\texttt{xiv}  & asymptotic formula \cite{MR1620682}\\
\texttt{xv} & asymptotic formula \cite{dp4-d5}\\
\hline
\end{tabular}
\end{center}
\caption{Summary of progress for the split singular del Pezzo surfaces
  of degree $4$}
\label{t:class2''}
\end{table}

It is also interesting to try and establish the Manin conjecture for
singular del Pezzo surfaces of degree $4$ that are not split over the
ground field.  In further joint work of the author with la Bret\`eche
\cite{dp4-d4}, the Manin conjecture is established for the  surface 
$$
x_1x_2-x_3^2=0,\quad x_1^2+x_2x_5+x_4^2 = 0.
$$
This surface has a $\Dfour$ singularity and is isomorphic over $\Q(i)$ to the 
surface of type \texttt{xiii} in Table \ref{t:class2'}.
The Picard group of $\tS$ has rank $4$ in this case, and an asymptotic
formula of the shape \eqref{c3} is obtained 
for any $\delta \in (0,3/32)$, with $P$ a polynomial of degree $3$.

\subsection{Degree $\geq 5$ surfaces}\label{s:dp5}

It turns out that all del Pezzo surfaces of degree $d\geq 7$ are toric
\cite[Proposition 8]{der1}, and that all non-singular del Pezzo
surfaces of degree $d\geq 6$ are toric. Thus \eqref{c1} already holds
in these cases by the work of Batyrev and Tschinkel \cite{MR1620682}. 
For non-singular del Pezzo surfaces $S \subset \bfP^5$ of degree~$5$,
the situation is rather less satisfactory.  In fact there are very few
instances for which the Manin conjecture has been
established.  The most significant of these is due to la Bret\`eche
\cite{MR2003m:14033}, who has proved the conjecture for the split
non-singular del Pezzo surface $S$ of degree $5$, in which the $10$
lines are all defined over $\Q$.  To be precise, 
if $U\subset S$ denotes the open subset formed
by deleting the lines from $S$, then 
la Bret\`eche shows that
$$
N_{U}(B) = c_{0}B (\log B)^{4}\Big(1+
O\Big(\frac{1}{\log\log B}\Big)\Big),
$$
for a certain constant $c_0>0$. This confirms Conjecture
\ref{c:manin-ns}, since we have seen in \eqref{eq:gPic} that
$\Pic (S)\cong \Z^5$ for split non-singular del Pezzo surfaces of
degree $5$.  The other 
major achievement in the setting of quintic del Pezzo surfaces is a
result of la Bret\`eche and Fouvry \cite{MR2099200}, where the 
Manin conjecture is established for a surface that is not split, but
contains lines defined over $\Q(i)$.

So far we have only discussed the situation for non-singular
del Pezzo surfaces of degree $d\geq 5$. Let us now turn to the
singular setting. When $d=6$ it emerges that there exist such surfaces
that are not toric, and so are not covered by \cite{MR1620682}. 
We will focus attention on the situation for del
Pezzo surfaces of degree $6$, following the investigation of Derenthal \cite{der1},
where the degree $5$ surfaces are also considered.
In view of \cite[Proposition 8.3]{c-t}, Table~\ref{t:class_deg6}
lists all possible types of singular del Pezzo surfaces
of degree $6$.

\begin{table}[!ht]
\begin{center}
\begin{tabular}{|c|c|c|}
\hline
type &  $\#$ lines & singularity \\
\hline
\hline
\texttt{i} &  $4$ & $\textbf{A}_1$\\
\texttt{ii} & $3$ & $\textbf{A}_1$\\
\texttt{iii} & $2$ & $2\textbf{A}_1$\\
\texttt{iv} & $2$ & $\textbf{A}_2$ \\
\texttt{v} & $1$ & $\textbf{A}_1+\textbf{A}_2$\\
\hline
\end{tabular}
\end{center}
\caption{Classification (over $\overline{\Q}$) of singular del Pezzo
surfaces of degree $6$}
\label{t:class_deg6}
\end{table}

As noted in \cite[\S 5]{der1}, the surfaces of type \texttt{i},
\texttt{iii} and \texttt{v} are all toric and so do not interest
us here. 
Any singular del Pezzo surface of degree $6$ can be realised as the
intersection of $9$ quadrics in $\bfP^6$.  For example, the type
\texttt{iv} surface is cut out by the system of equations
\begin{equation}\label{eq:a2}
\begin{split}
x_1x_6-x_4x_5&=x_1x_7-x_2x_5=x_1x_7-x_3x_4=x_3x_7+x_4x_5+x_5^2\\
&=x_5x_7-x_3x_4=x_2x_7+x_4^2+x_4x_5=x_4x_7-x_2x_6\\
&=x_4x_6+x_5x_6+x_7^2=x_2x_3-x_1x_4+x_1x_5=0.
\end{split}
\end{equation}
In this set of lecture notes we will establish the Manin
conjecture for the type \texttt{ii} surface, which has the simplest
possible singularity.
When $S\subset \bfP^6$ is a split surface of type \texttt{ii}, then
there is unimodular change of variables that takes $S$ into the
surface with equations 
\begin{equation}\label{eq:a1}
\begin{split}
x_1^2-x_2x_4&=x_1x_5-x_3x_4=x_1x_3-x_2x_5=x_1x_6-x_3x_5\\
&=x_2x_6-x_3^2=x_4x_6-x_5^2=x_1^2+x_1x_4+x_5x_7\\
&=x_1x_2+x_1^2+x_3x_7=x_1x_3+x_1x_5+x_6x_7=0.
\end{split}
\end{equation}
Let $\tS$ denote the minimal desingularisation of $S$. It follows
from \eqref{eq:gPic} that $\Pic(\tS)\cong \Z^4$ since $S$ is
split.
We will establish the following result in \S \ref{s:a1}.

\begin{thm}\label{t:a1}
Let $S\subset \bfP^6$ be the $\textbf{A}_1$ surface given by
\eqref{eq:a1}. Then there exist constants $c_1,c_2\geq 0$ such that 
$$
N_{U}(B)=c_1B (\log B)^3 +c_2B(\log B)^2  +O\big(B\log B \big),
$$
where
$$
c_1=\frac{\sigma_\infty}{144}\prod_p \Big(1-\frac{1}{p}  \Big)^4\Big(1+\frac{4}{p}+\frac{1}{p^2}\Big)
$$
and 
\begin{equation}
  \label{eq:a1-siginf}
\sigma_\infty=6\int_{\{u,t,v\in\R : ~0<u, ut^2, uv^2, |tv(t-v)|\leq 1\}}\d t\d u\d v.
\end{equation}
\end{thm}

Since $\Pic(\tS)$ has rank $4$, the exponents of $B$ and $\log B$ in
this asymptotic formula are in complete agreement with Conjecture
\ref{c:manin-ns}. Although we will not give details here, it turns
out that the value of the constant $c_1$ also 
confirms the prediction of Peyre \cite{MR1340296} in this case. It is hoped that our
proof of Theorem \ref{t:a1} will encourage other researchers to try
their hand at proving asymptotic formulae for $\nub$. With this in
mind Exercise \ref{ex:rp} is more of a research problem, and its
resolution will therefore conclude the proof of the Manin conjecture
for all split (non-singular {\em or} singular) del Pezzo surfaces of
degree 6.

\begin{ex}\label{ex:rp}
Establish an asymptotic formula for the type \texttt{iv} surface in
Table \ref{t:class_deg6}, with underlying equations \eqref{eq:a2}.
\end{ex}

\subsection{Universal torsors}\label{s:ut}

Universal torsors were originally introduced 
by Colliot-Th\'el\`ene and Sansuc \cite{ct1,ct2} 
to aid in the study of the Hasse principle and weak approximation for
rational varieties.  Since their inception it is now well-recognised that
they also have a central r\^ole to play in proofs of the Manin
conjecture for Fano varieties, and in particular, for del Pezzo surfaces.
Let $S \subset \bfP^d$ be a del Pezzo
surface of degree $d\in \{3,4,5,6\}$, and let $\tS$ denote the minimal desingularisation
of $S$ if it is singular, and $\tS=S$ otherwise.  Let $E_1,\ldots,E_{10-d} \in
\Div(\tS)$ be generators for $\Pic_{\cQ}(\tS)$, and
let $E_i^\times=E_i\setminus\{\mbox{zero section}\}$.  
Working over $\overline{\Q}$,  a {\em universal torsor} above $\tS$ is given by the
action of $\mathbb{G}_m^{10-d}$ on the map
$$
\pi: E_1^\times \times_\tS \cdots \times_\tS E_{10-d}^\times
\rightarrow \tS.
$$
A proper discussion of universal torsors would
take us too far afield, and the reader may
consult the survey of Peyre \cite{MR2029862} for further details, or indeed the
construction of Hassett and Tschinkel \cite{MR2029868}. The latter outlines
an alternative approach to universal torsors via the Cox ring.
Given the usual open subset $U\subset S$, the general theory of universal torsors 
ensures that there is a partition of $U(\Q)$
into a disjoint union of patches, each of which is in
bijection with a suitable set of integral points on 
a universal torsor above $\tS$.

The guiding principle behind the use of universal torsors is simply that they ought to be
arithmetically simpler than the original variety. In our work it will
suffice to think of universal torsors as ``particularly nice
parametrisations'' of rational points on the surface.
The universal torsors that we encounter in these lecture notes 
all have embeddings as affine hypersurfaces of high dimension. 
Moreover, in each case we will show how the underlying equation of the universal torsor
can be deduced in a completely elementary fashion, without
any recourse to geometry whatsoever.
The torsor equations we will meet all take the shape
$$
A+B+C=0, 
$$
for monomials $A,B,C$ of various degrees in the
appropriate variables.   As in many examples of counting problems for
higher dimensional varieties, one can occasionally gain leverage by
fixing some of the variables at the outset, in order to be left with a
counting problem for a family of small dimensional varieties.
If one is sufficiently clever about which variables to fix first, one is
sometimes left with a quantity 
that we know how to estimate --- and crucially --- whose error term we can control
once summed over the remaining variables.

As a concrete example, we note that Hassett and Tschinkel
\cite{MR2029868} have calculated the universal torsor for 
the cubic surface \eqref{eq:e6}.
It is shown that there is a unique universal torsor above
$\tS_1$, and that it is given by the equation
$$
y_\ell s_\ell^3 s_4^2 s_5+y_2^2s_2+y_1^3s_1^2s_3 = 0,
$$
for variables
$y_1,y_2,y_\ell,s_1,s_2,s_3,s_\ell,s_4,s_5,s_6$.
One of the variables does not explicitly appear in this equation, and
the torsor should be thought of as being embedded in $\A^{10}$.
It turns out that the way to proceed here is to fix all of the
variables apart from $y_1,y_2,y_\ell$.
One may then  view the  equation as a congruence
$$
y_2^2s_2\equiv -y_1^3s_1^2s_3 ~~\mod{s_\ell^3s_4^2s_5},
$$
in order to take care of the summation over $y_\ell$.
This is the approach taken in \cite{e6}, the next step being to employ
very standard facts about the number of integer solutions to 
polynomial congruences that are restricted to lie in certain regions.
One if left with a main term and an error term, which the
remaining variables need to be summed over.  While the treatment of
the main term is relatively routine, the treatment of the error term presents a much
more serious obstacle.

The universal torsors that turn up in the proofs of Theorems
\ref{main-d4} and \ref{t:a1} can also be embedded in affine space as
hypersurfaces.  We will see in \S \ref{s:a1} that the approach
discussed above also produces results for 
the del Pezzo surface of degree $6$ considered in Theorem
\ref{t:a1}.   In the proof of Theorem  \ref{main-d4} in \S
\ref{s:d4_cubic} our approach will
be more obviously geometric, and we will actually view the equation as a
family of projective lines, and also as a family of conics. We will
then call upon techniques from the 
geometry of numbers to count the relevant solutions.

\section{Further heuristics}\label{s:further}

We have seen in \S \ref{s:conjectures}, and in particular in the
statement of Conjecture~\ref{c:manin-ns}, that for any del Pezzo surface
$S\subset\bfP^d$ one expects a growth rate like $c_S B(\log B)^A$
for the counting function $\nub$. We have already given some
motivation for the exponent of $B$ in \eqref{eq:heuristic}. 
The focus of the present section is to produce much more sophisticated heuristics than we have
previously met. In particular we will gain an insight into the
exponent of $\log B$ that appears in the Manin conjecture.

For ease of presentation we restrict attention to
non-singular diagonal cubic surfaces $S\subset\bfP^3$. Thus
\begin{equation}
  \label{eq:111.1}
S=\{a_1x_1^3+a_2x_2^3+a_3x_3^3+a_4x_4^3=0\},  
\end{equation}
for $\ma{a}=(a_1,\ldots,a_4)\in\N^4$ such that
$\hcf(a_1,\ldots,a_4)=1$.  Define 
\begin{equation}
  \label{eq:P}
\mcal{P}:=\{3\}\cup\{p: p\mid a_1a_2a_3a_4\}.  
\end{equation}
This is the set of primes $p$ for which the reduction of $S$
modulo $p$  is singular.
Before passing to the arithmetic of diagonal cubic surfaces, we will
need to discuss some of their geometry.

\subsection{The lines on a cubic surface}\label{s:lines}

The facts that we will need in this section are explained in 
detail in the books of Hartshorne
\cite{hart} and Manin \cite{manin-book}.  In general, a non-singular cubic surface $S\subset
\bfP^3$ is obtained
by blowing up $\bfP^2$ along a collection of $6$ points
$P_1,\ldots,P_6$ in general position. By {\em general position} we mean that
no $3$ of them are collinear and they do not all lie on a conic.
The $27$ lines on the surface arise in the following way. There are
$6$ exceptional divisors $E_i$ above $P_i$, for $1 \leq i \leq 6$, and
$15$ strict transforms $L_{i,j}$ of the lines going through precisely
$2$ points $P_i,P_j$, for $1 \leq i<j \leq 6$. Finally there are the $6$ strict
transforms $Q_i$ of the conics going through all but one of the $6$
points.  If $\Lambda$ is the strict transform of a line in
$\bfP^2$ that doesn't go through any of the $P_i$, then a basis of
the geometric Picard group $\Pic_{\cQ}(S)$ is given by
$$
[\Lambda], [E_1], \ldots, [E_6].  
$$
The remaining divisors may be
expressed in terms of these elements via the relations
\begin{equation}
  \label{eq:111.3}
[L_{i,j}]= [\Lambda]-[E_i]-[E_j], \quad
[Q_{i}]=2[\Lambda]-\sum_{j\neq i}[E_j].
\end{equation}
The class of the anti-canonical divisor $-K_S$ is given by 
$
[-K_S]=3[\Lambda]-\sum_{i=1}^6[E_j],
$
although we will not need this fact in our work. One can check that
the hyperplane section has class 
$-3[\Lambda]+\sum_{i=1}^6[E_j]$ in $\Pic_{\cQ}(S)$, so that the cubic
surface has very ample anticanonical divisor, as claimed in \S \ref{s:conjectures}.

When $S$ takes the shape \eqref{eq:111.1} it is not hard to write down the $27$ lines 
explicitly. The calculations that we present below are based on those
carried out by Peyre and Tschinkel \cite[\S 2]{p-t2}.
Fix a cubic root $\alpha$ (resp. $\alpha',\alpha''$)
of $a_2/a_1$ (resp. $a_3/a_1$, $a_4/a_1$). We will assume that $\al\in\Q$
if $a_2/a_1$ (resp. $a_3/a_1$, $a_4/a_1$) is a cube in $\Q$.
Put 
$$
\beta=\frac{\alpha''}{\alpha'},\quad 
\beta'=\frac{\alpha}{\alpha''},\quad
\beta''=\frac{\alpha'}{\alpha}.
$$
We denote by $\theta$ a primitive cube root of one.
Let $i$ run over elements of $\Z/3\Z$.
Then the  $27$ lines on the cubic surface \eqref{eq:111.1}
are given by the equations
\[
\arraycolsep 0pt
\begin{array}{rclrclrcl}
L_i&:&\begin{cases}
x_1{+}\theta^i\alpha x_2=0,\\
x_3{+}\theta^i\beta x_4=0,
\end{cases}&
L_i'&:&\begin{cases}
x_1{+}\theta^i\alpha x_2=0,\\
x_3{+}\theta^{i{+}1}\beta x_4=0,
\end{cases}&
\hspace{-0.1cm}
L_i''&:&\begin{cases}
x_1{+}\theta^i\alpha x_2=0,\\
x_3{+}\theta^{i{+}2}\beta x_4=0,
\end{cases}\\
\noalign{\vskip0.75ex\penalty 1000}
M_i&:&\begin{cases}
x_1{+}\theta^i\alpha' x_3=0,\\
x_4{+}\theta^{i}\beta' x_2=0,
\end{cases}&
M_i'&:&\begin{cases}
x_1{+}\theta^i\alpha' x_3=0,\\
x_4{+}\theta^{i{+}1}\beta' x_2=0,
\end{cases}&
\hspace{-0.1cm}
M_i''&:&\begin{cases}
x_1{+}\theta^i\alpha' x_3=0,\\
x_4{+}\theta^{i{+}2}\beta' x_2=0,
\end{cases}\\
\noalign{\vskip0.75ex\penalty 1000}
N_i&:&\begin{cases}
x_1{+}\theta^i\alpha'' x_4=0,\\
x_2{+}\theta^{i}\beta'' x_3=0,
\end{cases}&
N_i'&:&\begin{cases}
x_1{+}\theta^i\alpha'' x_4=0,\\
x_2{+}\theta^{i{+}1}\beta'' x_3=0,
\end{cases}&
\hspace{-0.1cm}
N_i''&:&\begin{cases}
x_1{+}\theta^i\alpha'' x_4=0,\\
x_2{+}\theta^{i{+}2}\beta'' x_3=0.
\end{cases}
\end{array}
\]
Let $K=\Q(\theta,\alpha,\alpha',\alpha'').$
It is a Galois extension of $\Q$, and in the generic case
has degree $54$ with Galois group $\Gal(K/\Q)\cong (\Z/3\Z)^3 
\rtimes \Z/2\Z.$

We need to equate these lines to the divisors $E_i,L_{i,j},Q_i$
that we met earlier. There is a certain degree of freedom in doing
this, as discussed in \cite[\S V.4]{hart}, but it turns out that the choice 
\begin{equation}
\label{eq:111.2}
\begin{aligned}
E_1&=L_0,&E_2&=L_1,&E_3&=L_2,\\
E_4&=M_1,&E_5&=M'_2,&E_6&=M''_0,\\
Q_1&=L'_1,&Q_2&=L'_2,&Q_3&=L'_0,\\
Q_4&=M_0,\quad&Q_5&=M'_1,\quad&Q_6&=M''_2,\\
L_{1,2}&=L''_1,&L_{2,3}&=L''_2,&L_{1,3}&=L''_0,\\
L_{4,5}&=M''_1,&L_{5,6}&=M_2,&L_{4,6}&=M'_0,\\
L_{1,4}&=N_0,&L_{1,5}&=N_1,&L_{1,6}&=N_2,\\
L_{2,4}&=N'_1,&L_{2,5}&=N'_2,&L_{2,6}&=N'_0,\\
L_{3,4}&=N_2'',&L_{3,5}&=N''_0,&L_{3,6}&=N_1'',
\end{aligned}
\end{equation}
is satisfactory. 
In assigning lines to $E_1,\ldots,E_6$, all that is required is that
they should all be mutually skew. 
Given \emph{any} cubic surface of the shape \eqref{eq:111.1}, 
we now have the tools with which to compute the Picard group
\eqref{eq:pg}. In fact, from
this point forwards  the process requires little more than basic linear algebra.

Let us illustrate the procedure by calculating the Picard group for
a special case.  Consider 
the Fermat surface 
\begin{equation}
  \label{eq:S1}
S_1=\{x_1^3+x_2^3+x_3^3+x_4^3=0\}.
\end{equation}
In this case $\al=\al'=\al''=\be=\be'=\be''=1$, in the notation above,
and $K=\Q(\theta)$ is a quadratic field extension. 
We wish to find elements of the geometric Picard group
$\Pic_{\cQ}(S_1)$ that are fixed by the action of $\Gal(K/\Q)$. 
Thus we want vectors $\ma{c}=(c_0,\ldots,c_6)\in\Z^7$ such that
\begin{equation}
  \label{eq:fixed}
(c_0[\Lambda]+c_1[E_1]+\cdots+c_6[E_6])^\sigma = 
c_0[\Lambda]+c_1[E_1]+\cdots+c_6[E_6], 
\end{equation}
for every $\sigma\in\Gal(K/\Q)$.
Under the action of $\Gal(K/\Q)$ it is not hard to check that $\Lambda$ and $E_1$ are fixed,
that $E_2$ and $E_3$ are swapped, and that $E_4$ (resp. $E_5$, $E_6$) is taken to 
$L_{5,6}$ (resp. $L_{4,5}$, $L_{4,6}$).   Using 
\eqref{eq:111.3} one sees that the left hand side of \eqref{eq:fixed} is
equal to 
\begin{align*}
(c_0+c_4+c_5+c_6)[\Lambda]+&c_1[E_1]+c_2[E_3]+c_3[E_2]\\
&\quad
-(c_5+c_6)[E_4]
-(c_4+c_5)[E_5]-(c_4+c_6)[E_6],
\end{align*}
in $\Pic_{\cQ}(S_1)$. Thus we are interested in the space 
of $\ma{c}\in\Z^7$ for which
$$
c_4+c_5+c_6=0,\quad c_2-c_3=0,\quad c_4+2c_5=0,\quad c_4+2c_6=0.
$$
This  system of homogeneous linear equations in $7$ variables has
underlying matrix of rank $3$. Thus the space of solutions has rank
$7-3=4$, and so we may conclude that $\Pic(S_1)\cong\Z^4$. In fact a little thought
reveals that the $4$ elements 
$$
[\Lambda],  [E_1],  [E_2]+[E_3], -2[E_4]+[E_5]+[E_6],
$$
provide a basis for $\Pic(S_1)$.

Next, consider the surface
\begin{equation}
  \label{eq:111.5}
S_2=\{x_1^3+x_2^3+x_3^3+px_4^3=0\}
\end{equation}
for a prime $p$. When $p=2$ or $3$, the arithmetic of this surface has been considered in
some detail by Heath-Brown \cite{hb:density}, who provides some
numerical evidence to the effect that the corresponding counting
function $\nub$ should grow like $c_pB$ for a certain constant $c_p>0$. 

\begin{ex}
Show that  $\Pic(S_2)\cong\Z$. 
\end{ex}

This calculation has also been carried out by Colliot-Th\'el\`ene, 
Kanevsky and Sansuc \cite[p. 12]{ctks}. More generally, it is known
that the Picard group of the surface \eqref{eq:111.1} has rank $1$ if
and only if the ratio
$$
\frac{a_{\sigma(1)}a_{\sigma(2)}}{a_{\sigma(3)}a_{\sigma(4)}}
$$
is not a cube in $\Q$, for each permutation $\sigma$
of $(1,2,3,4)$.  This result is due to Segre \cite{segre}.

\subsection{Cubic characters and Jacobi sums}

Throughout this section let $p$ be a prime. 
Recall that  a (multiplicative) character on $\F_p=\Z/p\Z$ is a map
$\chi:\F_p^*\rightarrow \C^*$ such that
$$
\chi(ab)=\chi(a)\chi(b)
$$
for all $a,b\in\F_p^*$. The {\em trivial character} $\ve$ is defined by the
relation  $\ve(a)=1$ for all $a\in\F_p^*$. 
It is convenient to extend the domain of definition to all of $\F_p$
by assigning  $\chi(0)=0$ if $\chi\neq \ve$ and $\ve(0)=1$.

We begin by collecting together a few basic facts, all of
which are established in \cite[\S 8]{ir}.

\begin{lem}\label{lem:characters}
Let $p$ be a prime. Then the following hold:
\begin{enumerate}
\item 
Let $\chi$ be a character on $\F_p$ and let 
$a\in\F_p^*$. Then $\chi(1)=1$, $\chi(a)$ is a $(p-1)$-th root of
unity, and 
$
\chi(a^{-1})=\overline{\chi}(a)=\chi(a)^{-1}.
$
\item For any character $\chi$ on $\F_p$ we have 
$$
\sum_{a\in\F_p} \chi(a)=\left\{
\begin{array}{ll}
0,&\mbox{if $\chi\neq \ve$,}\\
p,&\mbox{if $\chi= \ve$.}
\end{array}
\right.
$$
\item The set of characters on $\F_p$ forms a cyclic group of order $p-1$.
\end{enumerate}
\end{lem}

It follows from part (3) of Lemma \ref{lem:characters} that
$\chi^{p-1}=\ve$ for any character on $\F_p$. We 
define the {\em order} of a character to be the least positive
integer $n$ such that $\chi^n=\ve$. In our work we will mainly be concerned
with the characters of order $3$. 
Let us turn briefly to the topic of generalised Jacobi sums. Given any characters
$\chi_1,\ldots,\chi_r$ on $\F_p$, a {\em Jacobi sum} is a sum of the shape
$$
J_0(\chi_1,\ldots,\chi_r):=\sum_{\colt{\ma{t}=(t_1,\ldots,t_r)\in \F_p^r
}{t_1+\cdots+t_r \equiv 0\bmod{p}}} \chi_1(t_1)
\cdots \chi_r(t_r). 
$$
The key fact that we will need concerning these sums is that 
\begin{equation}
  \label{eq:jo}
  |J_0(\chi_1,\ldots,\chi_r)|=\left\{
\begin{array}{ll}
0,& \mbox{if $\chi_1\cdots \chi_r \neq \ve$,}\\
(p-1)p^{r/2-1},& \mbox{if $\chi_1\cdots \chi_r =\ve$.}
\end{array}
\right.
\end{equation}
This is established in \cite[\S 8.5]{ir}.

Let $p$ be a rational prime. We proceed to consider $p$ as an element
of the ring of integers $\Z[\theta]$ associated to the quadratic field $\Q(\theta)$
obtained by adjoining a primitive cube root of unity $\theta$. It follows from
basic algebraic number theory that $p$ is a prime in
$\Z[\theta]$ if $p\equiv 2 \bmod{3}$, whereas it splits as
$p=\pi\overline{\pi}$ if $p\equiv 1 \bmod{3}$, where $\pi$ is a
prime in $\Z[\theta]$.  
When $p\equiv 2 \bmod{3}$ the only cubic character on $\F_p$ is the   
trivial character $\ve$. On the other hand, when $p=\pi\overline{\pi}\equiv 1\bmod{3}$ then there
are precisely two non-trivial cubic characters
$\chi_\pi,\chi_{\overline{\pi}}$ on $\F_p$,
where
$$
\chi_\omega(\cdot) = \Big(\frac{\cdot}{\omega}\Big)_3
$$
is the cubic residue symbol for any prime $\omega$ in $\Z[\theta]$.
All of these facts are established in \cite[\S 9]{ir}.

It turns out that Jacobi sums can be used to give 
formulae for the number of solutions to appropriate equations over
finite fields. 
Given any $q\in\N$, let 
\begin{equation}
  \label{eq:Nq}
N(q):=\#\{\x \bmod q: a_1x_1^3+\cdots+ 
a_4x_4^3\equiv 0 \pmod{q}\},
\end{equation}
and
\begin{equation}
  \label{eq:Nq*}
N^*(q):=\#\Big\{\x \bmod q: 
\begin{array}{l}
a_1x_1^3+\cdots+ 
a_4x_4^3\equiv 0 \pmod{q},\\
\hcf(q,x_1,\ldots,x_4)=1
\end{array}
\Big\}.
\end{equation}
When $p$ is a prime not belonging to the finite set of primes $\mcal{P}$
defined in \eqref{eq:P}, we can write down a very precise expression for $N(p)$. 
Thus it follows from \cite[\S 8.7]{ir} that
$$
N(p)=p^3+ \sum_{\chi_1,\chi_2,\chi_3,\chi_4}
{\chi_1}(a_1^{-1})
{\chi_2}(a_2^{-1})
{\chi_3}(a_3^{-1})
{\chi_4}(a_4^{-1})  J_0(\chi_1,\chi_2,\chi_3,\chi_4)
$$
where the summation is over all non-trivial cubic characters $\chi_i:
\F_p^*\rightarrow \C$ such that $\chi_1\chi_2\chi_3\chi_4=\ve$.

\begin{ex}\label{ex:jo-sgn}
Let $p\not\in\mcal{P}$ be a prime, and let $\chi_1,\chi_2,\chi_3,\chi_4$ be non-trivial cubic characters on
$\F_p$ such that $\chi_1\chi_2\chi_3\chi_4=\ve$.  
Deduce from \eqref{eq:jo} that
$$
J_0(\chi_1,\chi_2,\chi_3,\chi_4)=p(p-1).
$$
\end{ex}

It follows from Exercise \ref{ex:jo-sgn} that
\begin{equation}
  \label{eq:N*}
  N^*(p) = N(p)-1 = p^3+p(p-1)\delta_p(\ma{a})-1.
\end{equation}
for any prime $p\not\in\mcal{P}$,
where
$$
\delta_p(\ma{a}):=
\sum_{\chi_1,\chi_2,\chi_3,\chi_4}
{\chi_1}(a_1^{-1})
{\chi_2}(a_2^{-1})
{\chi_3}(a_3^{-1})
{\chi_4}(a_4^{-1}). 
$$
Let $a\in\F_p^*$ and suppose that $p$ splits as
$\pi\overline{\pi}$. Then it will be useful to observe that
\begin{equation}
  \label{eq:split}
\chi_\pi(a)+\chi_{\overline{\pi}}(a)=\left\{
\begin{array}{ll}
2, &\mbox{if $a$ is a cubic residue modulo $\pi$,}\\
-1, & \mbox{otherwise}.
\end{array}
\right.
\end{equation}
We have $\delta_p(\ma{a})=0$ when $p\equiv 2 \bmod{3}$, since there
are then no non-trivial cubic characters
modulo  $p$. When $p\equiv 1 \bmod{3}$, with $p=\pi\overline{\pi}\not\in\mcal{P}$, we have
\begin{align*}
\delta_p(\ma{a})
=&
\chi_\pi\Big(\frac{a_1a_2}{a_3a_4}\Big)+
\chi_{\overline{\pi}}\Big(\frac{a_1a_2}{a_3a_4}\Big)
+
\chi_\pi\Big(\frac{a_1a_3}{a_2a_4}\Big)+
\chi_{\overline{\pi}}\Big(\frac{a_1a_3}{a_2a_4}\Big)\\
&\qquad \qquad \qquad \qquad \qquad \qquad +
\chi_\pi\Big(\frac{a_1a_4}{a_2a_3}\Big)+
\chi_{\overline{\pi}}\Big(\frac{a_1a_4}{a_2a_3}\Big).
\end{align*}
Still with this choice of prime $p$, let 
$\nu_p(\ma{a})$ denote the number of indices $i\in \{2,3,4\}$ for which 
the cubic character $\chi_{\pi}(\frac{a_1a_i}{a_ja_k})$ is equal to
$1$, with $\{i,j,k\}$ a permutation of $\{2,3,4\}$. 
Then we may deduce from \eqref{eq:split} that
\begin{equation}
  \label{eq:dell}
\delta_p(\ma{a})=
\left\{
\begin{array}{ll}
0, & \mbox{if $p\equiv 2 \bmod{3}$},\\
3\nu_p(\ma{a})-3, & \mbox{if $p\equiv 1 \bmod{3}$},
\end{array}
\right.
\end{equation}
when $p\not\in\mcal{P}$.

\subsection{The Hardy--Littlewood circle method}\label{s:HL}

We are now ready to consider the counting function
$\nub$ that is associated to the diagonal cubic surface $S\subset \bfP^3$
given  in \eqref{eq:111.1}. Our aim is to provide heuristic evidence in support of 
Manin's original conjecture, and we will say rather little
about the predicted value of the constant. 
The Hardy--Littlewood circle method is an extremely effective means of
estimating counting functions associated to projective algebraic
varieties, but it only works when the dimension of
the variety is substantially larger than the degree. We have already
seen evidence of this in the statement of Theorem \ref{t:birch}, which is based on an
application of the circle method. 
Although it has not been made to produce asymptotic formulae for the
counting functions associated to del Pezzo surfaces, in this section
we will see how the
Hardy--Littlewood method can still be used as a useful heuristic tool.
The key idea is to consider only the contribution from the 
major arcs.

We will simplify matters by applying the heuristic to count all of
the rational points on $S$, rather than restricting attention to the
open subset $U$. Although the details are formidable, it is in fact
possible to obtain upper bounds for $\nub$ using the circle
method. Thus Heath-Brown \cite{hb-3/2} has shown that
$\nub=O_{\ve,S}(B^{3/2+\ve})$ under a certain hypothesis concerning
the Hasse--Weil $L$-function associated to the surface.
An interesting feature of this work is that the contribution from the rational points lying on
rational lines in the surface is successfully
separated out.  When the surface contains no lines defined over $\Q$, such as the surface
given  by \eqref{eq:111.5} for example, one obviously has
$$
\nub=N_{S}(B)+O(1).
$$
When $S$ contains lines defined over $\Q$ there is a general consensus
among people working on the circle method that the
dominant contribution (ie. the contribution from the points on
rational lines) should come from the minor arc integral.

In what follows let $e(z):=e^{2\pi i z}$ for any $z\in\R$. As usual, 
$Z^4$ denotes the set of primitive vectors in $\Z^4$.
The igniting spark in the Hardy--Littlewood circle method is the
simple identity
$$
\int_0^1 e(\al n)\d \al = \left\{
\begin{array}{ll}
1,& \mbox{if $n=0$,}\\
0,& \mbox{if $n\in\Z\setminus\{0\}$.}
\end{array}
\right.
$$
On taking into account the fact that $\x$ and $-\x$ represent the same
point in projective space, and applying Exercise \ref{e:mob}, we deduce that
\begin{align*}
N_S(B)
&=\frac{1}{2}\int_0^1 \sum_{\colt{\x\in Z^4}{|\x|\leq B}} e(\al(a_1x_1^3+\cdots
+a_4x_4^3))\d\alpha\\
&=\frac{1}{2}\sum_{k=1}^\infty \mu(k)\int_0^{1} \sum_{\colt{\x\in \Z^4}{|\x|\leq B/k}} e(\al(a_1x_1^3+\cdots
+a_4x_4^3))\d\alpha\\
&=\frac{1}{2}
\sum_{k=1}^\infty \mu(k)
\int_0^{1} S(\al)\d\alpha,
\end{align*}
say. 
Let us write $P=B/k$ and 
$I_{\mcal{A}}(P):=\int_\mcal{A} S(\al)\d\al$,
for any bounded subset $\mcal{A}\subset \R$.
The cubic exponential sum $S(\al)$ can actually be rather large when
$\al$ is well-approximated by a rational number with small
denominator. For example, we clearly have 
$S(0)=2^4P^4+O(P^3)$.
The philosophy that underpins the Hardy--Littlewood method is that
one expects $S(\al)$ to be small for values of $\al\in[0,1]$ that are
not well-approximated by rational numbers with small denominator. This
is notoriously difficult to prove in general, and as indicated above, 
is expected to be false in the present setting!

Our heuristic will be based on analysing $I_{\mathfrak{M}}(P)$ for a
suitable choice of ``major arcs'' $\mathfrak{M}$. We will not give
full details here, the gaps being easily filled by consulting
the relevant techniques in Davenport \cite{dav}.
Let $\ve>0$ be a
small parameter.  Given $a,q\in\Z$ such that 
\begin{equation}
  \label{eq:111.6}
1\leq a\leq q\leq P^\ve, \quad \hcf(a,q)=1,
\end{equation}
we define the interval
$$
\mathfrak{M}(a,q):=\Big[\frac{a}{q}-P^{-3+\ve}, ~\frac{a}{q}+P^{-3+\ve}\Big].
$$
We take as major arcs the union
$$
\mathfrak{M}:=\bigcup_{q\leq P^\ve}\bigcup_{\colt{1\leq a\leq
    q}{\hcf(a,q)=1}} \mathfrak{M}(a,q).
$$
It is clear that $\mathfrak{M}$ contains all the points in
the interval $[0,1]$ that are well-approximated by rational numbers
with small denominator.

\begin{ex}
Show that $\mathfrak{M}$ is a disjoint union for $\ve<1$.
\end{ex}

The ``minor arcs'' are defined to be
$\mathfrak{m}:=[0,1]\setminus \mathfrak{M}$, and we will proceed under the
assumption that the minor arc integral
$I_{\mathfrak{m}}(P)$ can be ignored. Actually we will also ignore the
contribution that this term makes once it is summed up over values of $k$. 
In truth there will be several points in the argument where we will simply {\em ignore} subsidiary
contributions. We will indicate all of these by an appearance of the word
``error''.   Thus, to begin with, we have
\begin{align*}
N_U(B)
=\frac{1}{2}\sum_{k=1}^\infty \mu(k)
I_{\mathfrak{M}}(P) +\err.
\end{align*}
Here we have made the further assumption that the contribution
$N_{S\setminus U}(B)$ from the points lying on lines in $S$
arises in the minor arc integral.

Let $a,q\in\Z$ such that \eqref{eq:111.6} holds, and let
$\al=a/q+z\in\mathfrak{M}(a,q)$. Let $C(\x)$ denote the
diagonal cubic form in \eqref{eq:111.1}. We now break the sum into
congruence classes modulo $q$, giving
\begin{equation}
  \label{eq:161.1}
S(a/q+z)
=\sum_{\ma{r}\bmod{q}}
e(aC(\ma{r})/q)
\sum_{\colt{\x\in \Z^4\cap[-P,P]^4}{\ma{x}\equiv \ma{r}\bmod{q}}}
e(zC(\ma{x})).
\end{equation}
We would like to replace the discrete variable $\ma{x}$ by a continuous
one in the inner sum, and the summation over $\ma{x}$ by an integral.  
For this we will appeal to the following general result.

\begin{lem}\label{E-M}
Let $P \geq 1$, let $\ma{a}\in \Z^n$ and let $r\in \N$ such that
$r\leq P$.
Let $F$ be a function on $\R^n$ all of whose first
order partial derivatives exist and are continuous on $\mcal{R}:=[-P,P]^n$.
Define
$$
M_F:=\sup_{\x\in\mcal{R}} \max_{1\leq i\leq n} \Big|\frac{\partial
  F}{\partial x_i}(\x)\Big|.
$$
Then we have
$$
\sum_{\colt{\x\in \Z^n\cap\mcal{R}}{\ma{x}\equiv \ma{a}\bmod{r}}}
e(F(\ma{x})) = \frac{1}{r^n}\int_{\mcal{R}} e(F(\ma{t})) \d \ma{t} +
O\Big(\frac{P^{n-1}(1+ PM_F)}{r^{n-1}}\Big).
$$
\end{lem}
\begin{proof}
Our proof of Lemma \ref{E-M} is based on the Euler--Maclaurin summation
formula \cite[\S I.0]{ten}. Let $B_k(x)$ denote the $k$th Bernoulli
polynomial, for $k \in \Z_{\geq 0}$, and let $s\in \Z_{\geq 0}$. Let
$A,B\in\Z$, with $A<B$. 
For any function $f:\R \rightarrow \C$ whose $(s+1)$-th derivative $f^{(s+1)}$ exists
and is continuous on the interval $[A,B]$, 
the Euler--Maclaurin summation formula states that
\begin{equation}\label{eq:EM}
\begin{split}
\sum_{A<n \leq B} f(n)=
\int_A^B f(t)\d t &+\sum_{k=0}^{s}
\frac{(-1)^{k+1}B_{k+1}(0)}{(k+1)!}\big(f^{(k)}(B)-f^{(k)}(A)\big) \\
&+\frac{(-1)^s}{(s+1)!} \int_A^B B_{s+1}(t)f^{(s+1)}(t)\d t.
\end{split}
\end{equation}
Let $a\in\Z$ and $r\in\N$. We will apply this result with
$f_0(x)=f(a+rx)$ and 
$$
A_0=\frac{A-a}{r},\quad B_0=\frac{B-a}{r}.
$$
Taking $s=0$ in the Euler--Maclaurin formula, we therefore deduce that
\begin{equation}\label{eq:em-q}
\sum_{\colt{A<n \leq B}{n\equiv a\bmod{r}}} f(n)=
\frac{1}{r}\int_A^B f(t)\d t -\frac{f(B_0)-f(A_0)}{2}+
\int_A^B B_{1}\Big(\frac{t-a}{r}\Big)f'(t)\d t,
\end{equation}
since $B_1(x)=x-[x]-\frac{1}{2}$.

We are now ready to establish Lemma \ref{E-M}, which we will do by
induction on $n$.  Write $\mcal{S}_n$ for the $n$-dimensional sum that is to
be estimated. 
The case $n=1$ of Lemma \ref{E-M} follows from \eqref{eq:em-q}
with $f(x)=e(F(x))$. Assuming now that $n\geq 2$, we have
$$
\mcal{S}_n=
\sum_{\colt{y\in\Z\cap[-P,P]}{y\equiv a_1 \bmod{r}}}
\sum_{x_2,\ldots,x_{n}}
e( G(x_2,\ldots,x_{n})),
$$
where $G(x_2,\ldots,x_n)=F(y,x_2,\ldots,x_{n})$,
and the sum over $x_2,\ldots,x_{n}$ is over all integers in
$[-P,P]$ such that $x_i\equiv a_i\bmod{r}$ for $2\leq i\leq n.$
We may employ the induction hypothesis to estimate the inner sum in
$n-1$ variables. It therefore follows that
\begin{align*}
\mcal{S}_n=
\sum_{\colt{y\in\Z\cap[-P,P]}{y\equiv a_1 \bmod{r}}}
\Big( \frac{1}{r^{n-1}}\int_{[-P,P]^{n-1}} &e(G(t_2,\ldots,t_n))\d
t_2\cdots \d t_n
\\ 
& +
O\Big(\frac{P^{(n-2)}(1+ PM_G)}{r^{n-2}}\Big)\Big).
\end{align*}
Now there are $O(P/r)$ integers $y$ in the interval $[-P,P]$ that are
congruent to $a_1$ modulo $r$, since $r\leq P$ by
assumption. Moreover, it is clear that $M_G\leq
M_F$. Hence 
\begin{align*}
\mcal{S}_n=
\frac{1}{r^{n-1}}\sum_{\colt{y\in\Z\cap[-P,P]}{y\equiv a_1 \bmod{r}}}
f(y)+
O\Big(\frac{P^{(n-1)}(1+ PM_F)}{r^{n-1}}\Big),
\end{align*}
where 
$$
f(y)=\int_{[-P,P]^{n-1}} e(F(y,t_2,\ldots,t_n))\d t_2\cdots \d t_n.
$$
The statement of  Lemma \ref{E-M} is now an easy consequence of \eqref{eq:em-q}.
\end{proof}

It is clear from the proof of Lemma  \ref{E-M} that when $F$ has
partial derivatives to a higher order, one may obtain a much sharper
estimate by including higher order terms in the Euler--Maclaurin
summation formula.  The present bound is satisfactory for our
purposes, however.

Returning to \eqref{eq:161.1} we apply Lemma \ref{E-M} with
$$
F(\x)=zC(\x),\quad n=4, \quad \ma{a}=\ma{r},\quad r=q.
$$
In particular we have $q\leq P^{\ve}\leq P$, as required for the lemma. Furthermore, 
$M_F=M_{zC}\ll |z|P^2 \leq P^{-1+\ve}$ for any
$\al=a/q+z\in\mathfrak{M}(a,q)$. It follows that
\begin{align*}
S(a/q+z)
&=
q^{-4}T(a,q)V_P(z) +O(P^{3+2\ve})
\end{align*}
on the major arcs, where
\begin{equation}
  \label{eq:111.8}
T(a,q):=\sum_{\ma{r}\bmod{q}} e(aC(\ma{r})/q), \quad
V_R(z):=\int_{[-R,R]^4} e(zC(\x))\d\x.
\end{equation}
The set of major arcs has $\meas(\mathfrak{M})=O(P^{-3+3\ve})$.  
On carrying out the integration over $z$ and the summation over $a$
and $q$ one is therefore led to the conclusion that
\begin{align*}
I_{\mathfrak{M}}(P)&=
P^4
\sum_{q\leq P^\ve} \sum_{\colt{1\leq a\leq
    q}{\hcf(a,q)=1}} q^{-4}T(a,q)\int_{|z|\leq {P^{-3+\ve}}} 
V_1(zP^3)\d z +O(P^{5\ve})\\
&=P
\sum_{q\leq P^\ve} \sum_{\colt{1\leq a\leq
    q}{\hcf(a,q)=1}} q^{-4}T(a,q)
\int_{|z|\leq P^{\ve}} 
V_1(z)\d z +O(P^{5\ve}).
\end{align*}
Define
$$
\mathfrak{I}(R):=\int_{|z|\leq R} 
V_1(z)\d z= 
\int_{|z|\leq R} \int_{[-1,1]^4} e(zC(\x))\d\x \d z.
$$
It can be shown that $\mathfrak{I}(R)$ is a bounded function of
$R$, and furthermore,
$\mathfrak{I}(R)\rightarrow \mathfrak{I}_0>0$ as $R\rightarrow \infty$. A standard
calculation reveals that the limit $\mathfrak{I}_0$ is equal to $2\sigma_{\infty}$, where
$$
\sigma_{\infty}:=\frac{1}{6a_4^{1/3}}\int \frac{\d
  x_1\d x_2 \d x_3}{(a_1x_1^3+a_2x_2^3+a_3x_3^3)^{2/3}}
$$
is the real density of solutions. 
Here the integral is over $x_1,x_2,x_3 \in [-1,1]$ such that
$|(a_1x_1^3+a_2x_2^3+a_3x_3^3)/a_4|\leq 1$.
We now define 
$$
\mathfrak{S}(R):= \sum_{q\leq R} q^{-4} \sum_{\colt{1\leq a\leq
    q}{\hcf(a,q)=1}} T(a,q).
$$
On bringing everything together, our investigation has so far succeeded in showing that
\begin{equation}
  \label{eq:121.1}
N_U(B)= \sigma_{\infty}\sum_{k=1}^\infty\mu(k)
 P \mathfrak{S}(P^\ve) +\err,
\end{equation}
for a suitably small value of $\ve>0,$  where $P=B/k$.

Our task is now to examine the sum $\mathfrak{S}(R)$, as $R\rightarrow
\infty$. Let us write
$$
S_q:=\sum_{\colt{1\leq a\leq q}{\hcf(a,q)=1}} T(a,q),
$$
where $T(a,q)$ is the complete exponential sum defined in \eqref{eq:111.8}.
Then $\mathfrak{S}(R)=\sum_{q\leq R}q^{-4}S_q$. 
It turns out that $S_q$ is a multiplicative function of $q$. 
This can be established along the lines of
\cite[Lemma 5.1]{dav}.  
Recall the definition \eqref{eq:Nq} of $N(q)$. 
We now come to the key relation between $S_{q}$ and $N(q)$ at prime power
values of $q$.

\begin{ex}\label{ex:S-N}
Let $p$ be a prime and let $e\geq 1$. Use Lemma \ref{lem:characters}
to show that
$$
S_{p^e}= p^eN(p^e)-p^{3+e}N(p^{e-1}).
$$
\end{ex}

Let us for the moment ignore considerations of convergence, and
consider the local factors $\sum_{e=0}^\infty p^{-4e}S_{p^e}$ in the infinite
product formula for $\mathfrak{S}(\infty)$. Now it follows from Exercise \ref{ex:S-N} that
\begin{align*}
\sum_{e=0}^E p^{-4e} S_{p^e}&=
1+\sum_{e=1}^E 
\big(p^{-3e}N(p^e)-p^{3-3e}N(p^{e-1})\big)
=
p^{-3E}N(p^{E}),
\end{align*}
for any $E\geq 1$.
Hence, formally speaking, we have
$
\mathfrak{S}(\infty)= \prod_p \tau_p,
$
where
$$
\tau_p:=\lim_{e\rightarrow \infty} p^{-3e}N(p^e).
$$
If $\mathfrak{S}(R)$ was convergent, which it certainly is
\emph{not} in general, 
we could then conclude from \eqref{eq:121.1} that
\begin{equation}
  \label{eq:222.1}
N_U(B)= B\sigma_{\infty}\sum_{k=1}^\infty \frac{\mu(k)}{k}\prod_p\tau_p +\err.
\end{equation}
Arguing formally, we now replace the summation over $k$ by its Euler
product, concluding that
\begin{equation}
  \label{eq:222.2}
N_U(B)= B\sigma_{\infty}\prod_p\sigma_p +\err,
\end{equation}
with $\sigma_p:=(1-1/p)\tau_p$.  Note that
\begin{equation}
  \label{eq:sigp*}
\sigma_p=\lim_{e\rightarrow \infty} p^{-3e}N^*(p^e),
\end{equation}
in the notation of \eqref{eq:Nq*}, which clearly follows 
from the observation that 
$$
N^*(p^e)=N(p^e)-8N(p^{e-3}),
$$
for any $e>3$. The estimate in \eqref{eq:222.2} therefore 
gives a heuristic asymptotic formula for $N_U(B)$ 
which is visibly a product of local densities.  
Among other things, we
have assumed that $\mathfrak{S}(R)$ is convergent in formulating this
heuristic. This is expected to be true for diagonal cubic surfaces
whose Picard group has rank $1$, but not in general.

One can make the transition from \eqref{eq:222.1} to \eqref{eq:222.2} 
completely rigorous if the Hardy--Littlewood heuristic produces a leading
term involving $P^\alpha$ with exponent $\alpha>1$. The outcome is  that to
go from counting points on the affine cone to counting projective points, one
merely replaces $N(p^e)$ by $N^*(p^e)$. For $\alpha=1$, however,  the
``renormalization'' procedure remains heuristic.  
Let
$$
\mathfrak{S}^*(R):=\sum_{q\leq R}q^{-4}S_q^*,
$$
where 
$$
S_q^*:=\sum_{\colt{1\leq a\leq q}{\hcf(a,q)=1}} 
\sum_{\colt{\ma{r}\bmod{q}}{\hcf(\ma{r},q)=1}} e(aC(\ma{r})/q).
$$
It is easily checked that $S_q^*$ is a multiplicative function of $q$, and
furthermore, that the corresponding version of Exercise \ref{ex:S-N} holds,
relating $S_{p^e}^*$ to $N^*(p^e)$. In fact, formally speaking, one
has
$$
\prod_p \sum_{e=0}^\infty p^{-4e}S_{p^e}^*= \prod_p \sigma_p,
$$
with $\sigma_p$ given by \eqref{eq:sigp*}.  
Bearing all of this in mind we will proceed under the bold assumption that 
 \eqref{eq:121.1} can be replaced by 
\begin{equation}  \label{eq:222.4}
N_U(B)= B\sigma_{\infty}\mathfrak{S}^*(B) +\err,
\end{equation}
where we have taken $\ve=1$ in the expressions for
$\mathfrak{S}(B^\ve)$ and $\mathfrak{S}^*(B^\ve)$.

We now turn to  a finer analysis of
$\mathfrak{S}^*(B)$, as $B\rightarrow \infty.$ Our task is to
determine the analytic properties of the corresponding Dirichlet series
\begin{equation}
  \label{eq:DS}
F(s):=\sum_{q=1}^\infty \frac{S_q^*}{q^{s}}
\end{equation}
for $s=\sigma+it\in\C$. Armed with this analysis we will ultimately
apply Perron's formula to obtain
an estimate for $\mathfrak{S}^*(R)$.
Using the multiplicativity of $q^{-s}S_q^*$ we deduce that
\begin{equation}
  \label{eq:Fs}
F(s)=\prod_p \sigma_p(s),\quad \sigma_p(s):=\sum_{e=0}^\infty p^{-es}S_{p^e}^*.
\end{equation}
In examining $F(s)$  it clearly
suffices to ignore the value of the factors $\sigma_p(s)$ at any finite collection
of primes $p$. With this in mind we will try and determine $\sigma_p(s)$ for
$p\not\in\mcal{P}$, where $\mcal{P}$ is given by \eqref{eq:P}.
Recall the definition \eqref{eq:Nq*} of $N^*(q)$.

\begin{ex}\label{ex:hensel}
Let $e\geq 1$ and let $p\not\in\mcal{P}$ be a prime. Use Hensel's lemma 
to show that $N^*(p^e)=p^{3e-3}N^*(p)$.
\end{ex}

It therefore follows from  Exercise \ref{ex:hensel} that
\begin{align*}
\sigma_p(s)&=
1+  \sum_{e=1}^\infty p^{-es}
\big(p^eN^*(p^e)-p^{3+e}N^*(p^{e-1})\big)
=
1- \frac{1}{p^{s-4}}+\frac{N^*(p)}{p^{s-1}},
\end{align*}
for any $p\not\in\mcal{P}$. Hence
\eqref{eq:N*} yields 
\begin{align*}
\sigma_p(s)
&=
1+ \frac{\delta_p(\ma{a})}{p^{s-3}}-\frac{\delta_p(\ma{a})}{p^{s-2}}-\frac{1}{p^{s-1}},
\end{align*}
where $\delta_p(\ma{a})$ is given by \eqref{eq:dell}.

\medskip

We now pursue our analysis in the special case $\ma{a}=(1,1,1,1)$ of 
the Fermat cubic surface \eqref{eq:S1}. 
Now it is clear from \eqref{eq:dell} that 
$\delta_p(1,1,1,1)=0$ if $p\equiv 2 \bmod{3}$ and 
$$
\delta_p(1,1,1,1)=3\nu_p(1,1,1,1)-3=6
$$
if $p\equiv 1 \bmod{3}$. 
Let $\lambda: \Z\rightarrow \C$ be the real Dirichlet
character of order $2$ defined by 
$$
\lambda(n):=\left\{
\begin{array}{ll}
(\frac{n}{3}), &\mbox{if $3\nmid n$,}\\
0, &\mbox{otherwise,}
\end{array}
\right.
$$
where $(\frac{n}{3})$ is the Legendre symbol.
Then we may write 
\begin{align*}
\sigma_p(s)
&= 
1+\frac{3(1+\lambda(p))}{p^{s-3}}-\frac{3(1+\lambda(p))}{p^{s-2}}-\frac{1}{p^{s-1}}\\
&=\Big(1-\frac{1}{p^{s-3}}\Big)^{-3}
\Big(1-\frac{\lambda(p)}{p^{s-3}}\Big)^{-3}
\Big(1+ O\Big( \frac{1}{p^{\min\{\sigma-2, 2\sigma-2\}}}\Big)\Big),
\end{align*}
for any $p\not\in\mcal{P}$.  Let $L(s,\lambda)$ denote the usual 
Dirichlet $L$-function associated to $\lambda$. When
$\ma{a}=(1,1,1,1)$ we have therefore succeeded in showing that
\begin{equation}
  \label{eq:181.1}
F(s)=\zeta(s-3)^{3}L(s-3,\lambda)^3G(s),
\end{equation}
where  $G(s)$ is a function that is holomorphic and bounded on the half-plane
$\sigma\geq 7/2+\delta$, for any $\delta>0$. 
For future reference we note that $G(4)$ has local factors
\begin{equation}
  \label{eq:G4}
G_p(4)=\left\{
\begin{array}{ll}
(1-\frac{1}{p})^7
(1+\frac{7}{p}+\frac{1}{p^{2}}),
&
\mbox{if $p\equiv 1 \bmod{3}$,}\\
(1-\frac{1}{p})^4(1+\frac{1}{p})^3(1+\frac{1}{p}+\frac{1}{p^2}),
&
\mbox{if $p\equiv 2 \bmod{3}$.}
\end{array}
\right.
\end{equation}
Although we will not prove it here, it can be deduced from \eqref{eq:sigp*}
and Hensel's lemma that
\begin{equation}
  \label{eq:G4.3}
G_3(4)=\frac{8}{27}\lim_{e\rightarrow \infty} 3^{-3e}N^*(3^e)=\frac{16}{27}.
\end{equation}
We are now ready for our application of Perron's formula, which we
will apply in the following form.

\begin{lem}\label{lem:perron}
Let $F(s)=\sum_{n=1}^{\infty}a_nn^{-s}$ be a Dirichlet series with
abscissa of absolute convergence $\sigma_a$.
Suppose that $x\not\in\Z$ and let $c>\sigma_a$. Then
we have
$$
\sum_{n\leq x}a_n =\frac{1}{2\pi i} \int_{c-iT}^{c+iT}
F(s)\frac{x^s}{s}\d s +O\Big(\frac{x^c}{T}\sum_{n=1}^\infty
\frac{|a_n|n^{-c}}{|\log (x/n)|}\Big),
$$
for any $T\geq 1$.
\end{lem}

\begin{proof}
Let $c>0$. The lemma follows from the identity
$$
\frac{1}{2\pi i} 
\int_{c-iT}^{c+iT}
\frac{x^s}{s}\d s=\left\{
\begin{array}{ll}
1+ O(x^c(T |\log x|)^{-1}), & \mbox{if $x>1$,}\\
\frac{1}{2}+ O(cT^{-1}), & \mbox{if $x=1$,}\\
O(x^c(T |\log x|)^{-1}), & \mbox{if $0<x<1$,}
\end{array}
\right.
$$
which is a straightforward exercise in contour integration.
\end{proof}

In our case we have $a_q=S_q^*$ and we are interested in the
Dirichlet series $F(s+4)$, in the notation of \eqref{eq:DS}. 
In order to apply Lemma \ref{lem:perron} we will need an upper
bound for this quantity.  For our purposes the trivial upper bound
$a_q\ll q^5$ is sufficient.
Thus the Dirichlet series $F(s+4)$ is absolutely 
convergent for $\sigma> 2$. Taking $c=2+\ve$ for any $\ve>0$, we
may deduce from Lemma \ref{lem:perron} that
$$
\mathfrak{S}^*(B)=\frac{1}{2\pi i} \int_{c-iT}^{c+iT}
F(s+4)\frac{B^s}{s}\d s +O\Big(\frac{B^{c}}{T}\sum_{n=1}^\infty
\frac{1}{n^{1+\ve}|\log (B/n)|}\Big),
$$
for any $T\geq 1$ and any $B\not\in\Z$.  It is not hard to see that
the error term here is 
$ \ll T^{-1} B^{c}$.
We now apply Cauchy's residue theorem to the rectangular contour $\mcal{C}$ 
joining  ${c'-iT}$, ${c'+iT}$,
${c+iT}$ and ${c-iT}$, where $c'=-1/2+\ve$.
The relation \eqref{eq:181.1} implies that in this
region $F(s+4)B^s/s$ has a unique pole at $s=0$, and it is a pole of
order $4$. It has residue 
$$
\Res_{s=0}\frac{F(s+4)B^s}{s}= \frac{L(1,\lambda)^3G(4)P(\log B)}{3!},
$$
where $P\in\R[x]$ is a monic polynomial of degree $3$. 
Putting all of this together we have therefore shown that
\begin{equation}
  \label{eq:SR1}
\mathfrak{S}^*(B)=
\frac{L(1,\lambda)^3G(4)P(\log B)}{3!} +O(E(B)),
\end{equation}
where 
$$
E(B)
= \frac{B^{c}}{T}
+
\Big(\int_{c'-iT}^{c'+iT}+\int_{c'-iT}^{c-iT}+ \int_{c+iT}^{c'+iT}\Big)
\Big|H(s)^3\frac{B^s}{s}\Big|\d s,
$$
for any $T\geq 1$, and where
$H(s)=\zeta(s+1)L(s+1,\lambda)$.  Here we have used the fact that $G(s+4)$ is 
bounded on the half-plane $\Re e(s)\geq c'$.

To make the analysis simpler, it will be convenient to proceed under
the assumption that the Lindel\"of hypothesis holds for $\zeta(s)$,
and also for the Dirichlet $L$-function $L(s,\lambda)$. This could be
avoided at the cost of extra effort, but there seems no harm in
supposing it here. Thus we may assume the bounds
$$
\zeta(\sigma+i t)\ll_\ve |t|^{\varepsilon},
\quad
L(\sigma+i t,\lambda)\ll_\ve
|t|^{\varepsilon},
$$
for any $\sigma\in[1/2,1]$ and any $|t|\geq 1$.  It therefore 
follows that
$$
  H(\sigma+it)\ll_\varepsilon \left\{
\begin{array}{ll}
|t|^{\varepsilon}, &\mbox{if $-1/2\leq \sigma\leq 0$,}\\ 
1, &\mbox{if $\sigma>0$,}
\end{array}
\right.
$$
for any $|t|\geq 1$, which gives
\begin{align*}
\int_{c'-iT}^{c-iT}\Big|H(s)^3\frac{B^s}{s}\Big|\d s
&\ll_\ve 
\int_{c'}^{c}B^\sigma T^{-1+3\ve}\d \sigma
\ll_\ve B^c T^{-1+3\ve}.
\end{align*}
One obtains the same estimate for the contribution from the remaining
horizontal contour. Turning to vertical integral, we find that
\begin{align*}
\int_{c'-iT}^{c'+iT}
\Big|H(s)^3\frac{B^s}{s}\Big|\d s 
&\ll
B^{c'} \int_{-T}^{T} \frac{|H(1/2+\ve+it)|^3}{1+|t|}  \d t \\
&\ll_\ve B^{c'} \int_{-T}^{T} (1+|t|)^{3\ve-1}  \d t \\
&\ll_\ve B^{c'}T^{3\ve},
\end{align*}
under the assumption of Lindel\"of hypothesis.
This shows that
$$
E(B)
\ll_\ve B^{\ve}T^{3\ve}\Big(\frac{B^{2}}{T}+\frac{1}{B^{1/2}}\Big),
$$
for any $T\geq 1$. Taking $T$ sufficiently large, we therefore conclude 
from \eqref{eq:SR1} that
$$
\mathfrak{S}^*(B)=
\frac{L(1,\lambda)^3G(4)P(\log B)}{3!} +O(B^{-\D}),
$$
for some $\Delta>0$.

We are now ready to return to the 
Hardy--Littlewood major arc analysis which led us to
\eqref{eq:222.4}. Substituting in our estimate for
$\mathfrak{S}^*(B)$, we conclude that
\begin{equation}
  \label{eq:fermat-a}
N_{U_1}(B)\sim 
c_1B (\log B)^3
\end{equation}
where $U_1\subset S_1$ is the usual open subset of the Fermat surface
\eqref{eq:S1}, and
$$
c_1= \frac{\sigma_{\infty} L(1,\lambda)^3G(4)}{3!}.
$$
Now it follows from the class number formula that
$
L(1,\lambda)=\pi \sqrt{3}/9.
$
Hence, on combining this with \eqref{eq:G4} and \eqref{eq:G4.3}, our
heuristic argument has led us to the expectation that
\eqref{eq:fermat-a} holds, with
$$
c_1=\frac{\sigma_\infty 2^{4} \pi^3\sqrt{3}}{3! 3^{8}}
\prod_{p \equiv 1 \bmod{3}}
\hspace{-0.2cm}
\Big(1-\frac{1}{p}\Big)^7\Big(1+\frac{7}{p}+\frac{1}{p^{2}}\Big)
\hspace{-0.2cm}
\prod_{p \equiv 2 \bmod{3}}
\hspace{-0.2cm}
\Big(1-\frac{1}{p^3}\Big)\Big(1-\frac{1}{p^2}\Big)^3.
$$
The exponents of $B$ and $\log B$ in \eqref{eq:fermat-a} agree with
the Manin 
conjecture, since we have already seen in \S \ref{s:lines} that the
Picard group of $S_1$ has rank $4$.  

It is interesting to compare our analysis with the work of Peyre and
Tschinkel \cite{p-t2}, who calculate the 
leading constant $c_{\mathrm{Peyre}}$ in Peyre's refinement
\cite{MR1340296} of the conjectured asymptotic formula for $N_{U_1}(B)$.
It turns out that 
$$
c_{\mathrm{Peyre}}=\gamma(S_1) c_1,
$$ 
with $\gamma(S_1)=7/3$. For a general non-singular cubic surface
$S\subset \bfP^3$, the constant $\gamma(S)$
is defined to be the volume
$$
\gamma(S):=
\int_{\Lambda_{\mathrm{eff}}^\vee(S)} e^{-\langle-K_{S},\ma{t} \rangle}\d\ma{t}.
$$
Thus, in general terms,  $\gamma(S)$ measures the volume of the polytope obtained by
intersecting the dual of $\Lambda_{\mathrm{eff}}(S)$ with a certain affine hyperplane.
In particular $\gamma(S)\in \Q$ for any non-singular cubic surface
$S$, and $\gamma(S)=1$ if and only if the corresponding Picard group has rank $1$.


\begin{ex}
Let $S_2$ denote the surface \eqref{eq:111.5}.
Using a similar argument, show that one expects an asymptotic formula
of the shape $N_{S_2}(B)\sim c_2 B$ for some constant $c_2\geq 0$.
Check your answer with the heuristic formula obtained by Heath-Brown \cite{hb:density}.
\end{ex}

\section{The $\textbf{A}_1$ del Pezzo surface of degree $6$}\label{s:a1}

In this section we will establish Theorem \ref{t:a1}. Any line in 
$\bfP^6$ is defined by the intersection of $5$ hyperplanes.  
It is not hard to see that the equations
\begin{equation}
  \label{eq:a1_U}
\left\{
\begin{array}{l}
x_1=x_2=x_3=x_5=x_6=0, \quad 
x_1=x_3=x_4=x_5=x_6=0,\\
x_3=x_5=x_6=x_1+x_4=x_1+x_2=0,
\end{array}
\right.
\end{equation}
all define lines contained in the singular del Pezzo
surface $S$ given by \eqref{eq:a1}. Table \ref{t:class_deg6} ensures that these are
the only lines contained in $S$.
By definition $U$ is the open subset of $S$ on which none of these
equations hold.  We begin by establishing the following result.

\begin{lem}\label{lem:a1.step1}
We have
$$
\nub=2M(B) +O(B),
$$
where $M(B)$ denotes the number of 
$\x\in \Z^7$ such that
\begin{equation}\label{eq:a1'}
\begin{split}
x_1^2-x_2x_4&=x_1x_5-x_3x_4=x_1x_3-x_2x_5=x_1x_6-x_3x_5\\
&=x_2x_6-x_3^2=x_4x_6-x_5^2=x_1^2-x_1x_4+x_5x_7\\
&=x_1^2-x_1x_2-x_3x_7=x_1x_3-x_1x_5+x_6x_7=0,
\end{split}
\end{equation}
with $\hcf(x_1,\ldots,x_7)=1$, 
$0<|x_1|,x_2,x_3,x_4,|x_5|,x_6 \leq B$ and $|x_7|\leq B$.
\end{lem}

\begin{proof}
In view of the fact that $\x$ and $-\x$ represent the same point in
$\bfP^6$, we have
$$
\nub=\frac{1}{2}\#\{\x\in Z^7: |\x|\leq B, ~\mbox{\eqref{eq:a1} holds,
but \eqref{eq:a1_U} does not}\},
$$
where $Z^7$ denotes the set of primitive vectors in
$\Z^7$. We need to consider the contribution to the right hand side from
points such that $x_i=0$, for some $1\leq i\leq 7$. Let us begin by
considering the contribution from vectors $\x\in Z^7$ for which
$x_1=0$. But then the equations in \eqref{eq:a1} imply that
$x_2x_4=0$. If $x_2=0$, it is straightforward to check that either 
$\x$ satisfies the first system of
equations in \eqref{eq:a1_U}, or else 
$$
x_0=x_2=x_3=x_7=0, \quad x_4x_6=x_5^2.
$$
Such points are therefore confined to a plane conic. We therefore obtain $O(B)$
points overall with $x_1=x_2=0$. If on the other hand $x_1=x_4=0$,
then a similar analysis shows that there are $O(B)$ points in this
case too. In view of the first equation in \eqref{eq:a1}, the
contribution from vectors $\x$ such that $x_2x_4=0$ is also $O(B)$. 
Let us now consider the contribution from vectors $\x$ such that
$x_3=0$ and $x_1x_2x_4\neq 0$. 
It is easily checked that the only such vectors have $x_5=x_6=0$ and
$x_1+x_4=x_1+x_2=0$, and so must lie on a line contained in $S$. 
Finally, arguing in a similar fashion, we see that there are no points contained in
$S$ with $x_5x_6=0$ and $x_1x_2x_3x_4\neq 0$. 
We have therefore shown that
$$
\nub=\frac{1}{2}\#\{\x\in Z^7 : x_1\cdots x_6\neq 0, ~|\x|\leq B, ~\mbox{\eqref{eq:a1'}
  holds}\} +O(B).
$$
Here we have noted that there is an obvious unimodular transformation
that takes the set of equations in \eqref{eq:a1} into \eqref{eq:a1'}.

We would now like to restrict our attention to positive values of
$x_1,\ldots,x_6$. The equations for $S$ imply that $x_2,x_4,x_6$ all share the same
sign. On absorbing the minus sign into $x_1$ there is a clear
bijection between solutions to \eqref{eq:a1'} with 
$x_2,x_4,x_6<0$ and solutions with $x_2,x_4,x_6>0$.  
We  choose to count the former. Arguing similarly, 
by absorbing the minus signs into $x_1$ and $x_7$, we see that 
there is a bijection between the solutions to \eqref{eq:a1'} with
$x_3<0$ and $x_2,x_4,x_6>0$, and the
solutions with $x_2, x_3,x_4,x_6>0$.  Fixing our attention on the latter set of points,
we therefore complete the proof of Lemma \ref{lem:a1.step1}.
\end{proof}

Let $\tS$ denote the minimal desingularisation of the surface $S$. 
By determining the Cox ring associated to $\tS$,
Derenthal \cite{der1} has calculated the universal torsor
above $\tS$. In this setting it is defined by a single equation
\begin{equation}
  \label{eq:ut-a1}
  s_1y_1-  s_2y_2+  s_3y_3=0,
\end{equation}
embedded in $\A^{7}$. In particular one of the variables does not
appear explicitly in the equation.

\subsection{Elementary considerations}\label{s.301.3}

As promised in \S \ref{s:ut}, we proceed to show how $\nub$ can be related to a count of the integer
points on the corresponding universal torsor, which in this case is
given by \eqref{eq:ut-a1}. Our deduction of this
fact is completely elementary, and is based on an analysis of the integer
solutions to the system of equations \eqref{eq:a1'}. It is still somewhat
mysterious as to how or why this rather low-brow process should
ultimately lead to the same outcome! Typical of the facts that we
will employ is the following.

\begin{ex}\label{ex:301.1}
Show that the general solution of the equation $xy=z^2$ is
$$
x=a^2c, \quad y=b^2c,\quad z=abc,
$$
with $|\mu(c)|=1$.
\end{ex}

Given $s_0\in \R$ and $\ma{s}=(s_1,s_2,s_3),\ma{y}=(y_1,y_2,y_3)\in\R^3$, define 
\begin{equation}
  \label{eq:height-a1}
\Psi(s_0,\ma{s},\ma{y}):=\max\big\{
|s_0^3s_1^2s_2^2s_3^2|, ~|y_1y_2y_3|,~
|s_0s_1^2y_1^2|, ~|s_0s_2^2y_2^2|
\big\}.
\end{equation}
We are now ready to record our translation of the problem to the
universal torsor.

\begin{lem}\label{lem:301.2}
We have
$$
\nub=2
\#\left\{(s_0,\ma{s},\ma{y})\in \Z^7:
\begin{array}{l}
\Psi(s_0,\ma{s},\ma{y})\leq B, ~\mbox{\eqref{eq:ut-a1}  holds},\\ 
s_0,s_1,s_2,s_3,y_1>0 ,\\
\hcf(y_i,s_0s_js_k)=1,\\
\hcf(s_i,s_j)=1
\end{array}
\right\}+O(B),
$$
with $i,j,k$ a permutation of $1,2,3$ in the coprimality conditions.
\end{lem}

\begin{proof}
Let $\x\in \Z^7$ be a primitive vector counted by $M(B)$, as defined
in the statement of Lemma \ref{lem:a1.step1}. Combining the first
equation in \eqref{eq:a1'} with Exercise \ref{ex:301.1} we see that
$$
x_1=a_1a_2a_4,\quad x_2=a_2^2 a_1,\quad x_4=a_4^2a_1,
$$
for integers $a_1,a_2,a_4$ such that $a_1,a_2>0$ and 
$$
|\mu(a_1)|=\hcf(a_1\hcf(a_2,a_4)^2,x_3,x_5,x_6,x_7)=1.
$$
Inserting this into the equation $x_2x_6=x_3^2$ we deduce that
$a_1a_2\mid x_3$, whence 
$$
x_3=a_1a_2a_3, \quad x_6=a_1a_3^2, 
$$
for a positive integer $a_3$ such that 
$$
|\mu(a_1)|=\hcf(a_1\hcf(a_2,a_4)^2,a_1a_3\hcf(a_2,a_3),x_5,x_7)=1.
$$
Substituting this into the equation $x_4x_6=x_5^2$, we deduce
that 
$$
x_5=a_1a_3a_4,
$$
with 
\begin{equation}
  \label{eq:152}
|\mu(a_1)|=\hcf(a_1,x_7)=
\hcf(a_2,a_3,a_4,x_7)=1.
\end{equation}
Note that the second equation in \eqref{eq:a1'} implies that $x_1,x_5$
must share the same sign, which here is the sign of $a_4$.
The equations $x_1x_5=x_3x_4$, $x_1x_3=x_2x_5$ and
$x_1x_6=x_3x_5$ reveal no new information.
Turning instead to the equation
$x_1^2=x_1x_4-x_5x_7$, we obtain
\begin{equation}
  \label{eq:142.1}
a_1a_2^2a_4=a_1a_2a_4^2-a_3x_7.
\end{equation}
The coprimality conditions imply that
$a_1 \mid a_3$. Moreover, we deduce from this equation that $a_2a_4\mid a_3x_7/a_1$. We may therefore write
$$
a_2=a_{23}a_{27},\quad a_4=a_{43}a_{47},
$$
for integers $a_{2i},a_{4i}$, with $i=3,7$, such that
$a_{2i},a_{43},|a_{47}|>0$, and 
$$
a_1a_{23}a_{43}\mid a_3,\quad a_{27}a_{47}\mid x_7.
$$
Thus there exist further integers $b_3,a_7$ with $b_3>0$,  such that
$$
a_{3}=a_1a_{23}a_{43}b_3,\quad x_7=a_{27}a_{47}a_7,
$$
with \eqref{eq:152} and \eqref{eq:142.1} becoming
$$
|\mu(a_1)|=\hcf(a_1,a_{27}a_{47}a_7)=
\hcf(a_{23}a_{27},a_{23}a_{43}b_3,a_{43}a_{47},a_{27}a_{47}a_7)=1,
$$
and 
$$
a_{23}a_{27}=a_{43}a_{47}-b_3a_7,
$$
respectively.
The final two equations are redundant.
Let us write $d$ for the highest common factor of
$a_{23},a_{43},b_3$. Thus 
$$
a_{23}=da_{23}',\quad a_{43}=da_{43}',\quad b_{3}=db_{3}',
$$
for positive integers $d,a_{23}',a_{43}',b_3'$. On making these
substitutions the equation remains the same, but with appropriate
accents added, whereas the coprimality conditions become
\begin{align*}
|\mu(a_1)|=\hcf(da_1,a_{27}a_{47}a_7)
&=\hcf(a_{23}'a_{27},a_{23}'a_{43}'b_3',a_{43}'a_{47},a_{27}a_{47}a_7)\\
&=\hcf(a_{23}',a_{43}',b_{3}')\\
&=1.
\end{align*}

Now any $n\in\N$ can be written uniquely in the form $n=ab^2$ for
$a,b\in\N$ such that $|\mu(a)|=1$. We may therefore make the change of variables 
$$
(s_{0}; s_1,s_2,s_3;y_1,y_2,y_3)
=(a_1d^2;a_{23}',a_{43}',b_3'; a_{27},a_{47},a_7). 
$$
Bringing everything together, we have therefore established the 
existence of $(s_0,\ma{s},\ma{y})\in\Z^7$ such that
\eqref{eq:ut-a1} holds, with
\begin{equation}
  \label{eq:pos}
 s_{0},s_1,s_2, s_3,y_1>0,
\end{equation}
and 
$$
\hcf(s_{0},y_1y_2y_3)=\hcf(s_1,s_2,s_3)=
\hcf(s_1y_1,s_2y_2, s_1s_2s_3,y_1y_2y_3)=1.
$$
Note that $y_2$ is automatically non-zero for $s_0,\ma{s},\ma{y}$
satisfying the remaining conditions.
Once combined with \eqref{eq:ut-a1}, it is easy to check that the
latter coprimality conditions are equivalent to the conditions 
\begin{equation}\label{eq:coprim-a1}
\left\{
\begin{array}{l}
\hcf(y_1,s_0s_2s_3)=\hcf(y_2,s_0s_1s_3)=
\hcf(y_3,s_0s_1s_2)=1,\\
\hcf(s_1,s_2)=\hcf(s_1,s_3)=\hcf(s_2,s_3)=1,
\end{array}
\right.
\end{equation}
that appear in the statement of the lemma.

At this point we may summarise our argument as follows.
Let $\mcal{T} \subset \Z^7$ denote the set of
$(s_0,\ma{s},\ma{y}) \in \Z^7$  such that 
\eqref{eq:ut-a1}, \eqref{eq:pos} and \eqref{eq:coprim-a1} hold.
Then for any primitive vector $\x$ counted by $M(B)$,
we have shown that there exists
$(s_0,\ma{s},\ma{y}) \in \mcal{T}$ such that
$$
\left\{
\begin{array}{l}
x_1 = s_0s_1s_2y_1y_2,\\
x_2 = s_0s_1^2y_1^2, \\
x_3 = s_0^2s_1^2s_2s_3y_1,\\
x_4 = s_0s_2^2y_2^2,\\
x_5 = s_0^2s_1s_2^2s_3y_2,\\
x_6 = s_0^3s_1^2s_2^2s_3^2,\\
x_7 = y_1y_2y_3.
\end{array}
\right.
$$
Conversely, we leave it as an exercise to check that any
$(s_0,\ma{s},\ma{y}) \in \mcal{T}$ produces a primitive point
$\x\in\Z^7$ such that \eqref{eq:a1'} holds,
with 
$$
|x_1|,x_2,x_3,x_4,|x_5|,x_6>0.
$$
We may now conclude
that $M(B)$ is equal to the number of $(s_0,\ma{s},\ma{y}) \in
\mcal{T}$ such that 
$$
\max_{i=1,2}\big\{|s_0s_1s_2y_1y_2|, ~|s_0s_i^2y_i^2|, ~
|s_0^2s_1s_2s_3s_iy_i|,~|s_0^3s_1^2s_2^2s_3^2|, ~|y_1y_2y_3|\big\} \leq B.
$$
In view of the fact that
$|s_0s_1s_2y_1y_2|=\sqrt{|s_0s_1^2y_1^2|}\sqrt{|s_0s_2^2y_2^2|}$, and
furthermore, 
$|s_0^2s_1s_2s_3s_iy_i|=\sqrt{|s_0^3s_1^2s_2^2s_3^2|}\sqrt{|s_0s_i^2y_i^2|}$,
it follows that this height condition is equivalent to 
$\Psi(s_0,\ma{s},\ma{y})\leq B$
for any $(s_0,\ma{s},\ma{y}) \in \mcal{T}$,
where $\Psi$ is given by \eqref{eq:height-a1}.
In summary we have therefore shown that $M(B)$ is equal to the number of 
$(s_0,\ma{s},\ma{y}) \in \mcal{T}$ such that
$\Psi(s_0,\ma{s},\ma{y})\leq B$. 
Once inserted into Lemma~\ref{lem:a1.step1},
this completes the proof of Lemma~\ref{lem:301.2}.
\end{proof}

At first glance it might seem a little odd that 
the height restriction $|s_0s_3^2y_3^2|\leq B$ doesn't explicitly
appear in the lemma. However, \eqref{eq:ut-a1} implies that 
$0<s_1y_1=s_2y_2-s_3y_3$ for any $(s_0,\ma{s},\ma{y}) \in
\mcal{T}$, whence the restriction $\Psi(s_0,\ma{s},\ma{y})\leq B$ is 
plainly equivalent to 
$\max\{|s_0s_3^2y_3^2|, \Psi(s_0,\ma{s},\ma{y})\}\leq B$. 
We have preferred not to include it explicitly in the statement of
Lemma \ref{lem:301.2} however.

\subsection{The asymptotic formula}

Our starting point is  Lemma \ref{lem:301.2}. Let 
$T(B)$ denote the quantity on the right hand side that is to be
estimated. 
Once taken together with \eqref{eq:ut-a1}, the height condition $\Psi(s_0,\ma{s},\ma{y})\leq B$
is equivalent to
$$
\max_{i=1,2}\big\{
|s_0^3s_1^2s_2^2s_3^2|, |s_0s_i^2y_i^2|, ~|y_1y_2(s_1y_1-s_2y_2)/s_3|
\big\}\leq B.
$$
Define
$$
X_0:=\Big( \frac{s_0^3s_1^2s_2^2s_3^2}{B}\Big)^{1/3},\quad
X_i:=\Big( \frac{s_1s_2s_3B}{s_i^3}\Big)^{1/3},\quad
$$
for $i=1,2$. Then the height conditions above can be rewritten as
$$
|X_0|\leq 1, \quad |f_1(y_1)|\leq 1, \quad |f_2(y_2)|\leq 1, \quad
|g(y_1,y_2)|\leq 1,
$$
where
$$
f_i(y):=X_0\Big(\frac{y}{X_i}\Big)^2,\quad 
g(y_1,y_2):=\frac{y_1y_2}{X_1X_2}\Big(\frac{y_1}{X_1}-\frac{y_2}{X_2}\Big)
$$
for $i=1,2$.
In order to count solutions to the equation \eqref{eq:ut-a1}, our plan
will be to view the equation as a congruence
$$
s_1y_1-  s_2y_2\equiv 0 \pmod{s_3},
$$
which has the effect of automatically taking care of the summation
over $y_3$. In order to make this approach viable we will need to
first extract the coprimality conditions on the $y_3$ variable.

Define the set 
\begin{equation}
  \label{eq:calS}
\mcal{S}:=\{(s_0,\ma{s})\in\N^4: \hcf(s_i,s_j)=1, ~X_0\leq 1\},
\end{equation}
with $i,j$ generic indices from the set $\{1,2,3\}$.
We now apply M\"obius inversion,
as in Exercise \ref{e:mob}, in order to remove the coprimality condition
$\hcf(y_3,s_0s_1s_2)=1$. Thus we find that
$$
T(B)=\sum_{(s_0,\ma{s})\in\mcal{S}} \sum_{k_3\mid
  s_0s_1s_2}\mu(k_3)\#\left\{\y\in\Z^3: 
\begin{array}{l}
\hcf(y_1,s_0s_2s_3)=1,\\
\hcf(y_2,s_0s_1s_3)=1,\\
y_1>0,\\
s_1y_1-  s_2y_2+  k_3s_3y_3=0,\\
|f_i(y_i)|\leq 1, ~|g(y_1,y_2)|\leq 1
\end{array}
\right\}.
$$
Now it is clear that the summand vanishes unless $\hcf(k_3,s_1s_2)=1$.
Hence 
\begin{equation}
  \label{eq:S1-a1}
T(B)=\sum_{(s_0,\ma{s})\in\mcal{S}} \sum_{\colt{k_3\mid
  s_0}{\hcf(k_3,s_1s_2)=1}}
\mu(k_3)S_{k_3}(B),
\end{equation}
where 
$$
S_{k_3}(B):=
\#\left\{y_1,y_2\in \Z:
\begin{array}{l}
\hcf(y_1,s_0s_2s_3)=1,\\
\hcf(y_2,s_0s_1s_3)=1,\\
y_1>0,\\
s_1y_1\equiv s_2y_2\bmod{k_3s_3},\\
|f_i(y_i)|\leq 1, ~|g(y_1,y_2)|\leq 1
\end{array}
\right\}.
$$
Clearly $S_{k_3}(B)$ depends on the parameters $s_0$ and $\ma{s}$,
in addition to $k_3$ and $B$. 
We now turn to the estimation of $S_{k_3}(B)$, for which we need the
following basic result.

\begin{ex}\label{ex:interval}
Let $b\geq a$ and $q>0$. Show  that
$$
\#\{n\in\Z\cap(a,b]: n\equiv n_0\bmod{q}\}=\frac{b-a}{q}+O(1).
$$
\end{ex}

We will fix $y_2$ and apply Exercise \ref{ex:interval} to handle the
summation over $y_1$. Before this we must use M\"obius inversion to
remove the coprimality condition $\hcf(y_1,s_0s_2s_3)=1$ from the
summand. Thus we find that
$$
S_{k_3}(B)=
\sum_{k_1\mid s_0s_2s_3}
\mu(k_1)
\#\left\{y_1,y_2\in \Z:
\begin{array}{l}
\hcf(y_2,s_0s_1s_3)=1,\\
k_1s_1y_1\equiv s_2y_2\bmod{k_3s_3},\\
|f_1(k_1y_1)|\leq 1, ~|f_2(y_2)|\leq 1,\\ 
|g(k_1y_1,y_2)|\leq 1, y_1>0
\end{array}
\right\}.
$$
In view of the other coprimality conditions, the summand plainly
vanishes unless $\hcf(k_1,k_3s_3)=1$. We may therefore write 
$\rho\in\Z$ for the (unique) inverse of $k_1s_1$ modulo $k_3s_3$, whence
\begin{equation}
  \label{eq:S13-a1}
S_{k_3}(B)=
\sum_{\colt{k_1\mid s_0s_2}{\hcf(k_1,k_3s_3)=1}}
\mu(k_1)
S_{k_1,k_3}(B),
\end{equation}
with
\begin{align*}
S_{k_1,k_3}(B)
&:=
\sum_{\colt{y_2\in\Z: ~|f_2(y_2)|\leq 1}{\hcf(y_2,s_0s_1s_3)=1}}
\#\left\{y_1\in \N:
\begin{array}{l}
y_1\equiv \rho s_2y_2\bmod{k_3s_3},\\
|f_1(k_1y_1)|\leq 1,\\ 
|g(k_1y_1,y_2)|\leq 1
\end{array}
\right\}.
\end{align*}
An application of Exercise \ref{ex:interval} now reveals that 
\begin{equation}
  \label{eq:k1k3}
S_{k_1,k_3}(B)
=
\sum_{\colt{y_2\in\Z: ~|f_2(y_2)|\leq 1}{\hcf(y_2,s_0s_1s_3)=1}}
\Big(
\frac{X_1F_1(X_0,y_2/X_2)}{k_1k_3s_3}+O(1)\Big),
\end{equation}
where 
$$
F_1(u,v):= \int_{\{t\in\R_{\geq 0}: ~|ut^2|,|tv(t-v)|\leq 1\}}\d t.
$$
We close this section by showing that once summed over all
$(s_0,\ma{s},y_2) \in \N^5$, the error term in \eqref{eq:k1k3} makes a
satisfactory overall contribution to the error term in Theorem \ref{t:a1}.
Using the fact that $\sum_{k\mid n}|\mu(k)|=2^{\omega(n)}$, 
we find that this contribution is 
\begin{align*}
\ll \sum_{(s_0,\ma{s})\in\mcal{S}}
\frac{4^{\omega(s_0)}
2^{\omega(s_2)}X_2}{X_0^{1/2}}
&= B^{1/2} \sum_{(s_0,\ma{s})\in\mcal{S}}
\frac{4^{\omega(s_0)}
2^{\omega(s_2)}}{s_0^{1/2}s_2}\\
&\ll B 
\sum_{\colt{s_0,s_1,s_2\in\N}{s_0^3s_1^2s_2^2\leq B}}
\frac{4^{\omega(s_0)}
2^{\omega(s_2)}}{s_0^{2}s_1s_2^2}\ll B \log B.
\end{align*}
This is satisfactory for Theorem \ref{t:a1}, and so we may henceforth
ignore the error term in the above estimate for $S_{k_1,k_3}(B)$.

Define the arithmetic function
$$
\phi^*(n):=\prod_{p\mid n}\Big(1-\frac{1}{p}\Big),
$$
where as is common convention the product is over distinct prime
divisors of $n$.
It will be useful to note that 
\begin{equation}
  \label{eq:phi*ab}
  \phi^*(mn)=\frac{\phi^*(m)\phi^*(n)}{\phi^*(\hcf(m,n))},
\end{equation}
for any $m,n\in\N$.
We must now sum over the variable $y_2$, for which we will employ the
following basic result.

\begin{ex}\label{ex:phi*}
Let $I\subset \R$ be an interval, let $a\in\N$ and let $f:\R\rightarrow
\R_{\geq 0}$ be a function that is continuously differentiable on $I$.
Use \eqref{eq:EM} to show that
$$
\sum_{\colt{n\in\Z\cap I}{\hcf(n,a)=1}}f(n) = \phi^*(a)\int_I f(t)\d t
+O\big(2^{\omega(a)}\sup_{t\in I}|f(t)|\big).
$$
\end{ex}

We may now return to \eqref{eq:k1k3}. Using Exercise \ref{ex:phi*} we
deduce that
\begin{equation}
  \label{eq:S13'-a1}
S_{k_1,k_3}(B)
=
\frac{\phi^*(s_0s_1s_3)X_1X_2F_2(X_0)}{k_1k_3s_3} +O\Big(\frac{2^{\omega(s_0s_1s_3)}X_1}{k_1k_3s_3}\Big),
\end{equation}
where 
$$
F_2(u):= \int_{\{t,v\in\R: ~t>0,~|ut^2|, |uv^2|, |tv(t-v)|\leq 1\}}\d t\d v.
$$
We must now estimate the overall contribution to $\nub$ from the error
term in this estimate, once summed up over the remaining variables.
This  gives 
\begin{align*}
\ll  \sum_{(s_0,\ma{s})\in\mcal{S}}
\frac{4^{\omega(s_0)}
2^{\omega(s_2)}2^{\omega(s_0s_1s_3)}X_1}{s_3}
&\ll B^{1/3}
\sum_{\colt{s_0,s_1,s_2,s_3\in\N}{s_0^3s_1^2s_2^2s_3^2\leq B}}
\frac{8^{\omega(s_0)}
2^{\omega(s_1s_2s_3)}s_2^{1/3}}{s_1^{2/3}s_3^{2/3}}\\
&\ll B \log B,
\end{align*}
by summing over $s_2\leq \sqrt{B/(s_0^3s_1^2s_3^2)}$.  
This is satisfactory for Theorem \ref{t:a1}, and so we may 
henceforth ignore the error term in \eqref{eq:S13'-a1}.
As pointed out to the author by R\'egis de la Bret\`eche, it is easy
to sharpen this error term to $O(B)$ using the fact that $\phi^*$ has
constant average order.

Now it is trivial to check that
$$
\sum_{\colt{d\mid n}{\hcf(d,a)=1}}\frac{\mu(d)}{d}=\frac{\phi^*(n)}{\phi^*(\hcf(a,n))},
$$
for any $a,n\in\N$. 
Bringing together \eqref{eq:S1-a1}, \eqref{eq:S13-a1}
and \eqref{eq:S13'-a1}, we conclude that
$$
T(B)=\sum_{(s_0,\ma{s})\in\mcal{S}}
\sum_{\colt{k_3\mid s_0}{\hcf(k_3,s_1s_2)=1}} 
\hspace{-0.2cm}
\frac{\mu(k_3)}{k_3} 
\frac{\phi^*(s_0s_2)\phi^*(s_0s_1s_3)}{\phi^*(\hcf(k_3s_3,s_0s_2))}
\frac{X_1X_2F_2(X_0)}{s_3},
$$
where $\mcal{S}$ is given by \eqref{eq:calS}.
It is clear that $\hcf(k_3s_3,s_0s_2)=\hcf(k_3s_3,s_0)$.
Let us define the arithmetic function
$$
\vartheta(s_0,\ma{s})=
\frac{\phi^*(s_0s_2)\phi^*(s_0s_1s_3)}{\phi^*(\hcf(s_0,s_3))} \sum_{\colt{k_3\mid s_0}{\hcf(k_3,s_1s_2)=1}} 
\frac{\mu(k_3)}{k_3} \frac{\phi^*(\hcf(k_3,s_0,s_3))}{\phi^*(\hcf(k_3,s_0))}
$$
when $\hcf(s_i,s_j)=1$ for $1\leq i<j\leq 3$, and
$\vartheta(s_0,\ma{s})=0$ otherwise. It follows from 
\eqref{eq:phi*ab} that 
\begin{align*}
\vartheta(s_0,\ma{s})
&=
\frac{\phi^*(s_0s_2)\phi^*(s_0s_1s_3)}{\phi^*(\hcf(s_0,s_3))} 
\prod_{\colt{p\mid \hcf(s_0,s_3)}{p \nmid s_1s_2}}\Big(1-\frac{1}{p}\Big)
\prod_{\colt{p\mid s_0}{p \nmid
    s_1s_2s_3}}\Big(\frac{1-\frac{2}{p}}{1-\frac{1}{p}}\Big)\\
&=
\phi^*(s_0s_2)\phi^*(s_0s_1s_3)
\prod_{\colt{p\mid s_0}{p \nmid s_1s_2s_3}}\Big(\frac{1-\frac{2}{p}}{1-\frac{1}{p}}\Big)\\
&=
\phi^*(s_0)\phi^*(s_1s_2s_3)
\prod_{\colt{p\mid s_0}{p \nmid s_1s_2s_3}}\Big(1-\frac{2}{p}\Big),
\end{align*}
when $\hcf(s_i,s_j)=1$ for $1\leq i<j\leq 3$.  On recalling the
definitions of $X_1,X_2$, we deduce that 
\begin{equation}\label{152.1}
T(B)
=B^{2/3}
\sum_{n\leq B} \D(n)F_2\big((n/B)^{1/3}\big),
\end{equation}
where
\begin{equation}
  \label{eq:DELTA}
  \D(n):=\sum_{n=s_0^3s_1^2s_2^2s_3^2}
\frac{\vartheta(s_0,\ma{s})}{(s_1s_2s_3)^{1/3}},
\end{equation}
for any $n\in \N$.

We will use Perron's formula to estimate $\sum_{n\leq B}\D(n)$, before
combining it with partial summation to estimate \eqref{152.1}. Consider the
Dirichlet series $D(s):=\sum_{n=1}^\infty \D(n)n^{-s}$. 
We have 
$$
D(s+1/3)=\sum_{s_0,s_1,s_2,s_3=1}^\infty 
\frac{\vartheta(s_0,\ma{s})}{s_0^{3s+1}s_1^{2s+1}s_2^{2s+1}s_3^{2s+1}},
$$
and it is  straightforward to check that
$D(s+1/3)=\prod_p a_p(s)$, with 
\begin{align*}
a_p(s)=
1
+\frac{3(1-1/p)}{p^{2s+1}(1-1/p^{2s+1})}
&+
\frac{(1-1/p)(1-2/p)}{p^{3s+1}(1-1/p^{3s+1})}\\
&+\frac{3(1-1/p)^2}{p^{5s+2}(1-1/p^{2s+1})(1-1/p^{3s+1})}.
\end{align*}
Hence $D(s+1/3)=E_1(s)E_2(s)$, where
$E_1(s)=\zeta(2s+1)^3\zeta(3s+1)$ and 
\begin{equation}
  \label{eq:defE2}
  E_2(s)=\frac{D(s+1/3)}{\zeta(2s+1)^3\zeta(3s+1)}=\prod_p\Big(1+O\Big(\frac{1}{p^{4\sigma +2}}\Big)\Big)
\end{equation}
on the half-plane $\Re e(s)> -1/2$. 
In particular, $E_1(s)$ has a meromorphic continuation to all of $\C$
with a pole of order $4$ at $s=0$, and $E_2(s)$ is holomorphic and
bounded on the half-plane $\Re e(s)> -1/4$.

Let $c=1/3+\ve$ for any $\ve>0$, and let $T \geq 1$.  
Then an application of Lemma~\ref{lem:perron} reveals that
$$
\sum_{n\leq B}\D(n)
= \frac{1}{2\pi i}
\int_{c-iT}^{c+iT}E_1(s-1/3)E_2(s-1/3)\frac{B^s}{s}\d s +
O_\ve\Big(\frac{B^{1/3+\ve}}{T}\Big),
$$
provided that $B\not\in \Z$.
We apply Cauchy's residue theorem to the rectangular contour $\mcal{C}$ 
joining the points ${1/6-iT}$, ${1/6+iT}$,
${c+iT}$ and ${c-iT}$.
Now the residue of $E_1(s-1/3)E_2(s-1/3)B^s/s$ at $s=1/3$ is clearly 
$$
\rom{Res}_{s=1/3}\Big\{E_1(s-1/3)E_2(s-1/3)\frac{B^s}{s}\Big\} =
\frac{E_2(0)}{48}B^{1/3}P(\log B),
$$
for some monic polynomial $P\in \R[x]$ of degree $3$. 
Define the difference
$$
\mcal{E}(B)=\sum_{n\leq B}\D(n) -
\frac{E_2(0)}{48}B^{1/3}P(\log B).
$$
Then it follows that
$$
\mcal{E}(B) 
\ll_\ve
\frac{B^{1/3+\ve}}{T}+
\Big(\int_{1/6-iT}^{1/6+iT}+\int_{1/6-iT}^{c-iT}+ \int_{c+iT}^{1/6+iT}\Big)
\Big|E_1(s-1/3)\frac{B^s}{s}\Big|\d s,
$$
since $E_2(s-1/3)$ is holomorphic and bounded on the half-plane $\Re
e(s)\geq 1/6$.

We proceed to estimate the contribution from the horizontal
contours. Recall the well-known convexity bounds
$$
\zeta(\sigma+i t)\ll_\ve 
\left\{
\begin{array}{ll}
|t|^{(1-\sigma)/3+\varepsilon}, & \mbox{if $\sigma\in[1/2,1]$,}\\
|t|^{(3-4\sigma)/6+\varepsilon}, & \mbox{if $\sigma\in[0,1/2]$,}
\end{array}
\right.
$$
for any $|t|\geq 1$.  
A proof of these can be found in \cite[\S II.3.4]{ten}, for example.
It therefore follows that
\begin{equation}\lab{majE1}
E_1(\sigma-1/3+it)\ll_\varepsilon
|t|^{1-3\sigma+\varepsilon}
\end{equation} 
for any $\sigma\in[1/6,1/3)$ and any $|t|\geq 1$. 
We may now deduce that
\begin{equation}\begin{split}
\int_{1/6-iT}^{c-iT}\Big|E_1(s-1/3)\frac{B^s}{s}\Big|\d s &\ll_\ve 
\int_{1/6}^{c}B^\sigma T^{-3\sigma+\ve}\d \sigma\\ 
&\ll_\ve \frac{B^{1/3+\ve}T^\ve}{T} + \frac{B^{1/6}T^\ve}{T^{1/2}}.
\lab{t2}
\end{split}
\end{equation}
One obtains the same estimate for the contribution from the remaining
horizontal contour. Turning to the vertical contour, \eqref{majE1} gives
\begin{align*}
\int_{1/6-iT}^{1/6+iT}\Big|E_1(s-1/3)\frac{B^s}{s}\Big|\d
s 
&\ll  
B^{1/6}\int_{-T}^{T} \frac{|E_1(-1/6+it)|}{1+|t|}  \d t \\
&\ll  
B^{1/6}\int_{-T}^{T} \frac{|t|^{1/2+\ve}}{1+|t|}  \d t\\ 
&\ll B^{1/6}T^{1/2+\ve}.
\end{align*}
Once combined with \eqref{t2}, we conclude that
$$
\mcal{E}(B)\ll_\ve 
B^{1/3+\ve}T^{-1+\ve} + B^{1/6}T^{1/2+\ve},
$$
for any $T \geq 1$. 
Taking $T=B^{1/9}$
we obtain
$$
\sum_{n\leq B}\D(n) =
\frac{E_2(0)}{48}B^{1/3}P(\log B) + O_\ve(B^{2/9+\ve}),
$$
for any $\ve>0$.

We are now ready to complete the proof of Theorem \ref{t:a1}. For this
it suffices to combine the latter estimate with partial summation in
\eqref{152.1}, and then apply Lemma \ref{lem:301.2}.  In this way we
deduce that
\begin{align*}
N_U(B)&=2T(B)+O(B\log B)\\
&=\frac{\sigma_{\infty}E_2(0)}{144}BQ(\log B) + O_\ve(B^{8/9+\ve})+O(B\log B),
\end{align*}
for a further cubic monic polynomial $Q\in \R[x]$. Here
$\sigma_{\infty}=6\int_0^1 F_2(u)\d u$ is given by \eqref{eq:a1-siginf}, and 
it follows from \eqref{eq:defE2} that
$$
E_2(0)=\prod_p\Big(1-\frac{1}{p}\Big)^{4}\Big(1+\frac{4}{p}+\frac{1}{p^2}\Big).
$$
This therefore completes the proof of Theorem \ref{t:a1}.

\section{The $\Dfour$ del Pezzo surface of degree $3$}\label{s:d4_cubic}

In this section we consider Manin's conjecture for the cubic surface
$$
S_2=\{[x_1,x_2,x_3,x_4]\in \bfP^3: 
x_1x_2(x_1+x_2)+x_4(x_1+x_2+x_3)^2=0\},
$$
considered in \eqref{eq:S2}
Let $U_2\subset S_2$ be the open subset formed by deleting
the lines \eqref{eq:d4_lines} from $S_2$. Our task is to estimate 
$N_{U_2}(B)$. 
In doing so it will clearly suffice to establish the estimate for any
surface that is obtained from $S_2$ via a unimodular transformation.
In view of this we will make the change of variables
$$
t_1=x_1, \quad t_2=x_2, \quad t_3=x_1+x_2+x_3,\quad t_4=-x_4,
$$
which brings $S_2$ into the shape
\begin{equation}
  \label{eq:d4}
t_1t_2(t_1+t_2)=t_3^2t_4,
\end{equation}
and which we henceforth denote by $S$.
The $6$ lines on the surface \eqref{eq:d4} take the shape
$$
t_i=t_j=0,\quad t_j=t_1+t_2=0,
$$
where $i$ denotes a generic element of the set $\{1,2\}$, and $j$ 
an element of $\{3,4\}$. 
If $U\subset S$ denotes the open subset formed by deleting these lines from the
surface, then we have $t_3t_4=0$ for any
$[\ma{t}]\not\in U$. It now follows that 
$$
\nub=\frac{1}{2}\#\{\ma{t}\in Z^4: ~\mbox{\eqref{eq:d4}
  holds}, ~|\ma{t}|\leq B, ~t_3t_4\neq 0\}.
$$
As in the argument of Lemma \ref{lem:a1.step1}, 
the factor $\frac{1}{2}$ reflects the fact that $\ma{t}$ and $-\ma{t}$
represent the same point in $\bfP^3$.   
There is a clear symmetry between solutions such that $t_3$ is
positive and negative. Similarly, \eqref{eq:d4} is
invariant under the transformation
$
t_1=-z_1, t_2=-z_2, t_3=z_3$ and $t_4=-z_4.
$
Thus we have
\begin{equation}
  \label{eq:2111.1}
\nub=2\#\{\ma{t}\in Z^4: ~\mbox{\eqref{eq:d4}
  holds}, ~|\ma{t}|\leq B, ~t_3,t_4\geq 1\},
\end{equation}
for any $B\geq 1$.

In the following section we will explicate the relation between
\eqref{eq:2111.1} and counting integral points on the corresponding
universal torsor.  When it comes to the latter task, we will be led to
consider the counting problem for rational points on plane curves of degree
$1$ and $2$. The estimates that we require will need to
be completely uniform in the coefficients of the equations defining the curves.

Given any plane curve $C\subset \bfP^2$ of degree $d\geq 1$, 
that is defined over $\Q$, let
$$
N_C(B):=\#\{x\in C(\Q): H(x)\leq B\}.
$$
As usual we write $Z^n$ for the set of primitive vectors
in $\Z^n$, and $Z_*^n$ for the set of vectors in $Z^n$ with no
components equal to zero. 
We will restrict our attention to curves that are defined by diagonal
ternary forms. We clearly have  
$N_{C}(B)=\frac{1}{2}M_d(\ma{a};B,B,B)$ for certain
non-zero integers $a_1,a_2,a_3$, where
\begin{equation}
  \label{eq:0312.2}
M_d(\ma{a};\ma{B}):=\#\{\x\in Z^3: a_1x_1^d+a_2x_2^d+a_3x_3^d,
~|x_i|\leq B_i\},
\end{equation}
and $\ma{B}=(B_1,B_2,B_3)$.
Let us begin with the situation for projective lines. 
The following result is  due to Heath-Brown  \cite[Lemma $3$]{h-b84}.

\begin{lem}\lab{p:line}
Let $\ma{a} \in Z_*^3$ and let
$B_i >0$.  Then we have
$$
M_1(\ma{a};\ma{B}) \ll 1+\frac{B_1B_2B_3}{\max |a_i|B_i}.
$$
\end{lem}

Lemma \ref{p:line} shows  that there are only
$O(1)$ rational points on lines of sufficiently large height.
If one has a line $L\subset \bfP^2$ given by the equation
$\ma{a}.\x=0$, for $\ma{a}\in Z_*^3$, then the height of $L$ is
simply defined to be $H(L):=|\ma{a}|$. 
When $L$ is an arbitrary line in $\bfP^n$, which is defined over $\Q$,
there is still a very natural way of defining its height. 
The height of $L$ is just the height of the rational point in the
Grassmannian $\mathbb{G}(1,n)$ that corresponds to the line. We will
not need this fact in our work.  
Let $L\subset \bfP^2$ be an arbitrary line defined over $\Q$. Then it
follows from Lemma \ref{p:line} that  
$$
N_L(B)\ll 1 +\frac{B^2}{H(L)}\ll B^2.
$$
This is essentially best possible, as can be seen by taking $n=2$ in
Exercise~\ref{ex:proj}.

Turning to curves of higher degree, we have the following result,
which is a special case of a result due to the author and Heath-Brown
\cite[Corollary 2]{bhb} 

\begin{lem} \lab{p:conic}
Let $\ma{a} \in Z_*^3$ such that $\hcf(a_i,a_j)=1$, and let $B_i>0$.
Then we have  
$$
M_2(\ma{a};\ma{B})
\ll 
\Big(1+\frac{B_1B_2B_3}{|a_1a_2a_3|}\Big)^{1/3} \tau(a_1a_2a_3),
$$
where $\tau(n):=\sum_{d\mid n}1$ denotes the usual divisor function.
\end{lem}

In keeping with our discussion of lines, let us consider to what extent
our estimate reflects the true growth rate of $N_C(B)$, for a
quadratic curve $C\subset \bfP^2$ that is defined
by a diagonal equation with pairwise coprime coefficients. 
Recall the estimate $\tau(n)=O_{\ve}(n^\ve)$, that holds  for any $\ve>0$. Writing $\|C\|$
for the maximum modulus of the coefficients defining $C$, we deduce
from Lemma  \ref{p:conic} that
$$
N_C(B)\ll_{\ve} \|C\|^\ve B.
$$
This should be compared with the work of Heath-Brown \cite[Theorem
3]{annal} that shows 
$N_C(B)\ll_{d,\ve} B^{2/d+\ve}$,
for any irreducible plane curve $C\subset \bfP^2$ of degree $d$.  

Both Lemma \ref{p:line} and Lemma \ref{p:conic} are established using
the geometry of numbers.

\subsection{Elementary considerations}

We proceed to show how $\nub$ can be related to a count of the integer
points on the corresponding universal torsor. Our argument is in
complete analogy to that presented in \S \ref{s.301.3}, although the
individual steps differ somewhat. 
If $\tS$ denotes the minimal desingularisation of the surface $S$,
then Derenthal \cite{der1} has calculated the universal torsor over 
$\tS$, it being embedded in $\A^{10}$ by a single equation
\begin{equation}
  \label{eq:ut-d4}
  s_1u_1y_1^2+  s_2u_2y_2^2+  s_3u_3y_3^2=0.
\end{equation}
Note that one of the variables does not
appear explicitly in the equation.
We will need the following basic fact.

\begin{ex}\label{ex:2111.2}
Let $a,b\in\N$. Show that $a\mid b^2$ if and only if $a=uv^2$ for
$u,v\in\N$ such that $u$ is square-free and $uv\mid b$.
\end{ex}

Given $v\in \R$ and $\ma{s},\ma{u},\ma{y}\in\R^3$, define 
\begin{equation}
  \label{eq:height-d4}
\Psi(v,\ma{s},\ma{u},\ma{y}):=\max\Big\{
\begin{array}{l}
|s_1s_2s_3|, ~|u_1^2u_2^2u_3^2v^3y_1y_2y_3|\\
|s_1u_1^2u_2u_3v^2y_1^2|, ~|s_2u_1u_2^2u_3v^2y_2^2|
\end{array}
\Big\}.
\end{equation}
We are now ready to record our translation of the problem to the
universal torsor.

\begin{lem}\label{lem:2111.5}
We have
$$
\nub=2\#\left\{(v,\ma{s},\ma{u},\ma{y})\in \N^4\times \Z^3\times \N^3: 
\begin{array}{l}
u_3>0, ~\Psi(v,\ma{s},\ma{u},\ma{y})\leq B\\
\mbox{\eqref{eq:ut-d4}  holds},\\ 
|\mu(u_1u_2u_3)|=1,\\
\hcf(s_1s_2s_3,u_1u_2u_3v)=1,\\
\hcf(y_i,y_j)=1,\\
\hcf(y_i,s_j,s_k)=1,
\end{array}
\right\}.
$$
where $i,j,k$ denote distinct elements from the set $\{1,2,3\}$.
\end{lem}

\begin{proof}
Let $\ma{t}\in \Z^4$ be a vector such that
\eqref{eq:d4} holds, with $t_3,t_4\geq 1$. Write 
$$
\eta_{14}=\hcf(t_1,t_4), \quad
\eta_{24}=\hcf(t_2,t_4/\eta_{14}), \quad
\eta_{12}=\hcf(t_1/\eta_{14},t_2/\eta_{24}).
$$
Then $\eta_{12},\eta_{14}, \eta_{24}\in\N$ and
there exists $z_4\in\N$ and $z_1,z_2\in\Z$ such that
$$
t_1=\eta_{12}\eta_{14}z_1, \quad
t_2=\eta_{12}\eta_{24}z_2,\quad
t_4=\eta_{14}\eta_{24}z_4.
$$
Moreover, it is not hard to deduce that
$$
\hcf(\eta_{12}z_1,\eta_{24}z_4)=\hcf(\eta_{12}z_2,z_4)=\hcf(z_1,z_2)=1,
$$
and
$$
\hcf(t_3,\eta_{14},\eta_{12}\eta_{24}z_2)=1.
$$
Under this substitution the equation \eqref{eq:d4} becomes
$$
\eta_{12}^3z_1z_2(\eta_{14}z_1+\eta_{24}z_2)=t_3^2z_4.
$$
It follows that $\eta_{12}^3 \mid t_3^2$
in any given integer solution.  Exercise \ref{ex:2111.2} therefore implies that there exist
$u,v,z_3\in\N$ such that $|\mu(u)|=1$ and 
$$
\eta_{12}=uv^2,\quad t_3=u^2v^3z_3,
$$
with 
$$
z_1z_2(\eta_{14}z_1+\eta_{24}z_2)=uz_3^2z_4.
$$
We proceed to consider the effect of the divisibility condition 
$z_1z_2\mid uz_3^2$ that this equation entails.

Recall that 
$
\hcf(z_1,z_2)=\hcf(z_1,z_4)=\hcf(z_2,z_4)=1.
$
Since $z_1z_2\mid uz_3^2$,  there must exist
$u_1,u_2,u_3,w_1,w_2,w_3\in\Z$ such that $w_1,w_2,w_3,u_3>0$ and
$$
u=u_1u_2u_3,\quad z_1=u_1w_1^2, \quad z_2=u_2w_2^2, \quad z_3=w_1w_2w_3.
$$
Here we have used the fact that if $p$ is a prime such that $p\nmid u$ and $p\mid
z_1z_2$, then $p$ must divide $z_1$ or $z_2$ to even order.
Under these substitutions our equation becomes 
$$
\eta_{14}u_1w_1^2+\eta_{24}u_2w_2^2=u_3w_3^2z_4.
$$
Moreover, we will have the corresponding coprimality conditions
\begin{equation}
  \label{eq:2111.3}
\hcf(u_1u_2u_3 v w_1,\eta_{24}z_4)=\hcf(u_1u_2u_3v w_2,z_4)=\hcf(u_1w_1,u_2w_2)=1,
\end{equation}
and
\begin{equation}
  \label{eq:2111.4}
|\mu(u_1u_2u_3)|=1,\quad \hcf(u_1^2u_2^2u_3^2v^3w_1w_2w_3,\eta_{14},\eta_{24}u_1u_2^2u_3v^2w_2^2)=1.
\end{equation}

We now set $\ma{s}=(\eta_{14},\eta_{24},z_4)$ and $\ma{y}=\ma{w}$, and
replace $(u_1,u_2,u_3)$ by $(-u_1,-u_2,u_3)$. Tracing through our
argument, one sees that we have made the transformation 
$$
\left\{
\begin{array}{l}
t_1=-s_1u_1^2u_2u_3v^2y_1^2,\\
t_2=-s_2u_1u_2^2u_3v^2y_2^2,\\
t_3=u_1^2u_2^2u_3^2v^3y_1y_2y_3,\\
t_4=s_1s_2s_3.
\end{array}
\right.
$$
In particular, it is clear that the height condition $|\x|\leq B$ is
equivalent to $\Psi(v,\ma{s},\ma{u},\ma{y})\leq B$, in the notation of \eqref{eq:height-d4}.
We now observe that under this transformation the equation
\eqref{eq:d4} becomes \eqref{eq:ut-d4}, and the coprimality relations
\eqref{eq:2111.3} and \eqref{eq:2111.4} can be rewritten
$$
\hcf(s_2s_3,u_1u_2u_3v y_1)=\hcf(s_3,u_1u_2u_3v y_2)=\hcf(u_1y_1,u_2y_2)=1,
$$
and
$$
|\mu(u_1u_2u_3)|=1,\quad
\hcf(s_1,u_1u_2u_3v y_2\hcf(y_3,s_2))=1.
$$
We can combine these relations with \eqref{eq:ut-d4} to simplify them
still further. In fact, once combined with \eqref{eq:ut-d4}, we claim that they are equivalent to the conditions
$$
|\mu(u_1u_2u_3)|=1,\quad \hcf(s_1s_2s_3,u_1u_2u_3v)=
\hcf(y_i,y_j)=\hcf(y_i,s_j,s_k)=1,
$$
appearing in the statement of the lemma. 
To establish the forward implication, it suffices to show that
$\hcf(y_1,y_3)=\hcf(y_2,y_3)=1$, the remaining conditions being
immediate.  But these two conditions follow on combining
\eqref{eq:ut-d4} with the fact that $\hcf(y_1,s_2u_2y_2)=1$.
To see the reverse implication, the conditions are all immediate apart
from 
$$
\hcf(y_1,s_2s_3)=\hcf(y_2,s_1s_3)=\hcf(u_1,y_2)=\hcf(u_2,y_1)=1.
$$
But each of these is an easy consequence of the assumed
coprimality relations, and \eqref{eq:ut-d4}. Finally, we leave it as an
exercise to the reader to check that each $(v,\ma{s},\ma{u},\ma{y})$
counted in the right hand side of Lemma \ref{lem:2111.5} produces a
primitive solution of \eqref{eq:d4} with $t_3,t_4\geq 1$.
This completes the proof of Lemma \ref{lem:2111.5}.
\end{proof}

In what follows let us write $i$ for a generic element of the set $\{1,2,3\}$.
Fix a choice of $v\in\N$ and $S_i, U_i, Y_i >0$, and write 
\begin{equation}
  \label{eq:N_dyadic}
\mcal{N}=\mcal{N}_v(\ma{S};\ma{U};\ma{Y})
\end{equation}
for the total contribution to $\nub$ in Lemma \ref{lem:2111.5} from $\ma{s},\ma{u},\ma{y}$ contained in the intervals
\begin{equation}\lab{range1}
S_i/2< s_i \leq S_i, \quad U_i/2< |u_i| \leq U_i, \quad
Y_i/2< y_i \leq Y_i.
\end{equation}
Write
$$
S=S_1S_2S_3,\quad U=U_1U_2U_3,\quad Y=Y_1Y_2Y_3.
$$
If $\mcal{N}=0$ there is nothing to prove, and so we assume henceforth
that the dyadic ranges in \eqref{range1} produce
a non-zero value of $\mcal{N}$. In particular we must have
\begin{equation}
  \label{eq:2111.6}
  S\ll B\quad U^2Y\ll B/v^3, \quad S_iUU_iY_i^2\ll B/v^2.
\end{equation}
In this set of lecture notes we will provide two upper bounds for
$\nub$. The object of our first bound is to merely establish linear
growth, without worrying  about the factor involving $\log B$
that we expect to see. By ignoring some of the technical machinery
needed to get better bounds it is hoped that the overall methodology
will be brought into focus.  Later we will indicate how the
expected upper bound can be retrieved with a little more work.

\subsection{A crude upper bound}

We begin by establishing linear growth for $\nub$. Note
that \eqref{eq:2111.6} forces the inequalities $S_i,U_i,Y_i\ll B$. We
proceed to establish the following upper bound.

\begin{lem}\label{lem:2111.7}
We have 
$$
\nub\ll (\log B)^9 \sum_{v\leq B^{1/3}}
  \max_{S_i,U_i,Y_i>0}
\mcal{N}_v(\ma{S};\ma{U};\ma{Y}),
$$
where the maximum is over $S_i,U_i,Y_i>0$ such that \eqref{eq:2111.6} holds.
\end{lem}

\begin{proof}
Our starting point is Lemma \ref{lem:2111.5}. It follows from 
\eqref{eq:height-d4} that $v\leq B^{1/3}$ for any 
$v,\ma{s},\ma{u},\y$ that contributes to the right hand side.
Let us fix a choice of $v\in\N$ such that $v\leq B^{1/3}$, and cover
the ranges  for $\ma{s},\ma{u},\y$ with dyadic intervals. 
Thus for fixed integers $\sigma_i,\nu_i, \eta_i \geq 0$,
we write 
$$
S_i=2^{\sigma_i}, \quad U_i=2^{\nu_i}, \quad Y_i=2^{\eta_i}, 
$$
and
consider the contribution from $\ma{s},\ma{u},\y$ in the range 
\eqref{range1}. But this is just
$\mcal{N}=\mcal{N}_v(\ma{S};\ma{U};\ma{Y})$. Now we have already seen that
$\mcal{N}=0$ unless  \eqref{eq:2111.6} holds. Finally, since 
each $S_i,U_i,Y_i$ is $O(B)$, it follows that the number of dyadic intervals
needed is $O((\log B)^9)$. This completes the proof of the lemma.
\end{proof}

We may now restrict our attention to bounding
$\mcal{N}_v(\ma{S};\ma{U};\ma{Y})$ for fixed values of $S_i,U_i,Y_i>0$
such that   \eqref{eq:2111.6} holds, and fixed $v\leq B^{1/3}$.
In the arguments that follow it will be necessary to focus attention
on primitive vectors $\ma{s}\in\N^3$. To enable this we draw out
possible common factors between $s_1,s_2,s_3$, obtaining 
\begin{equation}
  \label{eq:0312.1}
\mcal{N}_v(\ma{S};\ma{U};\ma{Y})=\sum_{k=1}^\infty \mcal{N}_v^*(k^{-1}\ma{S};\ma{U};\ma{Y}).
\end{equation}
Here $\mcal{N}_v^*(\ma{S};\ma{U};\ma{Y})$ is defined as for
$\mcal{N}_v(\ma{S};\ma{U};\ma{Y})$ but with the extra condition that $\hcf(s_1,s_2,s_3)=1$.
Let us write $S_i'=S_i/k$ and $\ma{S}'=k^{-1}\ma{S}$.

Recall the equation \eqref{eq:ut-d4} that we must count solutions
to, which it will be convenient to denote by $\mcal{T}$, and which we
will think of as defining a variety in $\bfP^2\times \bfP^2 \times \bfP^2$, with
homogeneous coordinates $\ma{s},\ma{u},\y$. 
The key idea will be to count points on the fibres of projections 
$\pi: \mcal{T}\rightarrow \bfP^2\times\bfP^2$.  This amounts to fixing
six of the variables and estimating the number of points on the
resulting plane curve. Since this family of curves will vary with $B$,
so it is vital to obtain bounds that are completely uniform in the
coefficients of the defining equation. 

Let us begin by fixing the variables $\ma{u},\y$, and estimating the
corresponding number of vectors $\ma{s}$. Now it follows from the coprimality conditions in Lemma
\ref{lem:2111.5} that 
$$
\hcf(u_1y_1^2,u_2y_2^2,u_3y_3^2)=1.
$$
For fixed $\ma{u},\y$, \eqref{eq:ut-d4} defines a line in $\bfP^2$. We clearly
have 
\begin{align*}
\mcal{N}_v^*(\ma{S}';\ma{U};\ma{Y})\leq \sum_{\ma{u},\y} M_1(\ma{a};\ma{S}'),
\end{align*}
in the notation of \eqref{eq:0312.2}, with $a_i=u_iy_i^2$. Since 
$\ma{a}$ is primitive, it therefore follows from
Lemma \ref{p:line} 
that
\begin{align*}
\mcal{N}_v^*(\ma{S}';\ma{U};\ma{Y})
\ll \sum_{\ma{u},\y} 
\Big(1+\frac{S}{k^2\max S_iU_iY_i^2}\Big)
&\ll UY+k^{-2}S^{2/3}U^{2/3}Y^{1/3}.
\end{align*}
Here we have used the trivial lower bound $\max\{a,b,c\}\geq
(abc)^{1/3}$, valid for any $a,b,c>0$.
Using \eqref{eq:2111.6} we conclude that
\begin{equation}
  \label{eq:0312.3}
\mcal{N}_v^*(\ma{S}';\ma{U};\ma{Y})
\ll UY+\frac{B}{k^2 v}.
\end{equation}
The second term here will be satisfactory from our point of view, but
the first is disastrous, since we will run into trouble
when it comes to summing over $k$ in \eqref{eq:0312.1}. 

It turns out that an altogether different bound is required to handle
the contribution from really small values of $\ma{S}'$.
For this we will fix values of $\ma{s},\ma{u}$ in \eqref{eq:ut-d4}, and count points on
the resulting family of conics. 
First we need to record the coprimality relation
$$
\hcf(s_iu_i,s_ju_j)=1,
$$
which we claim holds for any of the vectors $\ma{s},\ma{u},\y$ in which we are interested.
But this follows on noting that $\hcf(s_i,s_j)=1$ for any primitive
vector $\ma{s}\in\Z^3$ such that \eqref{eq:ut-d4} holds and
$\hcf(s_i,u_1u_2u_3)=\hcf(y_i,s_j,s_k)=1$.  We now have 
\begin{align*}
\mcal{N}_v^*(\ma{S}';\ma{U};\ma{Y})\leq \sum_{\ma{s},\ma{u}} M_2(\ma{a};\ma{Y}),
\end{align*}
in the notation of \eqref{eq:0312.2}, with $a_i=s_iu_i$. In
particular $a_i$ is non-zero and $\hcf(a_i,a_j)=1$ in the statement of Lemma
\ref{p:conic}, whence 
\begin{align*}
\mcal{N}_v^*(\ma{S}';\ma{U};\ma{Y})
\ll \sum_{\ma{s},\ma{u}} 
\Big(1+\frac{k^{3/2}Y^{1/2}}{S^{1/2}U^{1/2}}\Big)2^{\omega(s_1s_2s_3u_1u_2u_3)}.
\end{align*}
In view of the
bounds $S,U \ll B$, we clearly have 
$$
2^{\omega(s_1s_2s_3u_1u_2u_3)}\ll_{\ve} (s_1s_2s_3u_1u_2u_3)^\ve \ll_\ve
  (SU)^{\ve} \ll_\ve B^{2\ve}.
$$
Once inserted into our bound for
$\mcal{N}_v^*(\ma{S}';\ma{U};\ma{Y})$, and combined with
\eqref{eq:2111.6},  we deduce that
\begin{equation}
  \label{eq:0312.4}
\begin{split}
\mcal{N}_v^*(\ma{S}';\ma{U};\ma{Y})
&\ll_\ve B^{2\ve}
\Big(\frac{SU}{k^3}+\frac{S^{1/2}U^{1/2}Y^{1/2}}{k^{3/2}}\Big)\\
&\ll_\ve 
\frac{SU B^{2\ve}}{k^3}+\frac{B^{1+2\ve}}{k^{3/2}v^{3/2}}.
\end{split}
\end{equation}
Here the second term will provide a satisfactory contribution, and we
will balance the first term with our earlier estimate \eqref{eq:0312.3}.

Note that
$$
\min\Big\{\frac{SU}{k^{3}}, UY\Big\}\leq \frac{U\sqrt{SY}}{k^{3/2}}\ll \frac{B}{k^{3/2}v^{3/2}},
$$
by \eqref{eq:2111.6}.  It therefore follows from 
\eqref{eq:0312.4} and  \eqref{eq:0312.3} that
$$
\mcal{N}_v^*(\ma{S}';\ma{U};\ma{Y})
\ll_\ve 
\frac{B^{1+2\ve}}{k^{3/2}v}.
$$
Once inserted into \eqref{eq:0312.1}, and then into the statement of
Lemma \ref{lem:2111.7}, we may conclude that
\begin{align*}
\nub
&\ll_\ve (\log B)^9 \sum_{v\leq B^{1/3}}
\sum_{k=1}^\infty \frac{B^{1+2\ve}}{k^{3/2}v}\\
&\ll_\ve B^{1+2\ve}(\log B)^{10}\\
&\ll_\ve B^{1+3\ve}.
\end{align*}
Recall that $N_{U_2}(B)\leq N_{U}(B)$, where $U\subset S$ is the
open subset associated to the surface \eqref{eq:d4}, and $N_{U_2}(B)$ is
the counting function associated to \eqref{eq:S2}. On redefining the
choice of parameter $\ve>0$, we have therefore
established the following result.

\begin{thm}\label{sub-d4}
We have $N_{U_2}(B) \ll_\ve B^{1+\ve},$ for any $\ve>0$.
\end{thm}

The reader will note that there many places in our argument where we
have been wasteful. The most damaging has been in our use of the
trivial bound $2^{\omega(n)}=O_\ve(n^\ve)$, in the deduction of \eqref{eq:0312.4}.
Using the fact that $2^{\omega(n)}$ has average order $\zeta(2)^{-1}\log n$, it is
not particularly difficult to replace the $B^\ve$ in Theorem~\ref{sub-d4} with a
large power of $\log B$.

\begin{ex}\label{ex:expA}
By analysing the proof of Theorem \ref{sub-d4}, find
an explicit value of $A\geq 6$ such that $N_{U_2}(B) \ll B (\log
B)^A.$ 
\end{ex}

In the next section we will be able to show that the value $A=6$
is an admissible exponent, as claimed in Theorem \ref{main-d4}.

\subsection{A better upper bound}

Crucial to the proof of Theorem \ref{sub-d4} was an investigation of
the density of  integer solutions to the equation 
(\ref{eq:ut-d4}). It is in our treatment of this equation that we
will hope to gain some saving. 
Lets put the problem on a more general footing. 
For any $\ma{A},\ma{B},\ma{C}\in\R_{> 1}^3$, let 
$\mcal{M}(\ma{A},\ma{B},\ma{C})$ denote the number of $\ma{a},\ma{b}, \ma{c} \in Z_*^3$
such that 
\begin{equation}\lab{key-x}
a_1b_1c_1^2+a_2b_2c_2^2+a_3b_3c_3^2=0
\end{equation}
and 
$$
\lab{range-x}
|a_i| \leq A_i, \quad |b_i| \leq B_i, \quad |c_i| \leq C_i,
$$
with
\begin{equation}\lab{cat-x}
\hcf(a_i,c_j)=
\hcf(c_i,c_j)=1
\end{equation}
and 
\begin{equation}\lab{cut-x}
|\mu(a_1a_2a_3)|=1, \quad \hcf(a_i,b_j,b_k)=1.
\end{equation}
Here, we recall that $Z_{*}^3$ denotes the set of primitive vectors in
$\Z^3$ with all components non-zero. 
It will be convenient to set
$$
A=A_1A_2A_3, \quad B=B_1B_2B_3, \quad C=C_1C_2C_3.
$$
Arguing exactly as in the previous section, it is not 
difficult to deduce from Lemmas \ref{p:line} and \ref{p:conic}
that
\begin{align*}
\mcal{M}(\ma{A},\ma{B},\ma{C})
&\ll_\ve A\min\{C,
A^{\ve}B^{1+\ve}\}+ A^{2/3}B^{2/3}C^{1/3}+
A^{1/2+\ve}B^{1/2+\ve}C^{1/2}
\end{align*}
for any $\ve>0$, whence
\begin{equation}
  \label{eq:0812.1}
\mcal{M}(\ma{A},\ma{B},\ma{C})
\ll_\ve A^{2/3}B^{2/3}C^{1/3}+
A^{1+\ve}B^{1/2+\ve}C^{1/2}.
\end{equation}
By working a little harder, we would like to replace the terms $A^{\ve}$, $B^{\ve}$
with something rather smaller.

The main problem to be faced emerges in the application of
Lemma~\ref{p:conic}, which gives 
$$
\mcal{M}(\ma{A},\ma{B},\ma{C})
\ll \sum_{\ma{a},\ma{b}} \Big( 1+ \frac{C^{1/2}}{|a_1a_2a_3b_1b_2b_3|^{1/2}}\Big)2^{\omega(a_1a_2a_3b_1b_2b_3)}.
$$
Rather than using the trivial bound $2^{\omega(n)}=O_\ve({n^\ve})$, as
above, we can try to make use of the fact that $2^{\omega(n)}$ has average
order $\zeta(2)^{-1}\log n$ in order to get some saving. Following
this line of thought it is fairly straightforward to show that $A^{\ve}B^{\ve}$ can
be replaced by $(\log A)^{3/2}(\log B)^{3}$ in \eqref{eq:0812.1}. 
However this would still not be enough to deduce the best possible
upper bound for $\nub$ that we would like.
Let us simplify matters by considering only the contribution
$$
\mcal{S}(\ma{A},\ma{B})
= \sum_{\ma{a},\ma{b}} 2^{\omega(a_1a_2a_3b_1b_2b_3)},
$$
to the above estimate for 
$\mcal{M}(\ma{A},\ma{B},\ma{C})$. Then 
$\mcal{S}(\ma{A},\ma{B})$ has exact order of magnitude 
$$
AB\prod_{i=1}^3(\log A_i)(\log B_i),
$$
so how can we hope to do better than this?
The crucial observation comes in noting that we are only interested in
summing over values of $\ma{a},\ma{b}$ for which the corresponding
conic \eqref{key-x} has a non-zero solution $\ma{c}\in\Z^3$, with
$\hcf(c_i,c_j)=1$. If we denote 
this finer quantity by $\mcal{S}^*(\ma{A},\ma{B})$, then it is
actually possible to show that
\begin{equation}
  \label{eq:vs-serre}
\mcal{S}^*(\ma{A},\ma{B})\ll AB.
\end{equation}
This is established in \cite[Lemma 1]{d4}, and is simply a
facet of the well-known fact that a random plane conic doesn't have a rational point. 
This should be compared with the work of Serre \cite{serre}. Using 
the large sieve inequality, Serre has shown 
$$
\#\{\ma{y}\in\Z^3:|\ma{y}|\leq Y, ~(-y_1y_3,-y_2y_3)_\Q=1\} \ll \frac{Y^3}{(\log Y)^{3/2}},
$$
where 
$$
(a,b)_\Q=\left\{
\begin{array}{ll}
1,&\mbox{if $ax^2+by^2=z^2$ has a solution $(x,y,z)\neq \ma{0}$ in $\Q^3$,}\\
-1,&\mbox{otherwise,}
\end{array}
\right.
$$
denotes the Hilbert symbol. Guo \cite{g} has
established an asymptotic formula for the corresponding quantity in
which one counts only odd values of $y_1,y_2,y_3$ such that the
product $y_1y_2y_3$ is square-free.

Thus in addition to considering the density of integer solutions to 
diagonal quadratic equations, as in the previous section, we also
need to consider how often such an equation has at least one
non-trivial integer solution in order to derive sufficiently
sharp bounds. The outcome of this investigation is the following
result, which is established in \cite[Lemma 2]{d4}.

\begin{lem}\lab{M1-x}
For any $\ve>0$, we have
$$
\mcal{M}(\ma{A},\ma{B},\ma{C}) \ll_\ve A^{2/3}B^{2/3}C^{1/3} +
\sigma\tau AB^{1/2}C^{1/2},
$$
where
$$
\sigma=1+\frac{\min\{A, B\}^\ve}{\min\{B_iB_j\}^{1/16}},
\quad 
\tau=1+\frac{\log B}{\min\{B_iB_j\}^{1/16}}.
$$
\end{lem}

It is clear that this constitutes a substantial sharpening over our
earlier estimate \eqref{eq:0812.1} for
$\mcal{M}(\ma{A},\ma{B},\ma{C})$.
Nonetheless this is still not enough on its own, and we will need an
alternative estimate when $B_1, B_2, B_3$ have particularly awkward
sizes.  The following result is rather easy to establish.

\begin{lem}\lab{M2-x}
We have
$$
\mcal{M}(\ma{A},\ma{B},\ma{C}) \ll AB_iB_j(C_k +  C_iC_jA_k^{-1}) (\log AC)^2,
$$
for any permutation $\{i,j,k\}$ of the set $\{1,2,3\}$.  
\end{lem}

\begin{proof}
For fixed integers $a,b, q$,  let $\rho(q;a,b)$ denote the number of solutions to the
congruence 
$
at^2 + b \equiv 0 \bmod{q}.
$
We then have 
\begin{equation}\lab{broken}
\rho(q;a,b)\leq \sum_{d \mid q} |\mu(d)| \Big(\frac{-ab}{d}\Big).
\end{equation}
It will clearly suffice to establish Lemma \ref{M2-x} in the case $(i,j,k)=(1,2,3)$, say.
Now it follows from (\ref{key-x}) that for given $a_i, b_1,b_2,c_3$, and each corresponding
solution $t$ of the congruence
$$
a_1b_1t^2+a_2b_2\equiv 0 \pmod{a_3c_3^2},
$$
we must have
$c_1 \equiv t c_2 \bmod{a_3c_3^2}$ in any solution to be counted.    This gives rise to an equation of the
form $\ma{h}.\ma{w}=0$, with $\ma{h}=(1,-t,a_3c_3^2)$ and
$\ma{w}=(c_1,c_2,k)$.  Upon recalling that
$\hcf(c_1,c_2)=1$ from (\ref{cat-x}), an application of Lemma
\ref{p:line} therefore yields the bound 
$$
\ll \rho(a_3c_3^2;a_1b_2,a_2b_2)\Big(1+ \frac{C_1C_2}{|a_3c_3^2|}\Big),
$$ 
for the number of possible $b_3, c_1,c_2$ given fixed choices of
$a_i,b_1,b_2$ and $c_3$.  It now follows from (\ref{broken}) that
\begin{align*}
\mcal{M}(\ma{A},\ma{B},\ma{C})
&\ll 
\sum_{a_i,b_1,b_2,c_3} \rho(a_3c_3^2;a_1b_2,a_2b_2)\Big(1+ \frac{C_1C_2}{|a_3c_3^2|}\Big)\\ 
&\ll 
\sum_{a_i,b_1,b_2,c_3}\sum_{d
\mid a_3c_3} |\mu(d)| 
\Big(\frac{-a_1a_2b_1b_2}{d}\Big)\Big(1+ \frac{C_1C_2}{|a_3c_3^2|}\Big)\\
&\ll \sum_{a_i,b_1,b_2,c_3} \tau(a_3)\tau(c_3)+ 
C_1C_2\sum_{a_i,b_1,b_2,c_3} \frac{\tau(a_3)\tau(c_3)}{|a_3c_3^2|}.
\end{align*}
A simple application of partial summation now reveals that
\begin{align*}
\mcal{M}(\ma{A},\ma{B},\ma{C})
&\ll \big(AB_1B_2C_3+ A_1A_2B_1B_2C_1C_2\big) (\log AC)^2,
\end{align*}
as required to complete the proof of Lemma \ref{M2-x}.
\end{proof}

We are now ready to combine Lemmas \ref{M1-x} and \ref{M2-x} to get a
sharper upper bound for $N_U(B)$. Taking 
Lemma \ref{lem:2111.5} as our starting point we need to bound the 
quantity $\mcal{N}=\mcal{N}_v(\ma{S};\ma{U};\ma{Y})$ defined in \eqref{eq:N_dyadic}, for fixed choices of 
$v\in\N$ and $S_i, U_i, Y_i >0$. 
As previously we will need to extract common factors from
$s_1,s_2,s_3$, leading to the equality 
\eqref{eq:0312.1}.  Writing $S_i'=S_i/k$ and $\ma{S}'=k^{-1}\ma{S}$,
as before, it is a simple matter to check that we have 
$$
\mcal{N}_v^*(\ma{S}';\ma{U};\ma{Y})\leq 
\mcal{M}(\ma{U},\ma{S}',\ma{Y}),
$$
with  $(\ma{a},\ma{b},\ma{c})=(\ma{u},\ma{s},\ma{y})$.
Indeed we plainly have 
$$
\hcf(u_i,y_j)=
\hcf(y_i,y_j)=1, \quad |\mu(u_1u_2u_3)|=\hcf(u_i,s_j,s_k)=1,
$$
and $\ma{u}, \ma{s}, \ma{y}\in Z_*^4$,
for any vectors counted by $\mcal{N}_v^*(\ma{S}';\ma{U};\ma{Y})$, as
required for $\mcal{M}(\ma{U},\ma{S}',\ma{Y})$.
It now follows from \eqref{eq:0312.1} and Lemma \ref{M1-x} that
\begin{align*}
\mcal{N}_v(\ma{S};\ma{U};\ma{Y})
&\ll_\ve \sum_{k=1}^\infty  \Big(
\frac{U^{2/3}S^{2/3}Y^{1/3}}{k^2} +
k^{2/16}\sigma\tau \frac{US^{1/2}Y^{1/2}}{k^{3/2}}\Big)\\
&\ll_\ve U^{2/3}S^{2/3}Y^{1/3}+ \sigma\tau US^{1/2}Y^{1/2},
\end{align*}
for any $\ve>0$, where
$$
\sigma=1+\frac{\min\{S, U\}^\ve}{\min\{S_iS_j\}^{1/16}},
\quad 
\tau=1+\frac{\log B}{\min\{S_iS_j\}^{1/16}}.
$$
In order to obtain our final estimate for $\nub$ we need to sum this
bound over all positive integers $v\leq B^{1/3}$, as in Lemma
\ref{lem:2111.7}, and over all possible dyadic intervals 
for $S_i,U_i,Y_i$, subject to 
(\ref{eq:2111.6}).

Suppose for the moment that we want to sum over all
possible dyadic intervals $X \leq |x|<2X$, for which $|x| \leq
\mcal{X}$.  Then in deducing Lemma \ref{lem:2111.7} we employed the
basic bound
$O(\log \mcal{X})$ for the number of possible choices
for $X$.  In the present investigation we will be more
efficient and take advantage of the easily established estimates 
$$
\sum_{X} X^{\delta} \ll_\delta 
\left\{
\begin{array}{ll}
1, & \mbox{if $\delta<0$},\\
\mcal{X}^{\delta}, & \mbox{if $\delta>0$},
\end{array}
\right.
$$
where the sum is over dyadic intervals for $X\leq \mcal{X}$. We will
make frequent use of these bounds without further mention.

Returning to our estimate for $\mcal{N}_v(\ma{S};\ma{U};\ma{Y})$, we
may now conclude from the bound $Y_i\leq B^{1/2}/(v^2S_iUU_i)^{1/2}$
in \eqref{eq:2111.6} that
\begin{align*}
\nub
&\ll_\ve \sum_{v\leq B^{1/3}} \sum_{S_i,U_i,Y_i}\Big(
U^{2/3}S^{2/3}Y^{1/3}+ \sigma\tau US^{1/2} Y^{1/2}\Big)\\
&\ll_\ve B^{1/2}\sum_{v\leq B^{1/3}} \sum_{S_i,U_i}
\frac{S^{1/2}}{v}+ 
\sum_{v\leq B^{1/3}} \sum_{S_i,U_i,Y_i }\sigma\tau US^{1/2} Y^{1/2}\\
&\ll_\ve B(\log B)^6+
\sum_{v\leq B^{1/3}} \sum_{S_i,U_i,Y_i }\sigma\tau US^{1/2} Y^{1/2}.
\end{align*}
The first term on the right-hand side is clearly satisfactory, and it
remains to deal with the second term, which we denote by $R$ for
convenience. We would like to show that $R\ll B(\log B)^6$.

Suppose without loss of generality that $S_1\leq S_2\leq
S_3$, so that in particular $\min\{S_iS_j\}=S_1S_2$ in $\sigma$ and $\tau.$  If
there is a constant $A>0$ such that 
$
S_3\leq (S_1S_2)^A,
$
then it follows that
$$
\sigma\leq (S_1S_2)^{\ve-1/16}S_3^\ve\leq (S_1S_2)^{(1+A)\ve-1/16}\ll 1,
$$
provided that $\ve$ is sufficiently small.  Taking $\tau \ll \log B$,
we may then argue as above to conclude that there is a contribution of
$O(B(\log B)^6)$ to $R$ from this case.
Suppose now that there exists $A'>0$ such that $U\leq
(S_1S_2)^{A'}$. Then we have $\sigma\ll 1$ and $\tau\ll \log B$, so
that there is a contribution of $O(B(\log
B)^6)$ to $R$ in this case too.

Finally it remains to consider the contribution to $\nub$ from 
$S_i,U_i,Y_i$ such that
\begin{equation}
  \label{eq:friday}
S_1S_2\leq \min \{S_3,U\}^\delta, 
\end{equation}
for some small value of $\delta>0$, with $S_1\leq S_2\leq S_3$. 
Let us denote this contribution $N_0$, say. 
To estimate $N_0$ we will return to the task of estimating 
$\mcal{N}_v(\ma{S};\ma{U};\ma{Y})$ for fixed $v, S_i,U_i,Y_i$, but
this time apply Lemma \ref{M2-x} with $(i,j,k)=(1,2,3)$. This gives
\begin{align*}
\mcal{N}_v(\ma{S};\ma{U};\ma{Y})
&\ll (\log B)^2 \big(
US_1S_2Y_3 +  U_1U_2S_1S_2Y_1Y_2\big).
\end{align*}
We must now sum over dyadic intervals for $S_i,U_i,Y_i$.
Thus it follows from the bound $Y_i\leq B^{1/2}/(v^2S_iUU_i)^{1/2}$
in \eqref{eq:2111.6} that
\begin{align*}
N_0
&\ll (\log B)^2\sum_{v\leq B^{1/3}} \sum_{S_i,U_i,Y_i}
\big(US_1S_2Y_3 +  U_1U_2S_1S_2Y_1Y_2\big)\\
&\ll (\log B)^2\sum_{v\leq B^{1/3}}\Big(\sum_{\colt{S_i,U_i}{Y_1,Y_2}}
\frac{B^{1/2}U^{1/2}S_1S_2}{vS_3^{1/2}U_3^{1/2}} +
\sum_{S_i,U_i,Y_3} \frac{B (S_1S_2U_1U_2)^{1/2}}{v^2U}\Big).
\end{align*}
Since $U^2\ll B/(v^3Y_1Y_2)$ in  \eqref{eq:2111.6}, and $S_1S_2\ll
S_3^{\delta}$ by \eqref{eq:friday}, we therefore deduce that the
overall contribution from the first inner sum is
\begin{align*}
&\ll (\log B)^2\sum_{v\leq B^{1/3}}\sum_{\colt{S_i,Y_1,Y_2}{U_2,U_3}}
\frac{B^{3/4}S_1S_2}{v^{7/4}S_3^{1/2}U_3^{1/2}Y_1^{1/4}Y_2^{1/4}} \\
&\ll \sum_{\colt{S_2,S_3,Y_1,Y_2}{U_2,U_3}}
\frac{B^{3/4}(\log
  B)^2}{S_3^{1/2-\delta}U_3^{1/2}Y_1^{1/4}Y_2^{1/4}} 
\ll B.
\end{align*}
Turning to the contribution from the second inner sum, we deduce
from a second application of \eqref{eq:friday} that
\begin{align*}
N_0
&\ll B
+(\log B)^2\sum_{S_i,U_i,Y_3} \frac{B (S_1S_2U_1U_2)^{1/2}}{U}\\
&\ll B
+(\log B)^2\sum_{S_2,S_3,U_i,Y_3} \frac{B}{U^{(1-\delta)/2}}
\ll B(\log B)^5.
\end{align*}
Once combined with our earlier work, this 
therefore concludes the proof of Theorem \ref{main-d4}.

\begin{ex}
By mimicking the argument in \cite[\S 5]{d4},
establish the lower bound $N_{U_2}(B) \gg B (\log B)^6$.
\end{ex}

\begin{ack}
The author is grateful to Roger Heath-Brown for a number of useful
conversations relating to the contents of \S \ref{s:HL}, and to Michael
Harvey for spotting several typographical errors in an earlier draft.
\end{ack}


\begin{thebibliography}{99}

\bibitem{b-m} V.V. Batyrev and Y.I. Manin,
Sur le nombre des points rationnels de hauteur born\'e des
vari\'et\'es alg\'ebriques. {\em  Math. Ann.}  {\bf
286} (1990), 27--43.


\bibitem{b-t} V.V. Batyrev and Y. Tschinkel,
Tamagawa numbers of polarized algebraic varieties. {\em  Ast\'erisque}  {\bf
251} (1998), 299--340.


\bibitem{MR1620682} V.V. Batyrev and Y. Tschinkel, Manin's conjecture for
   toric varieties.
\emph{J. Alg. Geom.} {\bf 7} (1998),  15--53.

\bibitem{birch-61}
B.J. Birch,
Forms in many variables.
{\em Proc. Roy. Soc. Ser. A} {\bf 265} (1961), 245--263.

\bibitem{MR2000b:11074} R. de la Bret\`eche,
Sur le nombre de points de hauteur born\'ee d'une certaine
surface cubique singuli\`ere. {\em Ast\'erisque} {\bf 251} 
(1998), 51--77.

\bibitem{dlB-toric} R. de la Bret\`eche,  Compter des points d'une
  vari\'et\'e torique.  
{\em   J. Number Theory}  {\bf 87}  (2001),  315--331.

\bibitem{MR2003m:14033} R. de la Bret\`eche,  Nombre de points de hauteur
born\'ee sur les surfaces de
del Pezzo de degr\'e $5$.  {\em  Duke Math. J.} {\bf 113} (2002),
421--464.

\bibitem{dp4-d5} R. de la Bret\`eche and T.D. Browning,
On Manin's conjecture for singular del Pezzo surfaces of degree four,
{I}. {\em Michigan Math. J.}, to appear.

\bibitem{dp4-d4} R. de la Bret\`eche and T.D. Browning,
On Manin's conjecture for singular del Pezzo surfaces of degree four, 
{II}. {\em Math. Proc. Camb. Phil. Soc.}, to appear.

\bibitem{MR2099200} 
R. de la Bret{\`e}che  and {\'E.} Fouvry,
{L'\'eclat\'e du plan projectif en quatre points dont deux
conjugu\'es}. {\em J. Reine Angew. Math.} {\bf 576} (2004),  {63--122}.

\bibitem{b-swd} R. de la Bret\`eche and P. Swinnerton-Dyer,
Fonction z\^eta des hauteurs associ\'ee \`a une certaine surface cubique.
{\em Submitted}, 2006.

\bibitem{e6} R. de la Bret\`eche, T.D. Browning and U. Derenthal,
On Manin's conjecture for a certain singular cubic surface. {\em
  Ann. Sci. \'Ecole Norm. Sup.}, to appear.

\bibitem{d4}  T.D. Browning, {The 
density of rational points on a certain singular cubic
surface}. {\em J. Number Theory} {\bf 119},  (2006), 242--283.

\bibitem{gauss} T.D. Browning,
An overview of Manin's conjecture for del Pezzo surfaces.
{\em Proceedings of the Gauss--Dirichlet conference (G\"ottingen)}, to appear.

\bibitem{bhb} T.D. Browning and D.R. Heath-Brown, Counting rational
points on hypersurfaces.
{\em J. Reine Angew. Math.} {\bf 584} (2005), 83--115.

\bibitem{b-w}
{J.W. Bruce   and C.T.C. Wall},
{On the classification of cubic surfaces}.
{\em J. London Math. Soc.} {\bf 19} (1979), {245--256}.

\bibitem{cassels-book}
J.W.S. Cassels,
{\em Introduction to the geometry of numbers}. 2nd ed.,
Springer-Verlag, 1997.

\bibitem{cayley}
A. Cayley, A memoir on cubic surfaces. {\em Phil. Trans. Roy. Soc.}
{\bf 159} (1869), 231--326.

\bibitem{ct}
A. Chambert-Loir and Y. Tschinkel,
On the distribution of points of bounded height 
on equivariant compactifications of vector groups.
{\em  Invent. Math.} {\bf 148} (2002), 421--452.


\bibitem{ct1}J.-L. Colliot-Th\'{e}l\`{e}ne and J.-J. Sansuc,
Torseurs sous des groupes de type multiplicatif; applications \`{a} 
l'\'{e}tude des points rationnels de certaines vari\'{e}t\'{e}s 
alg\'{e}briques.  {\em C. R. Acad. Sci. Paris S\'{e}r. A-B}  {\bf 282}  (1976), 
1113--1116.

\bibitem{ct2}J.-L. Colliot-Th\'{e}l\`{e}ne and J.-J. Sansuc, 
La descente sur les vari\'{e}t\'{e}s rationnelles. II.  {\em 
Duke Math. J.}  {\bf 54}  (1987), 375--492.


\bibitem{ctks}
J.-L. Colliot-Th\'el\`ene, D. Kanevsky, et J.-J. Sansuc, 
{\em Arithm\'etique
  des surfaces cubiques diagonales}, Diophantine approximation and transcendenc
  theory (Bonn, 1985), Lecture Notes in Math. {\bf 1290}, Springer-Verlag,
1987, pp. 1--108.


\bibitem{c-t}
D.F. Coray and M.A. Tsfasman,
Arithmetic on singular Del Pezzo surfaces.
{\em Proc. London Math. Soc.} {\bf 57} (1988), 25--87.


\bibitem{dav16}
H. Davenport, Cubic forms in 16 variables. {\em Proc. Roy. Soc.
A} {\bf 272} (1963), 285--303.


\bibitem {dav} 
H. Davenport,
{\em Analytic Methods in Diophantine Equations and Diophantine
Inequalities}. 2nd ed., edited by T.D. Browning,
Cambridge University Press, 2005.

\bibitem{der1}
U. Derenthal,
Singular del Pezzo surfaces whose universal torsors are
hypersurfaces. 
{\em Submitted}, 2006.


\bibitem{der+tsch}
U. Derenthal and Y. Tschinkel,
Universal torsors over Del Pezzo surfaces and rational points, {\em
  Equidistribution in Number theory, An Introduction}, 169--196, NATO Science Series II {\bf
  237}, Springer, 2007.


\bibitem{faltings}
G. Faltings, Endlichkeitss\"{a}tze f\"{u}r abelsche 
Variet\"{a}ten \"{u}ber Zahlk\"{o}rpern. {\em Invent. Math.} 
{\bf 73} (1983), 349--366.


\bibitem{MR2000b:11075} \'E. Fouvry,
Sur la hauteur des points d'une certaine surface cubique singuli\`ere.
{\em  Ast\'erisque}  {\bf  251} (1998), 31--49.

\bibitem{granville} A. Granville,
On the number of solutions to the generalized Fermat equation. 
{\em Number theory (Halifax, NS, 1994)}, 197--207, CMS Conf. Proc.
{\bf 15}, Amer. Math. Soc., Providence, RI, 1995.

\bibitem{f-m-t}
J. Franke, Y.I. Manin and Y. Tschinkel,
Rational points of bounded height on {F}ano varieties.
{\em Invent. Math.} {\bf 95} (1989), 421--435.

\bibitem{g}
C.R. Guo,
On solvability of ternary quadratic forms.
{\em Proc. London Math. Soc.} {\bf 70} (1995), 241--263.


\bibitem{hart}
R. Hartshorne, 
{\em Algebraic geometry}. Springer-Verlag, 1977.


\bibitem{MR2029868}
B. Hassett and Y. Tschinkel,
Universal torsors and Cox rings.
{\em Arithmetic of higher-dimensional algebraic varieties (Palo Alto,
CA, 2002)}, 149--173,
Progr. Math. {\bf 226}, Birkh\"auser, 2004.



\bibitem{h-b84} D.R. Heath-Brown,
Diophantine approximation with square-free numbers.
{\em Math. Zeit.} {\bf 187} (1984), {335--344}.

\bibitem{hb:density}
D.R. Heath-Brown, The density of zeros of forms for which weak
  approximation fails. {\em Math. Comp.} {\bf 59} (1992), 613--623.

\bibitem{hb-crelle}
D.R. Heath-Brown, 
A new form of the circle method, and its application to quadratic
forms.  
{\em J. Reine Angew. Math.}  {\bf 481}  (1996), 149--206.


\bibitem{MR98h:11083} D.R. Heath-Brown,
{The density of rational points on cubic surfaces}.
{\em Acta Arith.} {\bf 79} (1997), 17--30.




\bibitem{hb-3/2} D.R. Heath-Brown,
{The circle method and diagonal cubic forms}.
{\em R. Soc. Lond. Philos. Trans. Ser. A} {\bf
  356}  (1998), 673--699.


\bibitem{hb-ast} D.R. Heath-Brown,
{Counting rational points on cubic surfaces}.
{\em Ast\'erisque} {\bf 251} (1998), 13--29.


\bibitem{annal} D.R. Heath-Brown,
The density of rational points on curves and surfaces. {\em
Annals of Math.} {\bf 155} (2002), 553--595.

\bibitem{hb-cayley} D.R. Heath-Brown, The density of rational points on
Cayley's cubic surface.
{\em Proceedings of the session in analytic number theory and
Diophantine equations}, Bonner
Math. Schriften {\bf 360},  2003.

\bibitem{MR2000f:11080} D.R. Heath-Brown and B.Z. Moroz,  
The density of rational points on the cubic surface $X_0^3=X_1X_2X_3$.
{\em Math. Proc.  Camb. Soc.} {\bf 125} (1999), 385--395.

\bibitem{hind}
M. Hindry and J. Silverman,
{\em Diophantine geometry.} Springer-Verlag, 2000.


\bibitem{h-p} W.V.D. Hodge and D. Pedoe, \emph{Methods of algebraic
    geometry}. Vol. 2, Cambridge University Press, 1952. 

\bibitem{ir} K. Ireland and M. Rosen, {\em A classical introduction to
  modern number theory}. 2nd ed., Springer-Verlag, 1990.


\bibitem{knorrer} H. Kn\"orrer, \emph{Isolierte singularit\"aten von
    durchschnitten zweier quadriken.} Bonner Mathematische Schriften
    \textbf{117}. Universit\"at Bonn Math. Institut, Bonn, 1980.



\bibitem{lipman}
J. Lipman, 
Rational singularities, with applications to algebraic surfaces and unique factorization.
{\em Inst. Hautes \'Etudes Sci. Publ. Math.} {\bf 36} (1969), 195--279.


\bibitem{manin-book} Y.I. Manin, {\em Cubic forms}. 2nd ed.,
    North-Holland Mathematical Library {\bf 4}, North-Holland Publishing Co.,  1986.


\bibitem{mordell}
L.J. Mordell, A remark on indeterminate equations in several
variables.
{\em J. London Math. Soc.} {\bf 12} (1937), 127--129.


\bibitem{MR1340296}
E. Peyre,
Hauteurs et nombres de Tamagawa sur les vari\'et\'es de Fano.
{\em Duke Math. J.} {\bf 79} (1995), 101--218.


\bibitem{MR2029862}
E. Peyre,
Counting points on varieties using universal torsors.
{\em Arithmetic of higher-dimensional algebraic 
varieties (Palo Alto, CA, 2002)}, 61--81,
Progr. Math. {\bf 226}, Birkh\"auser, 2004.



\bibitem{p-t1}
E. Peyre and Y. Tschinkel, 
Tamagawa numbers of diagonal cubic surfaces, numerical evidence.  
{\em Math. Comp.} {\bf 70} (2001), 367--387.

\bibitem{p-t2}
E. Peyre and Y. Tschinkel, 
Tamagawa numbers of diagonal cubic surfaces of higher rank. 
{\em Rational points on algebraic varieties}, 275--305,
Progr. Math., {\bf 199}, Birkh\"auser, 2001. 


\bibitem{MR1679841}
P. Salberger, Tamagawa measures on universal torsors and points of bounded height on
Fano varieties.
{\em Ast\'erisque}  {\bf 251} (1998), 91--258.

\bibitem{cayley'}
L. Schl\"afli, On the distribution of surfaces of the third order into
species. {\em Phil. Trans. Roy. Soc.}
{\bf 153} (1864), 193--247.

\bibitem{segre}
B. Segre, A note on arithmetical properties of cubic surfaces.
{\em  J. London Math. Soc.} {\bf 18} (1943), 24--31.


\bibitem{serre}
J.- P. Serre, Sp\'ecialisation des \'el\'ements de
  $\rom{Br}_2(\Q(T_1,\ldots,T_n))$. {\em C. R. Acad. Sci. Paris}
{\bf 311} (1990), 397--402.

\bibitem{s-swd}
J.B. Slater and P. Swinnerton-Dyer, 
Counting points on cubic surfaces. {I}. 
 {\em Ast\'erisque} {\bf 251} 
(1998), 1--12.


\bibitem{ten}
G. Tenenbaum, {\em Introduction to analytic and probabilistic number
   theory}. Translated from the 2nd French ed.,
Cambridge Studies in Advanced Mathematics {\bf 46}, Cambridge
University Press, 1995.

\bibitem{titch}
E.C. Titchmarsh, {\em The theory of the Riemann zeta-function}. 2nd
ed., edited by D.R. Heath-Brown,
Oxford University Press, 1986.

\bibitem{wall}
C.T.C. Wall, The first canonical stratum.  {\em J. London Math. Soc.}
{\bf 21}  (1980), 419--433.


\end{thebibliography}
\end{document}